\documentclass[12pt]{amsart}
\pdfoutput=1

\usepackage[utf8]{inputenc}
\usepackage{indentfirst}
\usepackage[backend=bibtex, style=alphabetic, citestyle=alphabetic]{biblatex}
\usepackage[bookmarks=true, colorlinks]{hyperref}

\usepackage{amsmath}
\usepackage{amsthm}
\usepackage{amssymb}
\usepackage{amsfonts}
\usepackage{microtype}
\usepackage{fancyhdr}
\usepackage{tikz-cd}
\usetikzlibrary{arrows}
\usetikzlibrary{positioning}
\usepackage{xcolor}
\usepackage{mathtools}
\usepackage{thmtools}

\usepackage[
  margin=1in,
  includefoot,
  footskip=30pt,
]{geometry}

\usepackage{import}
\usepackage{xifthen}
\usepackage{pdfpages}
\usepackage{transparent}
\usepackage{csquotes}

\hypersetup{
  colorlinks = true,
  citecolor = {purple!70!black},
  linkcolor = {red!60!black}
}

\newcommand{\A}{\mathbb{A}}

\renewcommand{\k}{\Bbbk}
\renewcommand{\O}{\mathcal{O}}
\renewcommand{\P}{\mathbb{P}}

\newcommand{\Z}{\mathbb{Z}}

\newcommand{\Spec}{\mathrm{Spec}\,}
\newcommand{\Ext}{\mathrm{Ext}}

\newcommand{\Hom}{\mathrm{Hom}}
\newcommand{\End}{\mathrm{End}}
\newcommand{\RHom}{\mathrm{RHom}}
\newcommand{\REnd}{\mathrm{REnd}}

\newcommand{\RGamma}{\mathrm{R}\Gamma\,}
\newcommand{\Pic}{\mathrm{Pic}\,}

\newcommand{\Coh}{\mathrm{Coh}}
\newcommand{\Dbcoh}{D^b_{\!\mathrm{coh}}}
\newcommand{\Dperf}{\mathrm{Perf}}

\newcommand{\Tot}{\mathrm{Tot}}

\newcommand{\PGL}{\mathrm{PGL}}

\newcommand{\mA}{\mathcal{A}}
\newcommand{\mB}{\mathcal{B}}
\newcommand{\mH}{\mathcal{H}}

\newcommand{\dual}{{\scriptstyle\vee}}
\newcommand{\iso}{\simeq}
\newcommand{\caniso}{\cong}
\newcommand{\isoarrow}{\xrightarrow{\sim}}
\newcommand{\monoarrow}{\hookrightarrow}
\newcommand{\epiarrow}{\twoheadrightarrow}

\newcommand{\supp}{\operatorname{supp}}

\newcommand{\Stab}{\operatorname{Stab}}
\newcommand{\Fit}{\operatorname{Fit}}

\newcommand{\mylength}{\ell}

\declaretheoremstyle[
headformat=\NUMBER.\,\NAME\NOTE,
postheadspace=.5em,
spaceabove=6pt,
headfont=\normalfont\small\scshape,
notefont=\normalfont\small\mdseries, notebraces={(}{)},
bodyfont=\normalfont\itshape
]{plainswap}
\declaretheoremstyle[
headformat=\NAME\NOTE,
postheadspace=.5em,
spaceabove=6pt,
headfont=\normalfont\small\scshape,
notefont=\normalfont\small\mdseries, notebraces={(}{)},
bodyfont=\normalfont\itshape
]{nonumplainswap}
\declaretheoremstyle[
headformat=\NUMBER.\,\NAME\NOTE,
postheadspace=.5em,
spaceabove=6pt,
headfont=\normalfont\small\scshape,
notefont=\normalfont\mdseries, notebraces={(}{)},
bodyfont=\normalfont
]{definitionswap}
\declaretheoremstyle[
headformat=\NAME\NOTE,
postheadspace=.5em,
spaceabove=6pt,
headfont=\normalfont\itshape,
notefont=\mdseries, notebraces={(}{)},
bodyfont=\normalfont
]{myremark}

\declaretheorem[style=plainswap, name=Theorem, sharenumber=subsection]{theorem}

\declaretheorem[style=plainswap, numberlike=theorem, name=Proposition]{proposition}
\declaretheorem[style=plainswap, numberlike=theorem, name=Lemma]{lemma}
\declaretheorem[style=plainswap, numberlike=theorem, name=Corollary]{corollary}
\declaretheorem[style=plainswap, numberlike=theorem, name=Conjecture]{conjecture}

\theoremstyle{definition}
\declaretheorem[style=definitionswap, numberlike=theorem, name=Definition]{definition}

\declaretheorem[style=definitionswap, numberlike=theorem, name=Setting]{setting}
\declaretheorem[style=definitionswap, numberlike=theorem, name=Example]{numberedexample}
\theoremstyle{myremark}
\newtheorem*{remark}{Remark}

\theoremstyle{remark}

\numberwithin{equation}{theorem}
\numberwithin{figure}{theorem}

\bibliography{references}
\AtEveryBibitem{\clearfield{issn}}
\AtEveryBibitem{\clearlist{issn}}

\AtEveryBibitem{\clearfield{language}}
\AtEveryBibitem{\clearlist{language}}

\AtEveryBibitem{\clearfield{doi}}
\AtEveryBibitem{\clearlist{doi}}

\AtEveryBibitem{%
  \ifentrytype{online}
    {}
    {\clearfield{urlyear}\clearfield{urlmonth}\clearfield{urlday}}}

\begin{document}

\title{Admissible subcategories of del Pezzo surfaces}
\author{Dmitrii Pirozhkov}
\address{Department of Mathematics, Columbia University, New York, New York, USA}
\email{dpirozhkov@math.columbia.edu}

\begin{abstract}
  We study admissible subcategories of derived categories of coherent sheaves on del Pezzo surfaces and rational elliptic surfaces. Using a relation between admissible subcategories and anticanonical divisors we prove the following results. First, we classify all admissible subcategories of the projective plane by showing that each is generated by a subcollection of a full exceptional collection. Second, we show that the derived categories of del Pezzo surfaces do not contain any phantom subcategories. This provides first examples of varieties of dimension larger than one that have some nontrivial admissible subcategories, but provably do not contain phantoms. We also prove that any admissible subcategory supported set-theoretically on a smooth $(-1)$-curve in a surface is generated by some twist of the structure sheaf of that curve.
\end{abstract}

\maketitle

\setcounter{tocdepth}{1}
\tableofcontents


\section{Introduction}

The derived category of coherent sheaves on an algebraic variety is a large and complicated invariant. It contains a lot of information about the variety, and many other invariants may be extracted out of the derived category. Working with this huge invariant directly is difficult, and thus an important notion in this field is the notion of a semiorthogonal decomposition. This is a particular way of decomposing the derived category into smaller pieces. Those pieces are called admissible subcategories.

We know many examples of semiorthogonal decompositions. For instance, a full exceptional collection is nothing but a semiorthogonal decomposition of a category with components equivalent to the derived category of vector spaces. The first example of such a decomposition was given in \cite{beil78} for projective spaces. Full exceptional collections are also known for Grassmannians, del Pezzo surfaces, and other varieties. The set of all semiorthogonal decompositions of a given category has a complicated structure; in particular, it carries the action of the braid group which acts by mutations, introduced in \cite{gorodentsev89,bond-kapr}. There are other tools of various complexity to produce new semiorthogonal decompositions, see, for instance, the survey \cite{kuznetsov-icm}.

Despite a large number of examples, we do not have good structural results for an arbitrary semiorthogonal decomposition. Most of the things we know are ``negative'': for instance, we know that the Jordan--H\"older property does not hold for semiorthogonal decompositions, with counterexamples given in \cite{jordan-holder,kuznetsov-jordanholder}. Another somewhat pathological behavior is the existence of so-called \emph{phantom subcategories}, shown in \cite{gorchinsky-orlov,BGKS-phantoms}, which are admissible subcategories which behave as zero subcategories on the level of $K$-theory.

Among the positive constraints on the structure of admissible subcategories, perhaps the most immediately useful result is proved in \cite{kawatani-okawa}: admissible subcategories are closed under small deformations of objects. The interaction of semiorthogonal decompositions with Hochschild (co)homology has been studied in \cite{kuznetsov-oldhochschild}. Some tools useful for the study of semiorthogonal decompositions are developed in the papers proving the non-existence of nontrivial semiorthogonal decompositions for some classes of varieties \cite{bridgeland,okawa,kawatani-okawa}.

It is generally expected that for sufficiently nice varieties, e.g., for projective spaces, there are more restrictions on the structure of admissible subcategories. For example, it is conjectured in \cite[Rem.~1.7]{kuznetsov-polishchuk} that there are no phantom subcategories in homogeneous spaces. It is surprisingly hard to check these expectations for any variety which is more complicated than the projective line $\P^1$.

In this paper we study admissible subcategories for some surfaces with many anticanonical divisors. Namely, we study the projective plane, rational elliptic surfaces, and del Pezzo surfaces. The strongest theorem is obtained in the case of projective plane. On $\P^2$ we are able to produce a full classification of admissible subcategories. As we prove in Section~\ref{sec: projective plane}, they all turn out to be generated by exceptional collections:

\begin{restatable*}{theorem}{statemaintheoremPlane}
  \label{thm: all admissible subcategories are standard}
  Any admissible subcategory in $\Dbcoh(\P^2)$ is generated by a subcollection of a mutation of the standard exceptional collection $\Dbcoh(\P^2) = \langle \O, \O(1), \O(2) \rangle$.
\end{restatable*}

This theorem gives a classification of admissible subcategories since mutations of the standard exceptional collection in the derived category $\Dbcoh(\P^2)$ have been classified by Gorodentsev and Rudakov in \cite{gorodentsev-rudakov}. In fact, we rely on this classification to prove the theorem above.

The case of $\P^2$ is quite special. In order to study admissible subcategories in other surfaces, we start with admissible subcategories supported set-theoretically on smooth $(-1)$-curves. By the support of an admissible subcategory we just mean the union of set-theoretic supports of all objects in the subcategory. In Section~\ref{sec: minus one curves} we prove a classification for admissible subcategories like that:

\begin{restatable*}{proposition}{statemaintheoremMinusOneCurves}
  \label{prop: new local classification on blow-ups}
  Let $S$ be a smooth proper surface, and let $j\colon E \monoarrow S$ be the embedding of a smooth $(-1)$-curve. Let $\mA \subset \Dbcoh(S)$ be a nonzero admissible subcategory supported on $E$. Then $\mA$ is generated by the exceptional sheaf $j_*\O_E(k)$ for some integer $k \in \Z$.
\end{restatable*}

As an application of this result, we produce some examples of surfaces without phantom subcategories in Corollary~\ref{cor: no phantoms in blow-ups of antifano}, where we show that blow-ups of distinct points on surfaces with globally generated canonical bundles do not contain any phantom subcategories.

In Section~\ref{sec: del pezzos} we consider admissible subcategories in rational elliptic surfaces. For us, as in the paper \cite{heckman-looijenga}, a rational elliptic surface is a smooth proper surface $S$ whose anticanonical linear system $|{-K_S}|$ defines a regular morphism $S \to \P^1$. Any such surface can be obtained by blowing up some subscheme of length nine in the projective plane. We do not have a full classification statement for admissible subcategories in rational elliptic surfaces. However, our methods are at least sufficient to put strong constraints on the possible phantom subcategories.

\begin{restatable*}{theorem}{statemaintheoremRES}
  \label{thm: rational elliptic surfaces}
  Let $\pi\colon S \to \P^1$ be a rational elliptic surface. If $\mB \subset \Dbcoh(S)$ is a phantom subcategory, then the support $\supp(\mB) \subset S$ is contained in the finite union of:
  \begin{itemize}
    \vspace{-\topsep}
    \item sections of the fibration $\pi\colon S \to \P^1$; and
    \item reducible fibers of the morphism $\pi$.
  \end{itemize}
\end{restatable*}

We proceed with a full classification of admissible subcategories, not necessarily phantom ones, whose support is contained in the union of sections of the anticanonical fibration. Note that any section of the fibration $\pi\colon S \to \P^1$ is a smooth $(-1)$-curve in the surface $S$, so Proposition~\ref{prop: new local classification on blow-ups} is the basic case of the classification. In fact, we show that this is essentially the only possibility, and there are no nontrivial admissible subcategories supported on several non-disjoint sections.

\begin{restatable*}{theorem}{statemaintheoremSections}
  \label{thm: admissible subcategories on sections}
  Let $\pi\colon S \to \P^1$ be a rational elliptic surface. Let $\mB \subset \Dbcoh(S)$ be an admissible subcategory such that $\supp(\mB)$ is contained in the union of sections of the map $\pi$. Then there exist finitely many pairwise disjoint sections $\{ E_i \subset S\}_{i \in I}$ and the corresponding integers~$\{ k_i \in \Z \}_{i \in I}$ such that $\mB$ is equal to the subcategory $\bigoplus_{i \in I} \langle \O_{E_i}(k_i E_i) \rangle \subset \Dbcoh(S)$.
\end{restatable*}

For any del Pezzo surface there exists a way to blow-up several points to obtain a rational elliptic surface.  We use the results above to deduce the following consequence for del Pezzo surfaces.

\begin{restatable*}{theorem}{statemaintheoremDelPezzo}
  \label{thm: no phantoms in del pezzos}
  There are no phantom subcategories in del Pezzo surfaces.
\end{restatable*}

The main technical tool that allows us to prove these results is a relation between admissible subcategories and autoequivalences of derived categories of anticanonical divisors, described in Section~\ref{sec: numerical lemmas}. This relation is due to Addington \cite[Prop.~2.1]{addington}. It generalizes the fact that the restriction of an exceptional object to an anticanonical divisor is a spherical object. Addington's result gives amazingly strong structural constraints for the semiorthogonal decompositions of surfaces, and we describe those constraints explicitly in Proposition~\ref{prop: numerical lemma without numbers}. 

\vspace{\baselineskip}

\textbf{Structure of the paper}.
In Section~\ref{sec: preliminaries} we state and prove miscellaneous lemmas about derived categories of coherent sheaves and admissible subcategories. Section~\ref{sec: numerical lemmas} is the technical core of the paper, containing Addington's result and its implications for the constraints that anticanonical divisors put on admissible subcategories. In Section~\ref{sec: projective plane} we prove the classification of admissible subcategories in the derived category of $\P^2$. Section~\ref{sec: minus one curves} contains a classification of admissible subcategories which are supported on a smooth $(-1)$-curve in a surface, and an application for phantom subcategories in some blow-ups. The classification result is used in Section~\ref{sec: del pezzos}, where we study admissible subcategories in rational elliptic surfaces and prove the non-existence of phantoms in del Pezzo surfaces.

\vspace{\baselineskip}

\textbf{Acknowledgements}.
This paper is an improved version of my PhD thesis \cite{mythesis}, which in particular contains a more lyrical list of acknowledgements. I thank my advisor Johan de~Jong for his guidance, enthusiasm, and suggestions. I thank Alexander Kuznetsov for many insightful questions and great expertise. I also thank the members of the thesis committee for their comments. Additionally, it was both pleasant and productive to discuss mathematics with many students in Columbia University and NRU HSE in Moscow. Especially I would like to thank
Raymond Cheng,
Remy van~Dobben de~Bruyn,
Roman Gonin,
Grigory Kondyrev,
Shizhang Li,
Grisha Papayanov,
Artem Prikhodko, and
Semon Rezchikov.



\section{Preliminaries}
\label{sec: preliminaries}

In this section we recall some facts about derived categories of coherent sheaves, exceptional objects, semiorthogonal decompositions, admissible subcategories, and Fourier--Mukai transforms. We recommend consulting the subsections as needed rather than reading this section through.

\subsection{Conventions and notation}

We work over an algebraically closed field $\k$ of characteristic zero. All varieties and triangulated categories in this paper are assumed to be over~$\k$. All functors are assumed to be derived, and a subcategory of a triangulated category is assumed to be a triangulated subcategory.

For an algebraic variety $X$ we denote by $\Dbcoh(X)$ the bounded derived category of coherent sheaves on $X$. We denote by $\Dperf(X)$ the triangulated category of perfect complexes on $X$. When $X$ is smooth, these two categories coincide.

For an algebraic variety $X$ and an object $F \in \Dbcoh(X)$ we denote by $\mH^i(F)$ the $i$'th cohomology sheaf of $F$. We also use the \emph{canonical truncation} $\tau_{\leq i}F$ which has the same cohomology sheaves as $F$ in degrees $\leq i$ and zero cohomology sheaves in degrees strictly greater than $i$. If an object $F \in \Dbcoh(X)$ is represented as a complex of sheaves, $\tau_{\leq i}F$ can be represented as a subcomplex. We define $\tau_{> i}F$ similarly.

\subsection{Exceptional objects and semiorthogonal decompositions}

In this subsection we fix the notation and cite several standard results about triangulated categories, exceptional objects, and semiorthogonal decompositions. These notations, definitions, and results are used throughout the paper. For a more detailed introduction, see, for example, \cite{bond-kapr}.

Until the end of this subsection, we work with an idempotent-complete triangulated category $T$.

For any two objects $A, B \in T$ we denote by $\RHom(A, B)$ the graded vector space $\oplus_{i \in \Z} \Hom_T(A, B[i])$. The graded components are referred as $\mathrm{R}^i\Hom(A, B)$ or $\Ext^i(A, B)$. Similarly, the symbol $\REnd(A)$ denotes $\RHom(A, A)$ and its graded components are referred to as $\End^i(A)$. Given any graded vector space $V^\bullet = \oplus_{i \in \Z} V^i$ and an object $F \in T$, the tensor product $V^\bullet \otimes F$ is an object of $T$ defined to be the direct sum of shifts $\bigoplus_{i \in \Z} F^{\oplus \dim V^i}[-i]$.

For an arbitrary object $F \in T$ we denote by $\langle F \rangle$ the smallest strictly full triangulated subcategory which contains $F$ and is closed under taking direct summands. We say that an object $F$ is a \emph{classical generator} of $T$ if $\langle F \rangle = T$. For any quasi-compact and quasi-separated scheme the category of perfect complexes has a classical generator \cite[Cor.~3.1.2]{bondal-vandenbergh}.

\begin{definition}
An object $E \in T$ is called \emph{exceptional} if $\REnd(E) \caniso \k[0]$. A sequence of exceptional objects $E_1, \ldots, E_n$ is called an \emph{exceptional collection} if $\RHom(E_j, E_i) = 0$ for any $j > i$. An exceptional collection is \emph{full} if the smallest strictly full triangulated subcategory containing every $E_i$ is all of $T$.
\end{definition}

\begin{definition}
  For a full subcategory $\mA \subset T$ we define the left and right \emph{orthogonal subcategories}:
  \begin{align*}
    {}^\perp \mA := \{ F \in T \,\, | \,\, \forall t \in \mA \,\,\, \RHom(F, t) = 0 \}, \\
    \mA^\perp := \{ F \in T \,\, | \,\, \forall t \in \mA \,\,\, \RHom(t, F) = 0 \}.
  \end{align*}
\end{definition}

\begin{lemma}[\cite{bondal-vandenbergh}]
  If $G \in \mA$ is a classical generator, then $F \in {}^\perp \mA$ if and only if $\RHom(F, G) = 0$, and similarly for $\mA^\perp$.
\end{lemma}

\begin{definition}
  A \emph{semiorthogonal decomposition} of a triangulated category $T$ is a sequence of strictly full triangulated subcategories $\mA_1, \ldots, \mA_n$ of $T$ such that $\mA_i \subset \mA_j^\perp$ for any $i < j$ and the smallest strictly full triangulated subcategory containing every $\mA_i$ is $T$. We denote this using angle brackets, i.e., by writing $T = \langle \mA_1, \ldots, \mA_n \rangle$.
\end{definition}

The key property of semiorthogonal decompositions is that any object of $T$ has a filtration whose associated graded components belong to the component subcategories $\mA_i \subset T$. In this paper we work mostly with semiorthogonal decompositions into two components, so to avoid introducing complicated notation, we only state this result for semiorthogonal decompositions like that.

\begin{definition}[\cite{bond-kapr}]
  \label{lem: projection triangles exist}
  Let $T = \langle \mA, \mB \rangle$ be a semiorthogonal decomposition. For any object $F \in T$ there exists a unique \emph{projection triangle} in $T$:
  \begin{equation}
    \label{eqn: prototypical projection triangle}
    \mB_R(F) \to F \to \mA_L(F) \to \mB_R(F)[1]
  \end{equation}
  such that the object $\mB_R(F)$ lies in $\mB$ and $\mA_L(F)$ lies in $\mA$. Moreover, the projection triangle is functorial in $F$, thus we obtain two functors: the \emph{right projection functor} $\mB_R\colon T \to \mB$ which is a right adjoint functor to the inclusion $\mB \monoarrow T$, and the \emph{left projection functor} $\mA_L\colon T \to \mA$ which is a left adjoint functor to the inclusion $\mA \monoarrow T$.
\end{definition}

\begin{corollary}
  \label{cor: universal property of projection triangles}
  Let $T = \langle \mA, \mB \rangle$ be a semiorthogonal decomposition. Let $F \in T$ be any object. The composition with the projection map $\mB_R(F) \to F$ from Definition~\textup{\ref{lem: projection triangles exist}} induces an isomorphism
  \(
    \REnd(\mB_R(F)) \isoarrow \RHom(\mB_R(F), F).
  \)
\end{corollary}
\begin{proof}
  This follows from the fact that the functor $\mB_R$ is an adjoint functor to the inclusion functor $\mB \monoarrow T$. Alternatively, we can deduce the statement from semiorthogonality: an application of the functor $\RHom(\mB_R(F), -)$ to the triangle (\ref{eqn: prototypical projection triangle}) results in the triangle
  \[
     \REnd(\mB_R(F)) \to \RHom(\mB_R(F), F) \to \RHom(\mB_R(F), \mA_L(F))
   \]
   in the derived category of vector spaces. Since $\mA$ is semiorthogonal to $\mB$, the graded vector space $\RHom(\mB_R(F), \mA_L(F))$ vanishes. Therefore the first arrow is an isomorphism.
\end{proof}

Exceptional collections may be used to construct many examples of semiorthogonal decompositions. A common abuse of notation in this context is to write an exceptional object~$E$ as a component in the semiorthogonal decomposition, having in mind the triangulated subcategory $\langle E \rangle \subset T$ generated by that object.

\begin{lemma}[\cite{bond-kapr}]
  \label{lem: exceptional collections are admissible}
  Let $\langle E_1, \ldots, E_n \rangle$ be an exceptional collection in $T$. Suppose that for any two objects $F, G \in T$ the graded vector space $\RHom_T(F, G)$ has finite total dimension.
  \begin{itemize}
  \item Let $\mA$ be the right orthogonal subcategory $\langle E_1, \ldots, E_n \rangle^\perp$. Then the sequence
    \[
      \langle \mA, E_1, \ldots, E_n \rangle
    \]
    is a semiorthogonal decomposition of $T$. If the exceptional collection consists of one object $E \in T$, then the projection functor $R_E$ is given by $F \mapsto E \otimes \RHom_T(E, F)$ and the projection triangle for $T = \langle \mA, E \rangle$ is a cone of the evaluation morphism
    \[
      E \otimes \RHom_T(E, F) \xrightarrow{\mathrm{ev}} F \to L_{E^\perp}(F).
    \]
  \item  Let $\mA$ be the left orthogonal subcategory ${}^\perp \langle E_1, \ldots, E_n \rangle$. Then the sequence
    \[
      \langle E_1, \ldots, E_n, \mA \rangle
    \]
    is a semiorthogonal decomposition of $T$. If the exceptional collection consists of one object $E \in T$, then the projection functor $L_E$ is given by $F \mapsto \RHom_T(F, E)^\dual \otimes E$ and the projection triangle for $T = \langle E, \mA \rangle$ is a fiber of the coevaluation morphism
    \[
      R_{{}^\perp E}(F) \to F \xrightarrow{\mathrm{coev}} \RHom_T(F, E)^\dual \otimes E.
    \]
  \end{itemize}
\end{lemma}

\begin{remark}
  The projection triangles for longer exceptional collections may also be written explicitly, in terms of \emph{dual exceptional collections}, as in \cite{kapranov88}. We omit this since this is not necessary for our paper.
\end{remark}

\subsection{Derived categories of coherent sheaves}

We continue with several miscellaneous lemmas, mostly related to homological algebra. We will use the following definitions throughout the paper.

\begin{definition}
  \label{def: support of a complex}
  Let $X$ be an algebraic variety, and let $E \in \Dbcoh(X)$ be an object. The (set-theoretic) \emph{support} $\supp(E)$ of the object $E$ is the union $\cup_{i \in \Z} \supp(\mH^i(E))$ of supports of cohomology sheaves.
\end{definition}

\begin{lemma}[{\cite[Ex.~3.30]{HuybFM}}]
  \label{lem: support and derived fibers}
  Let $X$ be an algebraic variety, and let $E \in \Dbcoh(X)$ be an object. Then a point $p \in X$ lies in $\supp(E)$ if and only if $\RHom(E, \O_p) \neq 0$, where $\O_p$ is the skyscraper sheaf at the point $p$.
\end{lemma}

\begin{lemma}[{\cite[Lem.~3.9]{HuybFM}}]
  \label{lem: disjoint support forces splitting}
  Let $X$ be an algebraic variety, and let $E \in \Dbcoh(X)$ be an object. Suppose that $\supp(E)$ is a disjoint union $Z_1 \sqcup Z_2$ of two closed subsets of $X$. Then there exists a unique decomposition $E \iso E_1 \oplus E_2$ into a direct sum such that $\supp(E_1) = Z_1$ and $\supp(E_2) = Z_2$.
\end{lemma}

\begin{definition}
  \label{def: locally free object}
  Let $X$ be an algebraic variety. An object $E \in \Dbcoh(X)$ is called \emph{locally free} if all cohomology sheaves of $E$ are locally free. Similarly, it is called a \emph{torsion object} if all cohomology sheaves are torsion sheaves.
\end{definition}

There are multiple ways to define locally free objects in a derived category. In the following lemma we show some equivalent characterizations. The lemma is well-known, but we include the proof due to the lack of a convenient reference.

\begin{definition}
  The \emph{length} of a graded vector space $V^\bullet$ is the number $l(V^\bullet) := \sum_i \dim V^i$. The length of a complex of vector spaces is the length of its cohomology viewed as a graded vector space.
\end{definition}

\begin{lemma}
  \label{lem: locally free has many definitions}
  Let $X$ be a smooth algebraic variety, and let $E \in \Dbcoh(X)$ be an object. The following are equivalent:
  \begin{enumerate}
    \item $E$ is a locally free object.
    \item For any point $x \in X$ there exists a Zariski-neighborhood $U \subset X$ containing $x$ such that the restriction $E|_U$ is isomorphic to $\O_U \otimes V^\bullet$ for some graded vector space $V^\bullet$.
    \item The length of the derived fiber of $E$ at each point $x \in X$ is the same.
  \end{enumerate}
\end{lemma}
\begin{proof}
  It is clear that the condition $(2)$ implies both $(1)$ and $(3)$. It is enough to show that $(1)$ implies $(2)$, and that $(3)$ implies $(1)$.

  $(1) \implies (2)$: Let $U \subset X$ be an affine open neighborhood of $x \in X$ such that each cohomology sheaf of $E$ becomes trivial. Such a neighborhood exists since $E$ has only finitely many nonzero cohomology sheaves. Then each cohomology sheaf of $E|_U$ is a direct sum of several copies of the structure sheaf $\O_U$. Since $U$ is affine, there are no higher $\Ext$'s between copies of the structure sheaf, and hence the complex $E|_U$ is formal, i.e., quasiisomorphic to a direct sum of its cohomology sheaves.

  $(3) \implies (1)$: For each point $x \in X$ denote by $\iota_x\colon \Spec \k \monoarrow X$ the inclusion morphism. For any $k \in \Z$ the dimension of the $k$'th derived pullback functor $L_k \iota_x^*(E)$ is an upper semicontinuous function, so the total length of the object $\iota_x^*(E)$ is constant if and only if the dimension of each $L_k \iota_x^*(E)$ is constant as a function of $x \in X$. Assume that some cohomology sheaf $\mH^i(E)$ is not locally free. Without loss of generality we may assume that each cohomology sheaf $\mH^j(E)$ with $j > i$ is locally free. Consider the spectral sequence for the derived pullback $\iota_x^*$ \cite[(3.10)]{HuybFM}:
  \[
    E_2^{p, q} = L_{-q} \iota_x^*(\mH^p(E)), \,\, d_r^{p, q}\colon E_r^{p, q} \to E_r^{p - r + 1, q + r} \quad \Rightarrow \quad \mH^{p+q}(\iota_x^*E).
  \]
  For any $j > i$ by assumption we know that $L_q \iota_x^*(\mH^j(E)) = 0$ for $q > 0$. This implies that the cell $E_2^{i, 0} = L_0 \iota_x^*(\mH^i(E))$ survives to $E_\infty$. In particular, $L_i \iota_x^*(E) \iso L_0 \iota_x^*(\mH^i(E))$ for any point $x \in X$. Since $\mH^i(E)$ is not locally free, its (nonderived) rank $L_0 \iota_x^*(\mH^i(E))$ is not a constant function, but then the dimension of $L_i \iota_x^*(E)$ is also not constant, a contradiction.
\end{proof}

\begin{lemma}
  \label{lem: morphisms from topmost cohomology sheaf extend}
  Let $X$ be an algebraic variety, and let $F \in \Dbcoh(X)$ be an object concentrated in nonpositive cohomology degrees. Then for any coherent sheaf $\mathcal{F}$ on $X$ there is a canonical isomorphism
  \[
    \mathrm{R}^0\Hom(\mH^0(F), \mathcal{F}) \isoarrow \mathrm{R}^0\Hom(F, \mathcal{F}).
  \]
\end{lemma}
\begin{proof}
  Let $\tau_{\leq -1}F$ denote the canonical truncation of the complex $F$. There exists a truncation triangle
  \[
    \tau_{\leq -1}F \to F \to \mH^0(F) \to (\tau_{\leq -1}F) [1].
  \]
  The application of the cohomological functor $\mathrm{R}^0\Hom(-, \mathcal{F})$ together with the fact that there are no negative $\Ext$'s between coherent sheaves finishes the proof of the lemma.
\end{proof}

\begin{lemma}
  \label{lem: derived restriction to a divisor}
  Let $Y$ be a variety, and let $j\colon D \monoarrow Y$ be an embedding of a Cartier divisor. Let $F \in \Dperf(Y)$ be an object. Then for every $i \in \Z$:
  \begin{enumerate}
  \item there exists a short exact sequence
  \[
    0 \to L_0 j^* \mH^i(F) \to \mH^i(j^*F) \to L_1j^*\mH^{i+1}(F) \to 0.
  \]
  \item $\supp(\mH^i(F)) \cap D \subset \supp \mH^i(j^*F)$;
  \item If $\mH^i(j^*F) = 0$, then the support of $\mH^i(F)$ does not intersect $D$.
  \end{enumerate}
\end{lemma}
\begin{proof}
  Consider the spectral sequence converging to the cohomology sheaves of the derived pullback $j^*F$:
  \[
    E_2^{p, q} = L_{-q}j^* \mH^p(F), \,\, d_r^{p, q}\colon E_r^{p, q} \to E_r^{p - r + 1, q + r} \quad \implies \quad \mH^{p+q}(j^*F).
  \]
  Since $j\colon D \monoarrow Y$ is an inclusion of a Cartier divisor, the $E_2$-page of that spectral sequence has only two rows, and therefore it degenerates at the second page by dimension reasons, producing a collection of short exact sequences as in the statement. The other two claims in the statement easily follow from this observation.
\end{proof}

The derived categories of coherent sheaves on curves and surfaces have some special convenient properties, which makes them easier to deal with than the derived categories for higher-dimensional varieties. We recall some of the properties in the following several well-known lemmas, and include the sketches of proofs for completeness.

\begin{lemma}
  \label{lem: objects on smooth curves split}
  Let $C$ be a smooth curve, and let $W \in \Dbcoh(C)$ be an object.
  \begin{enumerate}
  \item There is a decomposition $W \iso \bigoplus_{i \in \Z} \mathcal{H}^i(W)[-i]$ into a direct sum of shifts of cohomology sheaves.
  \item There is a direct sum decomposition $W \iso T \oplus V$ where $T$ is a torsion object and $V$ is a locally free object.
  \end{enumerate}
\end{lemma}
\begin{proof}
  The first claim follows from the fact that the category of coherent sheaves on a smooth curve has homological dimension one, see, e.g., \cite[Cor.~3.15]{HuybFM}. Consequently, it is enough to prove the second claim for coherent sheaves. Let $\mathcal{F}$ be a coherent sheaf on $C$. Denote by $\mathcal{T} \subset \mathcal{F}$ the torsion subsheaf. Then there is a short exact sequence
  \[
    0 \to \mathcal{T} \to \mathcal{F} \to \mathcal{F} / \mathcal{T} \to 0.
  \]
  The quotient sheaf $\mathcal{F} / \mathcal{T}$ is torsion-free on a smooth curve, so it is locally free. Then the space $\Ext^1(\mathcal{F}/ \mathcal{T}, \mathcal{T})$ vanishes and the extension splits.
\end{proof}

\begin{lemma}
  \label{lem: smooth torsion sheaves on curves}
  Let $C$ be a curve, and let $W$ be a coherent sheaf on $C$ supported at a smooth point $p \in C$.
  \begin{enumerate}
    \item There is a direct sum decomposition $W \iso \bigoplus_k \left(\O_C/\mathfrak{m}^k\right)^{\oplus w_k}$, where $\mathfrak{m}$ is the ideal sheaf of the point $p$ and $\{ w_k \}$ is some set of multiplicities.
    \item If $W \iso \O_C/\mathfrak{m}^n$ and $W^\prime \iso \O_C/\mathfrak{m}^m$ are two indecomposable torsion coherent sheaves on $C$ supported at point $p$, then $\dim \Hom_C(W, W^\prime) = \dim \Ext^1_C(W, W^\prime) = \min(m, n)$.
  \end{enumerate}
\end{lemma}

\begin{proof}
  A local ring of $C$ at a smooth point $p$ is a discrete valuation ring \cite[Prop.~11.1]{eisenComAl}. In particular it is a principal ideal domain. The classification of finitely generated modules over a PID establishes the first claim. If $f \in \mathfrak{m}$ is a generator, then the sheaf $\O_C/\mathfrak{m}^n$ has a two-term locally free resolution
  \[
    0 \to \O_C \xrightarrow{f^n} \O_C \to \O_C/\mathfrak{m}^n \to 0,
  \]
  which lets us compute $\Hom$ and $\Ext$ for the second part of the statement.
\end{proof}

\begin{lemma}
  \label{lem: complexes on a surface are given by adjacent glueings}
  Let $S$ be a smooth surface. A choice of an object $M \in \Dbcoh(S)$ up to an isomorphism is the same as a choice of the following two pieces of information:
  \begin{enumerate}
  \item a collection of cohomology sheaves $F^i := \mathcal{H}^i(M)$;
  \item a collection of glueing maps $\xi_i\in \Ext^2(F^i, F^{i-1})$.
  \end{enumerate}
\end{lemma}

\begin{remark}
  In general, the glueing data for an object in the derived category also includes additional information related to higher $\Ext$'s, and it is not easy to describe explicitly. On a smooth surface all $\Ext$'s of degree larger than two between coherent sheaves vanish, and this gives us a simpler description.
\end{remark}

\begin{proof}
  If $M$ is concentrated in a single cohomological degree, this is clear. Assume that the claim is proved for complexes concentrated in at most $n$ degrees, and let $M$ be an object concentrated in exactly $n+1$ different cohomological degrees. Let $i$ be the largest integer such that $\mH^i(M) \neq 0$. Consider the truncation triangle
  \[
    \tau_{\leq i-1}M \to M \to \mH^i(M)[-i] \to (\tau_{\leq i-1}M)[1].
  \]
  The object $M$ is determined up to an isomorphism by its truncation $\tau_{\leq i-1}M$ and the glueing map $\xi \in \Ext^1(\mH^i(M)[-i], \tau_{\leq i-1}M)$. By induction the lemma holds for the truncation. Thus it remains to show that $\Ext^1(\mH^i(M)[-i], \tau_{\leq i-1}M) \iso \Ext^2(\mH^{i}(M), \mH^{i-1}(M))$.

  Consider the truncation triangle for $\tau_{\leq i-1}M$:
  \[
    \tau_{\leq i-2}M \to \tau_{\leq i-1}M \to \mH^{i-1}(M)[-i+1] \to (\tau_{\leq i-2}M) [1].
  \]
  An application of the functor $\Ext^\bullet(\mH^i(M)[-i], -)$ leads to a long exact sequence of vector spaces. Since the homological dimension of a smooth surface is two, both vector spaces $\Ext^1(\mH^i(M)[-i], \tau_{\leq i-2}M)$ and $\Ext^1(\mH^i(M)[-i], (\tau_{\leq i-2}M)[1])$ vanish by dimension reasons. Thus the induction step is established.
\end{proof}

\begin{lemma}
  \label{lem: torsion-free sheaves on a surface}
  Let $S$ be a smooth surface, and let $\mathcal{F}$ be a torsion-free coherent sheaf on $S$. Then there exists a unique up to a unique isomorphism short exact sequence
  \[
    0 \to \mathcal{F} \to \mathcal{E} \to \mathcal{Q} \to 0,
  \]
  where $\mathcal{E}$ is locally free, and $\mathcal{Q}$ is a torsion sheaf supported on a zero-dimensional subset.
\end{lemma}
\begin{proof}
  Any morphism from $\mathcal{F}$ to a locally free sheaf factors through the double dual coherent sheaf $\mathcal{F}^{\dual \dual}$. On a smooth surface the double dual is locally free \cite[Lem.~2.1.1.10]{okonek-schneider-spindler}. The morphism $\mathcal{F} \to \mathcal{F}^{\dual \dual}$ is an isomorphism on an open set where $\mathcal{F}$ is locally free. The complement to that open set has codimension two \cite[Lem.~2.1.1.8]{okonek-schneider-spindler}, so the quotient is a zero-dimensional torsion sheaf. Uniqueness follows from the universal property of the double dual.
\end{proof}

\begin{lemma}
  \label{lem: torsion-free sheaves and divisors}
  Let $S$ be a smooth surface, and let $\mathcal{F}$ be a torsion-free coherent sheaf on $S$. For any divisor $j\colon D \monoarrow S$ the derived restriction $j^*\mathcal{F} \in \Dbcoh(D)$ is concentrated only in degree $0$.
\end{lemma}
\begin{proof}
  The object $j_* j^* \mathcal{F} \in \Dbcoh(S)$ can be represented as a cone of a morphism
  \[
    \mathcal{F} \otimes \O(-D) \to \mathcal{F}.
  \]
  Since $\mathcal{F}$ is torsion-free, this map is injective, and hence $j_* j^* \mathcal{F}$ is concentrated in degree zero. Since the pushforward $j_*$ is an exact functor, this implies that $j_*(L_1 j^* \mathcal{F}) = 0$. A pushforward of a nonzero coherent sheaf along the closed embedding is nonzero, so in fact $L_1 j^* \mathcal{F} = 0$, as claimed.
\end{proof}

\subsection{Spectral sequences for $\Ext$-groups}

In this subsection we describe two useful spectral sequences for computing $\Ext$'s between objects in the derived category. In this paper we need them only for Lemma~\ref{lem: objects with a special point have large ext1}, which is used in Section~\ref{sec: projective plane}. The spectral sequences are well-known (e.g., they are used in \cite{kuleshov-orlov}), but we include explicit statements for convenience of reference.

First, we discuss the spectral sequence that computes the self-$\Ext$'s of an object in the derived category in terms of the $\Ext$'s between its cohomology sheaves. It is a special case of the spectral sequence for $\Ext$'s between two objects admitting lifts to a filtered derived category, constructed in \cite[(3.1.3)]{bbd}. We work with the usual derived category, however any object has a canonical filtration whose associated graded factors are quasiisomorphic to the cohomology sheaves. We describe the resulting spectral sequence in this case explicitly for convenience.

\begin{lemma}
  \label{lem: self-ext spectral sequence}
  Let $X$ be a smooth algebraic variety, and let $F \in \Dbcoh(X)$ be an arbitrary object. There exists a $E_1$-spectral sequence with
  \[
    E_1^{p, q} = \bigoplus_{i \in \Z} \Ext^{2p+q}(\mH^i(F), \mH^{i-p}(F)) \qquad d_r^{p, q}\colon E_r^{p, q} \to E_r^{p+r, q -r + 1}
  \]
  which converges to $\Ext^{p+q}(F, F)$. The $d_1$ differential is given by pre- and post-compositions with glueing maps $\xi_{i+1} \in \Ext^2(\mH^{i+1}(F), \mH^i(F))$ and $\xi_{i-p} \in \Ext^2(\mH^{i-p}(F), \mH^{i-p-1}(F))$.
\end{lemma}

\begin{proof}
  Since $F$ is a bounded complex and smooth varieties have finite homological dimension, it is possible to find an injective resolution for $F$ which is a bounded complex equipped with a decreasing filtration whose associated graded factors are injective resolutions for the cohomology sheaves $\mH^i(F)$ such that the filtration in each degree is split. The resolution with this filtration represents an object in the filtered derived category. The spectral sequence in \cite[(3.1.3.4)]{bbd} computing $\Ext(F, F)$ in the usual derived category is the spectral sequence claimed in the statement.
\end{proof}

\begin{corollary}
  \label{cor: self-exts on a surface}
  Let $S$ be a smooth surface, and let $F \in \Dbcoh(S)$ be an object in the derived category. Then
  \[
    \dim \Ext^1(F, F) \geq \sum_{i \in \Z} \dim \Ext^1(\mH^i(F), \mH^i(F)).
  \]
\end{corollary}

\begin{proof}
  Consider the spectral sequence from Lemma~\ref{lem: self-ext spectral sequence}. Note that
  \[
    E_1^{0, 1} = \bigoplus_{i \in \Z}  \Ext^1(\mH^i(F), \mH^i(F)).
  \]
  On a smooth variety of dimension $n$ the spectral sequence degenerates at $E_n$ for dimension reasons, thus on a surface the only nonzero differential is $d_1$. Consider the cell $E_1^{0, 1}$ and the $d_1$-differentials starting and ending on that cell:
  \[
    \bigoplus_{i \in \Z} \Ext^{-1}(\mH^i(F), \mH^{i+1}(F))
    \xrightarrow{d_1}
    \fbox{$\displaystyle\bigoplus_{i \in \Z} \Ext^1(\mH^i(F), \mH^i(F))$}
    \xrightarrow{d_1}
    \bigoplus_{i \in \Z} \Ext^3(\mH^i(F), \mH^{i-1}(F)).
  \]
  On a smooth surface both $\Ext^{-1}$ and $\Ext^3$ between coherent sheaves are always zero, thus the vector space in $E_1^{1, 0}$ survives to $E_\infty$ and is a subquotient of $\Ext^1(F, F)$. This implies the inequality for dimensions of those vector spaces.
\end{proof}

Another useful spectral sequence is the following one. It lets us compute $\Ext$'s between cones of maps in $\Dbcoh(X)$. An important class of cones to keep in mind is the ones coming from short exact sequences of coherent sheaves on $X$.

\begin{lemma}
  \label{lem: cone spectral sequence}
  Let $X$ be a smooth algebraic variety. Suppose that there are two distinguished triangles in $\Dbcoh(X)$:
  \[
    A_1 \to B_1 \to C_1 \to A_1[1] \qquad A_2 \to B_2 \to C_2 \to A_2[1].
  \]
  There exists a $E_1$-spectral sequence which degenerates at $E_3$ and converges to $\Ext^*(C_1, C_2)$:
  \[
    E_1^{p, q} =
    \begin{cases}
      \Ext^q(B_1, A_2), & p = -1; \\
      \Ext^q(A_1, A_2) \oplus \Ext^q(B_1, B_2), & p = 0; \\
      \Ext^q(A_1, B_2), & p = 1; \\
      0, & \text{otherwise.}
    \end{cases}
  \]
  with differential $d_r^{p, q}\colon E_r^{p, q} \to E_r^{p+r, q - r + 1}$. The differential $d_1$ is given by compositions with the morphisms $A_1 \to B_1$ and $A_2 \to B_2$.
\end{lemma}

The key observation is that both $C_1$ and $C_2$ lift to objects in the filtered derived category. There are various notions of a filtration on an object in the triangulated category $\Dbcoh(X)$, and most of them do not allow lifting the object to the filtered category, but the two-step filtrations arising from the distinguished triangles are always sufficient.

\begin{proof}
  Choose injective resolutions for $A_1$ and $B_1$. Then the morphism $A_1 \to B_1$ in the derived category may be represented as an actual map of complexes. The cone of this map of complexes is a complex representing the object $C_1 \in \Dbcoh(X)$. This cone is equipped with a filtration whose associated graded components are quasiisomorphic to $A_1$ and $B_1$ respectively. A similar procedure applied to $C_2$ lets us conclude by invoking \cite[(3.1.3.4)]{bbd} again.
\end{proof}

We may use the spectral sequences from Lemmas~\ref{lem: self-ext spectral sequence} and~\ref{lem: cone spectral sequence} to obtain the following property of objects on smooth surfaces.

\begin{lemma}
  \label{lem: objects with a special point have large ext1}
  Let $S$ be a smooth surface, and let $p \in S$ be a point. Assume that $F \in \Dbcoh(S)$ is an object which is locally free away from $p$, but not locally free at $p$. Then $\dim \Ext^1(F, F) \geq 2$.
\end{lemma}

\begin{remark}
  For some surfaces such as $\P^2$ there is a geometric argument for this inequality. Consider a two-dimensional family of automorphisms of $\P^2$ which moves the point $p$ around. The pullbacks of $F$ with respect to that family form a deformation of $F$ over a two-dimensional base. It may be checked that, in characteristic zero, the first-order deformation along any direction of the two-dimensional base is nontrivial, and therefore $\dim \Ext^1(F, F) \geq 2$.
\end{remark}

\begin{proof}
  If $F$ is not locally free at $p$, by definition this means that at least one cohomology sheaf of $F$ is not locally free at $p$. By Corollary~\ref{cor: self-exts on a surface} it is enough to prove the inequality for the dimension of self-$\Ext^1$ of that cohomology sheaf. So suppose that $\mathcal{F}$ is a coherent sheaf which is not locally free at $p$, but locally free on the complement $S \setminus \{ p \}$. The inequality for coherent sheaves is related to the inequalities in \cite[Cor.~2.11~and~2.12]{mukai}, but we include a direct proof for completeness. We consider several cases to prove the inequality.

  Suppose first that $\mathcal{F}$ is a torsion sheaf supported at $p$. Then the Euler characteristic $\chi(\mathcal{F}, \mathcal{F})$ is zero since it stays constant in flat families and the sheaf $\mathcal{F}$ may be deformed by moving the point $p$ in a flat family. Since we are on a smooth surface we may use Serre duality to find the following expression for Euler characteristic. Note that the canonical bundle is trivial in a neighborhood of the point $p$, so:
  \[
    \chi(\mathcal{F}, \mathcal{F}) = 2 \cdot \dim \Hom(\mathcal{F}, \mathcal{F}) - \dim \Ext^1(\mathcal{F}, \mathcal{F}).
  \]
  The sheaf $\mathcal{F}$ is nonzero, so $\dim \Hom(\mathcal{F}, \mathcal{F}) \geq 1$. Therefore $\dim \Ext^1(\mathcal{F}, \mathcal{F}) \geq 2$.

  Suppose now that $\mathcal{F}$ is a torsion-free sheaf which is not locally free at $p$. Then  by Lemma~\ref{lem: torsion-free sheaves on a surface} there exists a short exact sequence
  \[
    0 \to \mathcal{F} \to \mathcal{E} \to \mathcal{Q} \to 0,
  \]
  where $\mathcal{E}$ is locally free and $\mathcal{Q}$ is a nonzero torsion sheaf supported at the point $p$. Consider the spectral sequence from Lemma~\ref{lem: cone spectral sequence} which computes $\Ext^*(\mathcal{F}, \mathcal{F})$ in terms of that short exact sequence. The $E_1$ page contains the following fragment:
  \[
    \Ext^1(\mathcal{Q}, \mathcal{E})
    \xrightarrow{d_1}
    \Ext^1(\mathcal{Q}, \mathcal{Q}) \oplus \Ext^1(\mathcal{E}, \mathcal{E})
    \xrightarrow{d_1}
    \Ext^1(\mathcal{E}, \mathcal{Q}).
  \]
  Since $\mathcal{E}$ is locally free and $\mathcal{Q}$ is supported on a zero-dimensional set, it is easy to see that $\Ext^\bullet(\mathcal{E}, \mathcal{Q})$ is concentrated only in degree zero and $\Ext^\bullet(\mathcal{Q}, \mathcal{E})$ is concentrated only in degree two. Thus the vector space in $E_1^{0, 1}$-cell, which contains $\Ext^1(\mathcal{Q}, \mathcal{Q})$ as a subspace, survives to $E_2$. By dimension reasons there are no nonzero differentials on the $E_2$ page which start or end at $E_2^{0, 1}$. Hence
  \[
    \dim \Ext^1(\mathcal{F}, \mathcal{F}) \geq \dim \Ext^1(\mathcal{Q}, \mathcal{Q}).
  \]
  From the previous case we considered we know that the right hand side is at least two, which confirms the claim for torsion-free sheaves.

  It remains to consider the case where $\mathcal{F}$ has a nonzero torsion subsheaf with a nonzero torsion-free quotient:
  \[
    0 \to \mathcal{T} \to \mathcal{F} \to \mathcal{G} \to 0.
  \]
  Consider again the spectral sequence from Lemma~\ref{lem: cone spectral sequence} which computes $\Ext^\bullet(\mathcal{F}, \mathcal{F})$. The $E_1$-page contains the following fragment:
  \[
    \Hom(\mathcal{T}, \mathcal{G})
    \xrightarrow{d_1}
    \Ext^1(\mathcal{T}, \mathcal{T}) \oplus \Ext^1(\mathcal{G}, \mathcal{G})
    \xrightarrow{d_1}
    \Ext^2(\mathcal{G}, \mathcal{T}).
  \]
  Since $\mathcal{G}$ is torsion-free, $\Hom(\mathcal{T}, \mathcal{G}) = 0$. Using the embedding from Lemma~\ref{lem: torsion-free sheaves on a surface} it is easy to see that $\Ext^2(\mathcal{G}, \mathcal{T})$ is also zero. Thus the $E_1^{0, 1}$-cell survives to $E_2$, and similarly to the previous case purely by dimension reasons it survives to $E_\infty$. Therefore the lemma is proved for all coherent sheaves, and hence for all objects in the derived category as well.
\end{proof}

\subsection{Admissible subcategories and their properties}
We begin with several general observations about admissible subcategories and also consider their consequences for admissible subcategories of projective spaces, especially $\P^2$.

\begin{definition}
  Let $X$ be an algebraic variety. A strictly full triangulated subcategory $\mA \subset \Dbcoh(X)$ is an \emph{admissible subcategory} if the inclusion functor admits both left and right adjoint functors. We denote the left adjoint by $\mA_L$ and the right adjoint by $R_\mA$.
\end{definition}

For smooth and proper varieties an admissible subcategory is essentially the same thing as a semiorthogonal decomposition with two components, and the choice of notation for adjoints is compatible with Definition~\ref{lem: projection triangles exist}. More precisely, we have the following statement.

\begin{lemma}[\cite{bond-kapr}]
  \label{lem: admissible subcategories are semiorthogonal decompositions}
  Let $X$ be a smooth and proper algebraic variety. If $\Dbcoh(X) = \langle \mA, \mB \rangle$ is a semiorthogonal decomposition, then both $\mA$ and $\mB$ are admissible subcategories of $\Dbcoh(X)$. Conversely, if $\mA \subset \Dbcoh(X)$ is an admissible subcategory, then both $\langle \mA^\perp, \mA \rangle$ and $\langle \mA, {}^\perp \mA \rangle$ are semiorthogonal decompositions of $\Dbcoh(X)$.
\end{lemma}

The main property of admissible subcategories in the geometric situation is the fact that they are closed under small deformations of objects in the following sense:

\begin{proposition}[{\cite[Cor.~3.12]{kawatani-okawa}}]
  \label{prop: admissible subcategories are open}
  Let $X$ be a smooth proper algebraic variety. Let $\mA \subset \Dbcoh(X)$ be an admissible subcategory. For any smooth variety $Y$ with a chosen point $y \in Y$, and any object $R \in \Dbcoh(X \times Y)$ such that the derived restriction $R|_{X \times \{ y \} } \in \Dbcoh(X)$ is in $\mA$, there exists a Zariski open neighborhood $U \subset Y$ of the point $y$ such that $R|_{ X \times \{ u \} } \in \mA$ for any $u \in U$. Moreover, $\mA$ is invariant under the action of the connected automorphism group $\mathrm{Aut}^\circ(X)$.
\end{proposition}

We often use the following definition.

\begin{definition}
  \label{def: support of a category}
  Let $X$ be an algebraic variety, and let $\mA \subset \Dbcoh(X)$ be an admissible subcategory. We define the support $\supp(\mA)$ to be the union of set-theoretic supports of all objects in~$\mA$.
\end{definition}

\begin{lemma}
  \label{lem: support of a category}
  Let $X$ be a smooth and proper algebraic variety, and let $\mA \subset \Dbcoh(X)$ be an admissible subcategory. Then $\supp(\mA) \subset X$ is a Zariski-closed subset which does not have any isolated points.
\end{lemma}
\begin{proof}
  Since $X$ is smooth and proper, by \cite[Th.~3.1.4]{bondal-vandenbergh} its derived category $\Dbcoh(X)$ has a classical generator $G$. The (left) projection of a classical generator to $\mA \subset \Dbcoh(X)$ is a classical generator $\mA_L(G)$ of the category $\mA$. Then $\supp(\mA) = \supp(\mA_L(G))$. In particular, this is a closed subset.

  Assume that $\supp(\mA)$ has an isolated point $p \in X$. Then by definition there exists an object $F \in \mA$ whose support contains $p$ as an isolated point. By Lemma~\ref{lem: disjoint support forces splitting} we see that the object $F$ has a direct summand $T$ supported only at the point $\{ p \}$. Admissible subcategories are closed under taking direct summands, so $T \in \mA$. Since $X$ is smooth, any object supported on a single point may be deformed by moving the point in $X$. As admissible subcategories are closed under small deformations by Proposition~\ref{prop: admissible subcategories are open}, we conclude that for each point $q$ in some Zariski-neighborhood of $p \in X$ there exists an object $T_q \in \mA$ supported only at the point $q$. This contradicts the fact $p$ is an isolated point in $\supp(\mA)$.
\end{proof}

In general, it is very difficult to control even the basic behavior of admissible subcategories. For example, the following question is still open:

\begin{conjecture}[{\cite{kuznetsov-oldhochschild}}]
  Let $X$ be a smooth projective variety. If 
  \[
    \mA_1 \subset \mA_2 \subset \ldots \subset \Dbcoh(X)
  \]
  is an infinite increasing chain of admissible subcategories, then it stabilizes at some finite step.
\end{conjecture}

Several results of this paper are related to phantom subcategories:

\begin{definition}
  Let $X$ be a smooth and proper variety, and let $\mA \subset \Dbcoh(X)$ be an admissible subcategory. It is called a \emph{phantom subcategory} if $\mA \neq 0$ and the Grothendieck group $K_0(\mA)$ vanishes.
\end{definition}

It is not easy to construct examples of phantom subcategories. It is expected that they do not exist for nice varieties, such as homogeneous spaces \cite[Rem.~1.7]{kuznetsov-polishchuk} or varieties admitting a full exceptional collection \cite[Conj.~1.10]{kuznetsov-icm}. In this paper we confirm this expectation for del Pezzo surfaces.

The invariance of admissible subcategories under the connected automorphism group, shown in Proposition~\ref{prop: admissible subcategories are open}, has several important implications.

\begin{lemma}
  \label{lem: projections are g-invariant}
  Let $X$ be a smooth proper variety. Let $\mA \subset \Dbcoh(X)$ be an admissible subcategory, and let $F \in \Dbcoh(X)$ be an object. Consider the projection triangle as in Definition~\textup{\ref{lem: projection triangles exist}}:
  \[
    B \to F \to A
  \]
  with $A \caniso \mA_L(F) \in \mA$ and $B \in {}^\perp \mA$. Assume that the object $F$ is invariant under the action of some subgroup $G \subset \mathrm{Aut}^\circ(X)$. Then both projections $A$ and $B$ are also invariant under the action of $G$.
\end{lemma}
\begin{proof}
  Pick an automorphism $g \in G$. By Proposition~\ref{prop: admissible subcategories are open} the pullbacks $g^*A$ and $g^*B$ lie in the subcategories $\mA$ and ${}^\perp \mA$ respectively. Thus the pullback of the projection triangle is another decomposition of $F \iso g^*F$ into components from $\mA$ and ${}^\perp \mA$. Such a decomposition is unique, thus $g^*A \iso A$ and $g^*B \iso B$.
\end{proof}

\begin{corollary}
  \label{cor: pgl-invariant generators exist}
  Every admissible subcategory of $\Dbcoh(\P^n)$ has a $\PGL(n+1)$-invariant classical generator.
\end{corollary}
\begin{proof}
  The category $\Dbcoh(\P^n)$ has a $\PGL(n+1)$-invariant classical generator $G = \bigoplus_{0 \leq i \leq n} \O(i)$. Let $\mA_L$ be the projection functor to $\mA$ as in Definition~\ref{lem: projection triangles exist}. Then $\mA_L(G)$ is a classical generator of $\mA$ which is $\PGL(n+1)$-invariant by Lemma~\ref{lem: projections are g-invariant}.
\end{proof}

\begin{lemma}
  \label{lem: maps from torsion cannot be nonzero on cohomology}
  Let $X$ be a smooth variety. Let $\mA \subset \Dbcoh(X)$ be an admissible subcategory, and let $A \in \mA$ be an object. Let $\O_p$ be a skyscraper sheaf at some point $p \in X$. If there exists a morphism $\O_p \to A[a]$ for some shift $a \in \Z$ which induces a nonzero map on the zeroth cohomology sheaves, then any object of the subcategory ${}^\perp \mA$ is set-theoretically supported on the complement to the point $X \setminus \{ p \}$.
\end{lemma}

\begin{proof}
  Let $B \in {}^\perp\mA$ be any object. Suppose that at least one of its cohomology sheaves is not zero at $p$. Without loss of generality we may assume that the support of $\mH^0(B)$ contains $p$, while the supports of $\mH^i(B)$ for $i > 0$ do not. It is easy to check that
  \[
    \mathrm{R}^0\Hom(B, \O_p) \caniso \mathrm{R}^0\Hom(\mH^0(B), \O_p) \neq 0.
  \]
  Pick any nonzero map $f\colon B \to \O_p$. Then the composition $B \to \O_p \to A[a]$ induces a nonzero map on the zeroth cohomology sheaves, but this contradicts semiorthogonality. Therefore any object in ${}^\perp \mA$ is supported away from $p$.
\end{proof}

\begin{corollary}
  \label{cor: maps from torsion cannot be nonzero on cohomology}
  Let $X$ be a smooth algebraic variety such that the connected automorphism group $\mathrm{Aut}^\circ(X)$ acts transitively on $X$. Let $\mA \subset \Dbcoh(X)$ be an admissible subcategory, and let $A \in \mA$ be an object. If there exists a morphism $\O_p \to A[a]$ from a skyscraper sheaf at some point $p \in X$ to a shift of $A$ which induces a nonzero map on the zeroth cohomology sheaves, then $\mA = \Dbcoh(X)$.
\end{corollary}

\begin{proof}
  By Lemma~\ref{lem: maps from torsion cannot be nonzero on cohomology} any object of the orthogonal subcategory ${}^\perp \mA$ is not supported at $p$. For any element $g \in \mathrm{Aut}^\circ(X)$, the pullback $g^*(\O_p \to A[a])$ lets us conclude similarly that any object of ${}^\perp \mA$ is not supported anywhere along the orbit of $p$ under $\mathrm{Aut}^\circ(X)$. Therefore the orthogonal subcategory~${}^\perp \mA$ is zero, and $\mA = \Dbcoh(X)$.
\end{proof}

\subsection{Projections of skyscraper sheaves}
To study admissible subcategories, in this paper we often consider the projections of skyscraper sheaves into them. The following several lemmas prove some important properties of the projections of skyscrapers.

\begin{lemma}
  \label{lem: projection triangle induces nonzero map on cohomology}
  Let $X$ be a smooth algebraic variety such that the connected automorphism group~$\mathrm{Aut}^\circ(X)$ acts transitively on $X$. Let $\Dbcoh(X) = \langle \mA, \mB \rangle$ be a semiorthogonal decomposition. Consider a projection triangle for a skyscraper sheaf $\O_p$ at some point $p \in X$:
  \[
    B \to \O_p \to A \to B[1].
  \]
  If $\mB \neq 0$, then the morphism $\mH^0(B) \to \O_p$ is surjective.
\end{lemma}
\begin{proof}
  If the map $\mH^0(B) \to \O_p$ is not surjective, then it is zero, and by the long exact sequence of cohomology this would imply that $\O_p \to A$ induces a nonzero map on $\mH^0$. The result follows from Corollary~\ref{cor: maps from torsion cannot be nonzero on cohomology}.
\end{proof}

\begin{lemma}
  \label{lem: projections of a skyscraper}
  Let $X$ be a smooth algebraic variety such that the connected automorphism group~$\mathrm{Aut}^\circ(X)$ acts transitively on $X$. Let $\Dbcoh(X) = \langle \mA, \mB \rangle$ be a semiorthogonal decomposition. Consider a projection triangle for a skyscraper sheaf $\O_p$ at some point $p \in X$:
  \[
    B \to \O_p \to A \to B[1].
  \]
  \begin{enumerate}
    \item If $\dim X > 1$, at least one of $A$ and $B$ is not a locally free object at the point $p$.
    \item Both $A$ and $B$ are invariant under the action of $\Stab(p) \subset \mathrm{Aut}^\circ(X)$.
    \item If $B$ is set-theoretically supported at the point $p$, then $\mA = 0$ and $\mB = \Dbcoh(X)$.
  \end{enumerate}
\end{lemma}
\begin{proof}

  If $\mB = 0$, all properties are clear. So we assume that $\mB$ is a nonzero admissible subcategory.
  
  $(1)$: consider the fragment of the long exact sequence of cohomology sheaves associated to the projection triangle:
  \[
    0 \to \mH^{-1}(A) \to \mH^0(B) \to \O_p \to \mH^0(A).
  \]
  Since $\mB \neq 0$, by Lemma~\ref{lem: projection triangle induces nonzero map on cohomology} we see that the morphism $\mH^0(B) \to \O_p$ is surjective. Then the fragment above produces a short exact sequence
  \[
    0 \to \mH^{-1}(A) \to \mH^0(B) \to \O_p \to 0.
  \]
  If both $\mH^{-1}(A)$ and $\mH^0(B)$ are locally free at $p$, this produces a locally free resolution of $\O_p$ of length one, which is impossible by homological dimension reasons if $\dim X$ is greater than one. Thus at least one of those two cohomology sheaves is not locally free.

  $(2)$ is an immediate consequence of Lemma~\ref{lem: projections are g-invariant}.

  $(3)$ If $B$ is supported at $p$, then there is a nonzero morphism from a skyscraper sheaf $\O_p$ to the leftmost cohomology sheaf of $B$.  By Corollary~\ref{cor: maps from torsion cannot be nonzero on cohomology} this is equivalent to $\mB = \Dbcoh(X)$ and hence $\mA = 0$.
\end{proof}

\subsection{Fourier--Mukai transforms}

We recall some material about Fourier--Mukai transforms. For a more detailed exposition, see, for example, the book \cite[Ch.~5]{HuybFM}.

\begin{definition}
  \label{def: fourier-mukai transform}
  Let $X$ and $Y$ be two smooth and proper varieties. Let $\pi_X, \pi_Y$ be the projection maps from $X \times Y$ to $X$ and $Y$ respectively. Let $K \in \Dbcoh(X \times Y)$ be any object. Then the \emph{Fourier--Mukai transform with kernel $K$} is the functor $\Phi_K\colon \Dbcoh(X) \to \Dbcoh(Y)$ given by the formula $\Phi_K(-) := \pi_{Y *}(\pi_X^*(-) \otimes K)$.
\end{definition}

Most natural functors between derived categories of sheaves are Fourier--Mukai transforms. The identity functor on $\Dbcoh(X)$ is given by a Fourier--Mukai transform with respect to the structure sheaf of the diagonal $\O_{\Delta_X} \in \Dbcoh(X \times X)$. See \cite[Ex.~5.4]{HuybFM} for many other examples.

If $X$ and $Y$ are proper varieties, not necessarily smooth, then an object $K \in \Dbcoh(X \times Y)$ naturally defines a functor $\Phi_K\colon \Dperf(X) \to \Dbcoh(Y)$ using the same formula. This functor is also called a Fourier--Mukai transform with kernel $K$.

\begin{proposition}[{\cite[Prop.~5.9]{HuybFM}}]
  \label{prop: adjoints for fm transforms}
  Let $X$ and $Y$ be smooth and proper varieties. For any object $K \in \Dbcoh(X \times Y)$ the functor $\Phi_K$ has both left and right adjoint functors, and they are also Fourier--Mukai transforms.
\end{proposition}


\begin{proposition}[{\cite[Th.~7.1]{kuznetsov-basechange}}]
  \label{prop: fm kernels for projections}
  Let $X$ be a smooth and proper variety, and  let $\langle \mA, \mB \rangle$ be a semiorthogonal decomposition of $\Dbcoh(X)$. Then the projection functors $\mB_R$ and $\mA_L$ from Definition~\textup{\ref{lem: projection triangles exist}} are Fourier--Mukai transforms. The kernels of those functors, which we also denote by $\mB_R$ and $\mA_L$, fit into a distinguished triangle
  \[
    \mB_R \to \O_{\Delta_X} \to \mA_L
  \]
  of objects in $\Dbcoh(X \times X)$.
\end{proposition}

For future reference we record the following two results.

\begin{lemma}
  \label{lem: support of the kernel}
  Let $X$ be a smooth and proper variety, and let $\mA \subset \Dbcoh(X)$ be an admissible subcategory such that $\supp(\mA) = Z \subset X$ is a proper closed subset of $X$. Then the kernel $\mA_L$ for the (left) projection functor to $\mA$ is set-theoretically supported on $Z \times Z \subset X \times X$.
\end{lemma}
\begin{proof}
  Let $p \in X$ be any point. The projection $\mA_L(\O_p)$ of the skyscraper sheaf to $\mA$ is set-theoretically supported on the subset $Z \subset X$. By the definition of the Fourier--Mukai transforms this implies that the object $\mA_L$ is supported on the subset $X \times Z \subset X \times X$.

  Now let $p$ be a point in the open set $X \setminus Z$. Consider the projection triangle as in Lemma~\ref{lem: projection triangles exist}:
  \[
    A \to \O_p \to B
  \]
  where $A = \mA_L(\O_p) \in \mA$ and $B \in {}^\perp\mA$. Since $\supp(A)$ is a subset of $Z$, and $p$ does not lie in $Z$, there are no nonzero morphisms $A \to \O_p$. Thus the first arrow in the distinguished triangle above is zero, which implies that $B \iso \O_p \oplus A[1]$. Since $B$ and $A$ are semiorthogonal, this is only possible if the object $A = \mA_L(\O_p)$ is zero. By the definition of Fourier--Mukai transforms this implies that $\mA_L$ is supported on the closed subset $Z \times X \subset X \times X$.

  Thus $\mA_L$ is supported on the intersection of the two subsets mentioned above, i.e., on the subset $Z \times Z \subset X \times X$, as expected.
\end{proof}

\begin{remark}
  Lemma~\ref{lem: support of the kernel} may alternatively be deduced from a deeper statement about the structure of the Fourier--Mukai kernels for projection functors in \cite[Prop.~3.8]{kuznetsov-basechange}.
\end{remark}

\begin{lemma}
  \label{lem: fourier-mukai restriction}
  Let $X$ be a smooth variety, and let $f\colon Y \to X$ be a proper morphism. Let $\Dbcoh(X) = \langle \mA, \mB \rangle$ be a semiorthogonal decomposition, with the right projection functor defined by the Fourier--Mukai kernel $\mB_R \in \Dbcoh(X \times X)$. Then the Fourier--Mukai transform along the object $(f, f)^*\mB_R \in \Dbcoh(Y \times Y)$ is the functor $f^* \circ \mB_R \circ f_*\colon \Dbcoh(Y) \to \Dbcoh(Y)$. 
\end{lemma}

\begin{proof}
  Consider the commutative diagram:
  \[
    \begin{tikzcd}
      & & Y \times Y \arrow[ld] \arrow[rd]  & & \\
      &[4mm] Y \times X \arrow[ld, swap, "\pi_1"] \arrow[r, ->] &[0mm] X \times X \arrow[ld, swap, "\pi_1"] \arrow[rd, "\pi_2"] &[0mm] X \times Y \arrow[l, ->] \arrow[rd, "\pi_2"] &[4mm] \\
      Y \arrow[r, ->] & X & & X & Y \arrow[l, ->]
    \end{tikzcd}
  \]
  All three commutative squares in this diagram are Cartesian and are easily seen to be Tor-independent (\cite[Tag~08IA]{stacks-project}). The claimed formula follows by diagram chasing using the projection formula and the base change theorem for Tor-independent squares (see, for example, \cite[Tag~08IB]{stacks-project}).
\end{proof}


\section{Semiorthogonal decompositions and anticanonical divisors}
\label{sec: numerical lemmas}

Many standard examples of semiorthogonal decompositions arise from exceptional objects. An important tool helpful for studying exceptional objects is a restriction to an anticanonical divisor. It has been used, for example, in \cite{zube}, to prove the stability of exceptional vector bundles on $\P^3$. Given an arbitrary semiorthogonal decomposition which does not arise from an exceptional collection, it is more difficult to apply this approach. It is not even clear what exactly should we restrict to the divisor.

A natural generalization of this method to arbitrary admissible subcategories has been discovered by Nicolas Addington \cite[Prop.~2.1]{addington}. We recall his result in Theorem~\ref{thm: admissible subcategories produce autoequivalences}. Roughly speaking, he shows that any admissible subcategory induces an autoequivalence on each anticanonical divisor. This is especially strong in the case of surfaces, since autoequivalences of curves are well-understood. We give explicit corollaries for surfaces in Subsection~\ref{ssec: consequences on surfaces}.

\subsection{Admissible subcategories and autoequivalences}

The following theorem is essentially due to Addington. Note that for us an anticanonical divisor $D$ in a smooth and proper variety $X$ is a subscheme cut out by any nonzero section of $K_X^\dual$, not the reduced scheme structure on that subvariety. Recall that if $D$ is not smooth, then the Fourier--Mukai transform with respect to any object $K \in \Dbcoh(D \times D)$ still makes sense as a functor $\Dperf(D) \to \Dbcoh(D)$.

\begin{theorem}[{\cite[Prop.~2.1]{addington}}]
  \label{thm: admissible subcategories produce autoequivalences}
  Let $X$ be a smooth proper variety. Let $\mB \subset \Dbcoh(X)$ be an admissible subcategory, and let $\mB_R \in \Dperf(X \times X)$ be a Fourier--Mukai kernel for the right projection functor to $\mB$, equipped with the morphism $\varphi_\mB\colon \mB_R \to \O_{\Delta_X}$ as in Proposition~\textup{\ref{prop: fm kernels for projections}}.

  Let $j\colon D \monoarrow X$ be an inclusion morphism of an anticanonical divisor. Consider the composition of the restricted morphism $\varphi_\mB|_{D \times D}\colon \mB_R|_{D \times D} \to \O_{\Delta_X}|_{D \times D}$ with the tautological map $\O_{\Delta_X}|_{D \times D} \to \O_{\Delta_D}$. Take the cone of this composition to obtain a distinguished triangle in $\Dbcoh(D \times D)$:
  \begin{equation}
    \label{eqn: twist triangle}
    \mB_R|_{D \times D} \to \O_{\Delta_D} \to T.
  \end{equation}
  Then the Fourier--Mukai transform with respect to the object $T \in \Dbcoh(D \times D)$ is a functor $\Dperf(D) \to \Dbcoh(D)$ whose image lies inside the subcategory $\Dperf(D) \subset \Dbcoh(D)$ and it induces an autoequivalence of the category $\Dperf(D)$.
\end{theorem}

To deduce this from Addington's paper, we need to recall a notion of a spherical functor.

\begin{definition}
A functor $F\colon T_1 \to T_2$ between two triangulated categories which admits a left adjoint $L$ and a right adjoint $R$ is called \emph{spherical} if the endofunctors obtained as cones of the unit natural transformations $\mathrm{Id}_{T_1} \Rightarrow R \circ F$ and $\mathrm{Id}_{T_2} \Rightarrow F \circ L$ are both autoequivalences, of $T_1$ and $T_2$, respectively.
\end{definition}

Of course, this definition only makes sense in settings where taking a cone of a natural transformation is a meaningful operation. This is not really possible in the realm of triangulated subcategories, and we need either dg-enhancements or stable $(\infty, 1)$-categories to make this into a rigorous definition (\cite{anno-logvinenko}; however, see \cite[Def.~2.8]{kuznetsov-calabi} for an alternative approach). There exist several other equivalent definitions.

\begin{numberedexample}
  \label{ex: spherical restriction}
  Let $X$ be a smooth and proper variety, and let $j\colon D \monoarrow X$ be an embedding of any divisor. Then the restriction functor $j^*\colon \Dperf(X) \to \Dperf(D)$ is a spherical functor. This is well-known, see, e.g., \cite[2.2 (4)]{addington} for details, but we sketch the argument here for completeness.

  To make sense of this statement we need to know that the restriction functor $j^*$ has both a left and a right adjoint, find a way to construct cones of the unit natural transformations, and check that those cones are autoequivalences. The right adjoint functor is the pushforward $j_*$, the left adjoint functor to $j^*$ is the functor
  \[
    j_!(-) := j_*( (-) \otimes \omega_D \otimes j^*\omega_X^\dual[-1])
  \]
  where $\omega_X$ and $\omega_D$ are dualizing line bundles, and they exist since $X$ is smooth and $D$ is a locally complete intersection in $X$, in particular Gorenstein. Compare the construction of the right adjoint functor to the pushforward in \cite[Th.~3.34]{HuybFM}. We can take cones of unit natural transformations by representing the functors via their Fourier--Mukai kernels in $\Dbcoh(X \times X)$ and $\Dbcoh(D \times D)$. Note that apriori an object in $\Dbcoh(D \times D)$ only defines a functor $\Dperf(D) \to \Dbcoh(D)$, but it may be checked that in our case the "twist" object defines a functor which sends perfect objects to perfect objects.

  Finally, the fact that the cones are autoequivalences follows essentially from the fact that for any $F \in \Dperf(X)$ there exists a functorial distinguished triangle
  \[
    F(-D) \to F \to j_* j^*(F) \to F(-D)[1]
  \]
  and here the functor $F \mapsto F(-D)$ ("spherical cotwist") is an autoequivalence. Thus the restriction $j^*$ is indeed a spherical functor.
\end{numberedexample}

Addington proved the following result. It involves the notion of the Serre functor of a triangulated category, introduced in \cite{bond-kapr}. We will confine ourselves to mentioning that for a smooth and proper variety $X$, the Serre functor of the category $\Dbcoh(X)$ is $(-) \otimes \omega_X[\dim X]$, where $\omega_X$ is the canonical line bundle.

\begin{theorem}[{\cite[Prop.~2.1]{addington}}]
  \label{thm: addington main theorem}
  Let $F\colon T_1 \to T_2$ be a spherical functor such that the spherical cotwist is, up to a shift, isomorphic to the Serre functor of $T_1$. If $\mA \subset T_1$ is an admissible subcategory, then the composition $\mA \monoarrow T_1 \xrightarrow{F} T_2$ is also a spherical functor.
\end{theorem}

We can now deduce Theorem~\ref{thm: admissible subcategories produce autoequivalences} from the result above.

\begin{proof}[Proof of Theorem~\textup{\ref{thm: admissible subcategories produce autoequivalences}}]
  Recall from Example~\ref{ex: spherical restriction} that the restriction functor
  \[
    j^*\colon \Dperf(X) \to \Dperf(D)
  \]
  to any divisor is a spherical functor, and the induced autoequivalence on $\Dperf(X)$ is the functor $F \mapsto F(-D)$. Since $D$ is an anticanonical divisor, this is a shift of the Serre functor.

  Then Theorem~\ref{thm: addington main theorem} shows that the composition
  \[
    \mB \monoarrow \Dbcoh(X) = \Dperf(X) \xrightarrow{j^*} \Dperf(D)
  \]
  is also a spherical functor. Compare to \cite[2.2 ($4^\prime$)]{addington}, where the same argument is used for an admissible subcategory spanned by a single exceptional object. Spherical functors are associated with many endofunctors, and one may check that a so-called \emph{spherical twist} in this context is exactly a Fourier--Mukai transform along the object $T \in \Dbcoh(D \times D)$. Spherical twists are always autoequivalences \cite[Th.~2.3]{addington}.
\end{proof}

\begin{remark}
  As is clear from the argument, Addington's result is not specifically about anticanonical divisors. Some other spherical functors may be used instead, like a pullback to a two-sheeted covering of $X$ branched in $|-2K_X|$, or an embedding of $X$ into the total space of the canonical bundle $X \monoarrow \Tot(K_X)$. The advantage of the anticanonical divisors is that they have smaller dimension than $X$, and this is very important for the study of surfaces in this paper.
\end{remark}

\begin{corollary}
  \label{cor: autoequivalence triangle}
  In the notation of Theorem~\textup{\ref{thm: admissible subcategories produce autoequivalences}}, for any object $F \in \Dperf(D)$ there exists a distinguished triangle
  \(
    j^* \mB_R(j_*F) \to F \to T(F)
  \) in $\Dperf(D)$.
\end{corollary}

\begin{proof}
  The only thing to check is that the Fourier--Mukai transform of an object $F$ with respect to the kernel $\mB_R|_{D \times D}$ is isomorphic to $j^* \mB_R(j_*F)$, but this is true by Lemma~\ref{lem: fourier-mukai restriction}.
\end{proof}

\subsection{Consequences on surfaces}
\label{ssec: consequences on surfaces}

When the ambient variety is a surface, we can deduce from Theorem~\ref{thm: admissible subcategories produce autoequivalences} a strong structural result that lets us control the behavior of arbitrary admissible subcategories. The following proposition is used multiple times in the proofs of the major results of this paper: the classification of admissible subcategories of $\P^2$ in Section~\ref{sec: projective plane}, and the study of admissible subcategories in rational elliptic surfaces and del Pezzo surfaces in Section~\ref{sec: del pezzos}.

\begin{proposition}
  \label{prop: numerical lemma without numbers}
  Let $S$ be a smooth proper surface, let $j\colon E \monoarrow S$ be a reduced and irreducible anticanonical divisor, and let $p \in E$ be a smooth point of $E$. Let $\mB \subset \Dbcoh(S)$ be an admissible subcategory. Denote by $B := \mB_R(\O_p)$ the (right) projection of a skyscraper sheaf $\O_p$ to the subcategory $\mB$. Then the object $j^*B$ is isomorphic to one of the following options:
  
  \begin{enumerate}
  \item $j^*B = 0$;
  \item $j^*B \iso \O_p[0] \oplus \O_q[a]$ for a smooth point $q \in E$ which may coincide with $p$, and some shift $a \in \Z$;
  \item $j^*B \iso \O_{2p}[0]$, where $\O_{2p} \in \mathrm{Coh}(E)$ is a quotient of $\O_E$ by the square of the maximal ideal of the point $p \in E$;
  \item $j^*B \iso \O_p[0] \oplus M[a]$ for some simple vector bundle $M$ on $E$ and some shift $a \in \Z$;
  \item $j^*B \iso \widetilde{M}[0]$, where $\widetilde{M}$ is a vector bundle on $E$ which fits into a short exact sequence
    \[
      0 \to M \to \widetilde{M} \to \O_p \to 0
    \]
    where $M$ is a simple vector bundle on $E$.
  \end{enumerate}
\end{proposition}

\begin{proof}
  The object $B$ is by definition isomorphic to $\mB_R(j_*\O_p)$. Since $p \in E$ is a smooth point, the skyscraper sheaf $\O_p \in \Dbcoh(E)$ is a perfect object. By Corollary~\ref{cor: autoequivalence triangle} the derived pullback $j^*B \in \Dperf(E)$ fits into a triangle
  \begin{equation}
    \label{eqn: skyscraper-like objects}
    j^*B \to \O_p \to C \to j^*B [1]
  \end{equation}
  where $C := T(\O_p) \in \Dperf(E)$ is some object. By Theorem~\ref{thm: admissible subcategories produce autoequivalences} the functor $T$ is an autoequivalence, thus
  \[
    \REnd_E(C) \iso \REnd_E(\O_p) \iso \k[0] \oplus \k[-1].
  \]

  For the duration of this proof, we use the term \emph{skyscraper-like object} to mean an object in the category $\Dperf(E)$ with this $\Ext$-algebra. Since $E$ is reduced and irreducible, we can use the classification of skyscraper-like perfect objects given in \cite[Prop.~4.13]{burban-kreussler}. It shows that the object $C$ is, up to a shift, either a skyscraper sheaf on some smooth point $q \in E$, or a simple vector bundle $M$ on $E$.
  

  Assume first that $C$ is a shift of a skyscraper sheaf on a smooth point $q \in E$. If $q$ is not the same point as $p$, then the map $\O_p \to \O_q[a]$ from (\ref{eqn: skyscraper-like objects}) is necessarily zero, and hence we have an isomorphism $j^*B \iso \O_p[0] \oplus \O_q[a-1]$. If $q = p$ and the map $\O_p \to \O_p[a]$ from (\ref{eqn: skyscraper-like objects}) is nonzero, there are two cases. Either $a = 0$ and the map is an isomorphism, and then the cone $j^*B$ is zero. Or $a = 1$ and the map is a nonzero element of $\Ext_E^1(\O_p, \O_p) \iso \k$, in which case the object $j^*B$ is isomorphic to the unique nontrivial extension of a skyscraper sheaf on a smooth point by itself, i.e., $j^*B \iso \O_{2p}[0]$.

  The other option given in \cite[Prop.~4.13]{burban-kreussler} for a skyscraper-like object $C$ is a shift of a simple vector bundle. Assume that $C \iso M[a]$ for some $a \in \Z$ and some simple vector bundle~$M$. If the morphism $\O_p \to M[a]$ in (\ref{eqn: skyscraper-like objects}) is zero, then $j^*B \iso \O_p[0] \oplus M[a-1]$. Since the point $p$ is smooth in $E$, $\Ext_E^a(\O_p, M)$ is nontrivial only when $a = 1$, so any nonzero map in the triangle (\ref{eqn: skyscraper-like objects}) arises from some short exact sequence
  \[
    0 \to M \to \widetilde{M} \to \O_p \to 0,
  \]
  and for those maps in $\Ext^1_E(\O_p, M)$ we have an isomorphism $j^*B \iso \widetilde{M}[0]$ in (\ref{eqn: skyscraper-like objects}). Thus the list of possible isomorphism classes of $j^*B$ in the statement is exhaustive.
\end{proof}

\begin{remark}
  If $E$ is a smooth anticanonical divisor, the classification of skyscraper-like perfect objects can be easily deduced from Lemmas~\ref{lem: objects on smooth curves split} and~\ref{lem: smooth torsion sheaves on curves}. Moreover, if we assume that the support of the object $j^*B$ is in the smooth locus of $E$, then the conclusion of Proposition~\ref{prop: numerical lemma without numbers} holds when $j^*B$ is a torsion object, even in the case where if $E$ is a reducible curve.
\end{remark}

The description of $j^*B$ in the proposition above implies an interesting property for restrictions of the object $B$ to general anticanonical divisors. Suppose that we are in a situation where $j^*B$ is isomorphic to a direct sum of two skyscraper sheaves. Consider a different anticanonical divisor, $j^\prime\colon E^\prime \monoarrow S$, which does not necessarily pass through the point $p \in S$. If $E^\prime$ is in some sense "close" to the divisor $E$, it is reasonable to expect that $j^{\prime *}B$ is also a torsion object, and by semicontinuity it should not be significantly more complicated than two skyscrapers. This imprecise intuitive picture may be improved to a rigorous statement. In fact, this holds for an arbitrary anticanonical divisor as soon as the restriction is a torsion object.

\begin{lemma}
  \label{lem: arbitrary anticanonical divisors}
  Let $S$, $j\colon E \monoarrow S$, $p \in E$ and $B \in \mB \subset \Dbcoh(S)$ be as in Proposition~\textup{\ref{prop: numerical lemma without numbers}}. Assume that the support of $j^*B$ consists of two distinct smooth points of $E$. Let~$j^\prime\colon E^\prime \monoarrow S$ be another anticanonical divisor, not necessarily passing through the point $p$.  Suppose that $j^{\prime *}B$ is a torsion object supported on a smooth part of the curve $E^\prime$. Then $j^{\prime *}B$ is isomorphic to one of the following options:
  \begin{enumerate}
    \item $j^{\prime *}B = 0$;
    \item $j^{\prime *}B \iso \O_q[a]$ for some point $q \in E^\prime$ and a shift $a \in \Z$; 
    \item $j^{\prime *}B \iso \O_q[a] \oplus \O_r[b]$ for some points $q, r \in E^\prime$ and shifts $a, b \in \Z$;
    \item $j^{\prime *}B \iso \O_{2q}[a]$ for some point $q \in E^\prime$ and a shift $a \in \Z$, where $\O_{2q}$ is the quotient of the structure sheaf $\O_E$ by the square of the maximal ideal of the point $q \in E$.
  \end{enumerate}
\end{lemma}

\begin{proof}
  Consider a restriction triangle for the object $B \in \Dbcoh(S)$ to the divisor $E \subset S$:
  \begin{equation}
    \label{eqn: triangle of restrictions}
    B \otimes K_S \to B \to j_*j^*B
  \end{equation}
  An application of the functor $\RHom(B, -)$ produces a triangle of graded vector spaces
  \begin{equation}
    \label{eqn: triangle of endomorphisms}
    \RHom(B, B \otimes K_S) \to \RHom(B, B) \to \REnd(j^*B).
  \end{equation}
  From Proposition~\ref{prop: numerical lemma without numbers} we know that $j^*B$ is isomorphic to a direct sum of two distinct skyscrapers. Then the length of the graded vector space $\REnd(j^*B)$ is equal to four.

  Since the object $B$ comes equipped with the natural morphism $B \to \O_p$, the first morphism in the triangle (\ref{eqn: triangle of restrictions}) may be extended into the commutative square
  \begin{equation}
    \label{eqn: vanishing of an arrow}
    \begin{tikzcd}
      B \otimes K_S \arrow[r] \arrow[d] & B \arrow[d] \\
      \O_p \otimes K_S \arrow[r] & \O_p
    \end{tikzcd}
  \end{equation}
  Note that the bottom arrow is zero since $p$ lies on the divisor $E$. An application of the functor $\RHom(B, -)$ produces a commutative square
  \[
    \begin{tikzcd}
      \RHom(B, B \otimes K_S) \arrow[r] \arrow[d] & \RHom(B, B) \arrow[d] \\
      \RHom(B, \O_p \otimes K_S) \arrow[r] & \RHom(B, \O_p)
    \end{tikzcd}
  \]
  Here the right vertical map is an isomorphism by the definition of the projection functor, and the bottom horizontal map is zero. Therefore the upper horizontal map is also zero.

  Thus the first arrow in the triangle (\ref{eqn: triangle of endomorphisms}) is zero, so $\REnd(j^*B)$ is isomorphic to a direct sum of the other two terms. Therefore we get
  \[
    \mylength(\RHom(B, B \otimes K_S)) + \mylength(\RHom(B, B)) = 4.
  \]

  By a similar procedure we obtain a triangle of graded vector spaces corresponding to the restriction to the divisor $j^\prime\colon E^\prime \monoarrow S$:
  \[
    \RHom(B, B \otimes K_S) \to \RHom(B, B) \to \REnd(j^{\prime *}B).
  \]
  The length of the cone is bounded from above by the sum of lengths of the first two terms. This produces an inequality:
  \[
    \mylength(\REnd(j^{\prime *}B)) \leq \mylength(\RHom(B, B \otimes K_S)) + \mylength(\RHom(B, B)) = 4.
  \]

  Using Lemmas~\ref{lem: objects on smooth curves split} and~\ref{lem: smooth torsion sheaves on curves} it is easy to see that a torsion object $j^{\prime *}B$ on a smooth locus of the curve $E^\prime$ with $\mylength(\REnd(j^{\prime *}B)) \leq 4$ is isomorphic to one of the four options listed above.
\end{proof}

\begin{remark}
  Using the commutative square (\ref{eqn: vanishing of an arrow}) to show the vanishing of the first arrow in the triangle (\ref{eqn: triangle of endomorphisms}) is an interesting trick that may be built upon to produce an alternative proof of Theorem~\ref{thm: admissible subcategories produce autoequivalences} which does not refer to spherical functors and their properties. However, that proof is longer, more complicated, and less versatile. Interested readers may read it in the thesis version of this paper \cite[Sec.~3]{mythesis}.
\end{remark}


\section{Classification of admissible subcategories of $\P^2$}
\label{sec: projective plane}

\subsection{Overview}

The goal of this chapter is to prove the following result about admissible subcategories in the derived category $\Dbcoh(\P^2)$ of coherent sheaves on $\P^2$. Since exceptional objects and exceptional collections in $\Dbcoh(\P^2)$ have been classified in \cite{gorodentsev-rudakov}, this theorem may be described as a classification of admissible subcategories.

\statemaintheoremPlane

This classification immediately implies the following.


\begin{restatable}{corollary}{statemaintheoremPlanePhantoms}
  \label{cor: no phantoms in the plane}
  There are no phantom subcategories in $\Dbcoh(\P^2)$.
\end{restatable}
 
\begin{proof}
  For any category $\mA$ with a full exceptional collection of length $n$ the Grothendieck group $K_0(\mA)$ is a free abelian group on $n$ generators. Thus by Theorem~\ref{thm: all admissible subcategories are standard} an admissible subcategory of $\Dbcoh(\P^2)$ is either a zero category, or has non-vanishing $K_0$.
\end{proof}

As mentioned in Lemma~\ref{lem: admissible subcategories are semiorthogonal decompositions}, any admissible subcategory $\mA \subset \Dbcoh(\P^2)$ leads to a semiorthogonal decomposition of that category, $\Dbcoh(\P^2) = \langle \mA, {}^\perp\mA \rangle$. Since the length of any full exceptional collection in $\Dbcoh(\P^2)$ is three, the result above implies that in any nontrivial decomposition at least one of the subcategories $\mA$ and ${}^\perp \mA$ is generated by a single exceptional object. In fact, in the proof of Theorem~\ref{thm: all admissible subcategories are standard} we do not construct nontrivial exceptional collections directly, but rather recognize which of the subcategories $\mA$ and ${}^\perp \mA$ is a simpler one. More precisely, Theorem~\ref{thm: all admissible subcategories are standard} is implied by the following statement:

\begin{theorem}
  \label{thm: reduction to B not locally free}
  Let $\Dbcoh(\P^2) = \langle \mA, \mB \rangle$ be a semiorthogonal decomposition such that $\mA \neq 0$ and~$\mB \neq 0$. Pick a point $p \in \P^2$. Consider a projection triangle for the skyscraper sheaf $\O_p$:
  \[
    B \to \O_p \to A \to B[1]
  \]
  with $B \caniso \mB_R(\O_p) \in \mB$ and $A \caniso \mA_L(\O_p) \in \mA$. Assume that $B$ is not locally free at $p$. Then the subcategory $\mA \subset \Dbcoh(\P^2)$ is generated by a single exceptional vector bundle.
\end{theorem}

The strategy of the proof of Theorem~\ref{thm: reduction to B not locally free} is discussed in Subsection~\ref{ssec: overview of the proof}. First we show how to deduce Theorem~\ref{thm: all admissible subcategories are standard} from this statement.

\begin{proof}[{Proof of the implication \textup{(\ref{thm: reduction to B not locally free})} $\implies$ \textup{(\ref{thm: all admissible subcategories are standard})}}]
  Let $\mA \subset \Dbcoh(\P^2)$ be an arbitrary admissible subcategory. Denote the orthogonal subcategory ${}^\perp\mA \subset \Dbcoh(\P^2)$ by $\mB$. Then $\Dbcoh(\P^2) = \langle \mA, \mB \rangle$ is a semiorthogonal decomposition. If either $\mA$ or $\mB$ is a zero subcategory, there is nothing to prove, so we assume that the decomposition is nontrivial. Let $p \in \P^2$ be a point. Consider a projection triangle
  \[
    B \to \O_p \to A \to B[1]
  \]
  of the skyscraper sheaf. By parts $(1)$ and $(2)$ of Lemma~\ref{lem: projections of a skyscraper} we know that at least one of projections $A$ and $B$ is not locally free at $p$.

  Suppose $B$ is not locally free. Then Theorem~\ref{thm: reduction to B not locally free} implies that there is an exceptional vector bundle $E \in \Dbcoh(\P^2)$ such that $\mA = \langle E \rangle$. By \cite[Th.~5.10]{gorodentsev-rudakov} there exists a mutation of the standard exceptional collection $\Dbcoh(\P^2) = \langle \O, \O(1), \O(2) \rangle$ which contains $E$, confirming Theorem~\ref{thm: all admissible subcategories are standard} in this case.

  Suppose now that $A$ is not locally free. Observe that the dualized and shifted triangle
  \[
    A^\dual[2] \to \O_p \to B^\dual[2] \to A^\dual[3]
  \]
  is the projection triangle of the skyscraper $\O_p$ corresponding to the dual semiorthogonal decomposition $\Dbcoh(\P^2) = \langle \mB^\dual, \mA^\dual \rangle$. Note that $A$ is locally free if and only if $A^\dual[2]$ is. By the same argument as above we see that $\mB^\dual = \langle E \rangle$ for some exceptional bundle $E \in \Dbcoh(\P^2)$. This implies that $\mB$ is generated by a single exceptional bundle $E^\dual$.

  By \cite[Th.~5.10]{gorodentsev-rudakov} any exceptional vector bundle $E^\dual$ on $\P^2$ is a member of some full exceptional collection $\langle E^\prime, E^{\prime \prime}, E^\dual \rangle$, which is moreover a mutation of the standard exceptional collection on $\P^2$. Therefore the category $\mA = \mB^\perp$ is equal to the subcategory $\langle E^\prime, E^{\prime \prime} \rangle$. Thus Theorem~\ref{thm: all admissible subcategories are standard} holds in this case as well.
\end{proof}

\subsection{Strategy of the proof}
\label{ssec: overview of the proof}

The proof of Theorem~\ref{thm: reduction to B not locally free} relies on the properties of the restriction of $B$ to a cubic curve passing through the point $p \in \P^2$. The proof is split into several parts. First in Subsection~\ref{ssec: restrictions to cubic curve satisfy the conditions} we use the results from Section~\ref{sec: numerical lemmas} to study the restriction of the object $B$ to a cubic curve, i.e., to an anticanonical divisor. We use the classification from Proposition~\ref{prop: numerical lemma without numbers} to deduce strong constraints on the object $B$ itself. For instance, in that subsection we show that $B$ is concentrated in at most two cohomology degrees. Then in Subsection~\ref{ssec: new reduction steps} we prove that the zeroth cohomology sheaf of $B$ is a skyscraper sheaf $\O_p$ and the minus first cohomology sheaf is locally free. Finally, in Subsection~\ref{ssec: exceptional bundles from local freeness} we conclude that $A$ is a direct sum of several copies of a single exceptional vector bundle, which lets us finish the proof by Lemma~\ref{lem: projection is rigid implies exceptional vector bundle}.

\subsection{Restricting projections of a skyscraper to a cubic curve}
\label{ssec: restrictions to cubic curve satisfy the conditions}

\begin{setting}
  \label{set: b-projector}
  From here on we fix the data involved in Theorem~\ref{thm: reduction to B not locally free}, namely a semiorthogonal decomposition $\Dbcoh(\P^2) = \langle \mA, \mB \rangle$ such that $\mA \neq 0$ and~$\mB \neq 0$, a point $p \in \P^2$, and the projection triangle for the skyscraper sheaf
  \[
    B \to \O_p \to A \to B[1]
  \]
  with $B \in \mB$ and $A \in \mA$, such that $B$ is not locally free at $p$. We also fix a smooth cubic curve $j\colon E \to \P^2$ cut out by an equation $s \in \Gamma(\P^2, \O(3))$ which passes through $p$.
\end{setting}

\begin{remark}
  In our approach to the proof of Theorem~\ref{thm: reduction to B not locally free} we often use the fact that $\PGL(3)$, the automorphism group of $\P^2$, acts doubly transitively on $\P^2$. For example, this implies that the stabilizer subgroup $\Stab(p) \subset \PGL(3)$ of the point $p \in \P^2$, which acts on the projections of the skyscraper sheaf by Lemma~\ref{lem: projections of a skyscraper} (2), has only two orbits in $\P^2$. It is possible to avoid most instances of relying on symmetry by using general cubic curves instead of fixing the curve $E$ in Setting~\ref{set: b-projector}. We use a strategy like that in some parts of Section~\ref{sec: del pezzos}, where we deal with del Pezzo surfaces. However, for Theorem~\ref{thm: reduction to B not locally free} we need some global geometric properties of $\P^2$ in any case, so there is no immediate benefit from circumventing the arguments based on symmetry.
\end{remark}

\begin{lemma}
  \label{lem: nontrivial decomposition is not pointwise}
  Let $B$ be as in Setting~\textup{\ref{set: b-projector}}. Then the support $\supp(B)$ is $\P^2$.
\end{lemma}
\begin{proof}
  The object $B$ is invariant under the action of the group $\Stab(p) \subset \PGL(3)$ by Lemma~\ref{lem: projections of a skyscraper}, so $\supp(B)$ is a closed $\Stab(p)$-invariant subset of $\P^2$. Thus it is either $\P^2$, or a point $p$.

  Assume that $B$ is an object set-theoretically supported only at the point $p \in \P^2$. Pick the smallest integer $i \in \Z$ such that $\mH^i(B) \neq 0$. Then there exists a morphism $\mH^i(B)[-i] \to B$ in the derived category inducing the identity map on the $i$'th cohomology sheaves. Since $\mH^i(B)$ is a nonzero torsion sheaf supported at a point $p$, there exists a inclusion $\O_p \monoarrow \mH^i(B)$ of sheaves. The composition $\O_p[-i] \to \mH^i(B)[-i] \to B$ is a map inducing a nonzero morphism on the $i$'th cohomology sheaves, so by Corollary~\ref{cor: maps from torsion cannot be nonzero on cohomology} this implies that $\mB = \Dbcoh(\P^2)$ and~$\mA = 0$. This is a contradiction with the assumption that $\mA \neq 0$.
\end{proof}

\begin{lemma}
  \label{lem: p2 numerical lemma}
  Let $B$ be as in Setting~\textup{\ref{set: b-projector}}. For any smooth cubic curve $j\colon E \to \P^2$ which passes through $p$, the derived restriction $j^*B$ is isomorphic to $\O_p[0] \oplus M[a]$ for some simple vector bundle $M$ on the curve $E$ and some shift $a \in \Z$.
\end{lemma}

\begin{proof}
  Note that we are exactly in the situation of Proposition~\ref{prop: numerical lemma without numbers}: we restrict a projection of a skyscraper to a smooth anticanonical divisor on a surface. It only remains to rule out all options except $\O_p[0] \oplus M[a]$.

  The object $B$ is $\Stab(p)$-invariant by Lemma~\ref{lem: projections of a skyscraper}. There are only two orbits of $\Stab(p)$ on $\P^2$, the point $p$ and the complement $\P^2 \setminus \{ p \}$. Thus if $B$ is not locally free at $p$, by Lemma~\ref{lem: locally free has many definitions} the length of the derived fiber at~$p$ is strictly larger than at any other point of $\P^2$. This implies that the restriction $j^*B$ to $E$ is also not locally free at $p \in E$ since the (derived) restriction does not change the lengths of derived fibers. By Lemma~\ref{lem: nontrivial decomposition is not pointwise} the support of $j^*B$ is the curve $E$, so the pullback $j^*B$ is not a torsion object. Among the options listed in Proposition~\ref{prop: numerical lemma without numbers}, only one is an object which is not torsion and not locally free at~$p$, and therefore $j^*B \iso \O_p[0] \oplus M[a]$ for a simple vector bundle $M$ on $E$, as claimed in the statement.
\end{proof}

\begin{lemma}
  \label{lem: cohomology sheaves do not disappear after a restriction}
  Let $B$ and $j\colon E \to \P^2$ be as in Setting~\textup{\ref{set: b-projector}}. If $\mathcal{H}^i(j^*B) = 0$, then $\mathcal{H}^i(B) = 0$.
\end{lemma}

\begin{proof}
  Since $j\colon E \to \P^2$ is an inclusion of a (Cartier) divisor, by Lemma~\ref{lem: derived restriction to a divisor} the vanishing of~$\mH^i(j^*B)$ implies that $\supp(\mH^{i}(B)) \cap E = \emptyset$. By Lemma~\ref{lem: projections of a skyscraper} the object $B$ is $\Stab(p)$-invariant, hence $\mH^i(B)$ is also $\Stab(p)$-invariant. Since $E$ passes through $p$ and $\Stab(p)$ acts transitively on $\P^2 \setminus \{ p \}$, we obtain that the nonderived restriction of $\mH^i(B)$ to any point of~$\P^2$ is zero, but this implies $\mH^i(B) = 0$.
\end{proof}

\begin{corollary}
  \label{cor: reduction to at most two cohomology sheaves}
  Let $B$ be as in Setting~\textup{\ref{set: b-projector}}. Then $B$ has at most two nonzero cohomology sheaves, and at most one of them is not a torsion sheaf supported at $p$.
\end{corollary}
\begin{proof}
  Pick an elliptic curve $j\colon E \to \P^2$ which passes through $p$. Then Lemmas~\ref{lem: p2 numerical lemma} and~\ref{lem: cohomology sheaves do not disappear after a restriction} imply that $B$ has at most two nonzero cohomology sheaves. Moreover, we see that the (derived) restriction of $B$ to some point $q \in E$ which is distinct from $p$ is concentrated in a single degree. Since $B$ is $\Stab(p)$-invariant, it is locally free away from $p$ and thus only one of cohomology sheaves is nonzero around the point $q$.
\end{proof}

\subsection{The structure of $B$}
\label{ssec: new reduction steps}

\begin{lemma}
  \label{lem: lines detect point-torsion}
  Let $\mathcal{F}$ be a nonzero coherent sheaf on a smooth surface $S$ supported at a single point $p \in S$. Then for any curve $j\colon C \monoarrow S$ passing through~$p$ we have~$L_1j^*\mathcal{F} \neq 0$ and~$L_0j^*\mathcal{F} \neq 0$. Moreover, those two zero-dimensional sheaves have the same length.
\end{lemma}
\begin{proof}
  We may work locally and assume that $S$ is a spectrum of a local ring. Let $\mathfrak{m} \subset \O_S$ be the ideal sheaf of the point $p$. The curve $C$ is given by $f = 0$ for some $f \in \mathfrak{m}$. The derived pullback $j^*\mathcal{F}$ is computed by the complex $\mathcal{F} \xrightarrow{f \cdot -} \mathcal{F}$.
  Since $\mathcal{F}$ is set-theoretically supported at the point $p$, for some $n \gg 0$ we have $f^n \in \mathrm{Ann}(\mathcal{F})$. The multiplication by $f$ thus cannot be an automorphism of $\mathcal{F}$. Since $\mathcal{F}$ is a vector space of finite dimension, this means the kernel and cokernel of the multiplication map are both nonzero and have the same dimension.
\end{proof}

\begin{lemma}
  \label{lem: lines detect skyscrapers}
  Let $\mathcal{F}$ be a nonzero coherent sheaf on a smooth surface $S$ supported at a single point $p \in S$. Assume that for any tangent direction at $p$ there exists a smooth curve $j\colon C \monoarrow S$ passing through $p$ with that tangent direction such that the torsion sheaf $L_0j^*\mathcal{F}$ has length one. Then $\mathcal{F}$ is isomorphic to a skyscraper sheaf $\O_p$ on $S$.
\end{lemma}
\begin{proof}
  Let $A := \O_{S, p}$ be the local ring of the point $p \in S$, and denote by $\mathfrak{m}$ the maximal ideal of $A$. Let $C \subset S$ be one of the curves from the statement, and let $f \in \mathfrak{m}$ be an equation of the curve $C$. Then the nonderived restriction $L_0j^*\mathcal{F}$ is isomorphic to $\mathcal{F}/f \mathcal{F}$. Note that the quotient $\mathcal{F}/\mathfrak{m}\mathcal{F}$ is nonzero since $\mathcal{F}$ is a nonzero sheaf. Since the length of $\mathcal{F}/f \mathcal{F}$ is one, this implies that $\mathcal{F}/\mathfrak{m}\mathcal{F}$ is an one-dimensional vector space. By Nakayama's lemma $\mathcal{F}$ is a cyclic module, i.e., $\mathcal{F} \iso A/I$ for some ideal $I \subset A$ contained in $\mathfrak{m}$.

  Let $I_p$ be the image of $I \subset \mathfrak{m}$ in the cotangent space $T_p^\dual := \mathfrak{m}/\mathfrak{m}^2$. If $I_p = T_p^\dual$, then by Nakayama's lemma $I = \mathfrak{m}$ and then $\mathcal{F} \iso A/\mathfrak{m} \iso \O_p$, so the lemma is proved. Assume now that $I_p$ is a proper subset of $T_p^\dual$. For an equation $f \in \mathfrak{m}$ of a curve $C$ as in the statement let $[ f ] \in T_p^\dual$ denote its class in $T_p^\dual$. If $I_p$ is a nonzero subspace, choose a curve $C = \{ f = 0 \}$ such that $[ f ] \in I_p$, and if $I_p$ is zero, choose an arbitrary $C$. The assumption on the length of $L_0j^*\mathcal{F}$ implies that $(I, f) = \mathfrak{m}$. But by the choice of $f$ the image of the ideal $(I, f)$ in the cotangent space $T_p^\dual$ is a proper subset of $T_p^\dual$, a contradiction. Thus $I = \mathfrak{m}$ is the only option.
\end{proof}

\begin{lemma}
  \label{lem: always torsion in cohomology sheaves of B}
  Let $B$ be as in Setting~\textup{\ref{set: b-projector}}. At least one cohomology sheaf $\mH^i(B)$ has torsion.
\end{lemma}

\begin{proof}
  Assume that all cohomology sheaves are torsion-free. By Corollary~\ref{cor: reduction to at most two cohomology sheaves} the object $B$ has only one nonzero cohomology sheaf. Moreover, by Lemma~\ref{lem: projection triangle induces nonzero map on cohomology} the sheaf $\mH^0(B)$ is not zero. Hence $B \iso \mathcal{F}[0]$ for some $\Stab(p)$-invariant torsion-free coherent sheaf $\mathcal{F}$ on $\P^2$. By Lemma~\ref{lem: torsion-free sheaves and divisors} the derived restriction $j^*\mathcal{F}$ is concentrated in degree zero. From Lemma~\ref{lem: p2 numerical lemma} we conclude that $L_0j^*\mathcal{F} \iso \O_p \oplus M$ for a vector bundle $M$ on the curve $E$, and $L_1j^*\mathcal{F} = 0$.
  
  Since $\mathcal{F}$ is a torsion-free sheaf on a surface, we may consider the short exact sequence from Lemma~\ref{lem: torsion-free sheaves on a surface}:
  \begin{equation}
    \label{eqn: torsion-free sheaf on a surface}
    0 \to \mathcal{F} \to \mathcal{E} \to \mathcal{Q} \to 0
  \end{equation}
  where $\mathcal{E}$ is locally free and $\mathcal{Q}$ is a torsion sheaf. By $\Stab(p)$-invariance of $\mathcal{F}$ and the uniqueness of the short exact sequence the torsion sheaf $\mathcal{Q}$ is supported only at the point $p$.

  Consider the long exact sequence of derived pullbacks $L_\bullet j^*$ induced by the short exact sequence (\ref{eqn: torsion-free sheaf on a surface}):
  \[
    0 \to L_1 j^* \mathcal{Q} \to L_0 j^*\mathcal{F} \to L_0 j^*\mathcal{E} \to L_0 j^* \mathcal{Q} \to 0.
  \]
  The sheaf $L_1 j^* \mathcal{Q}$ is a nonzero torsion sheaf by Lemma~\ref{lem: lines detect point-torsion}. Since the torsion part of $L_0 j^*\mathcal{F}$ is isomorphic to a skyscraper $\O_p$, this implies that $L_1 j^* \mathcal{Q} \iso \O_p$. By Lemma~\ref{lem: lines detect point-torsion} the nonderived pullback $L_0 j^*\mathcal{Q}$ is also isomorphic to a skyscraper at $p$. Since $\mathcal{Q}$ is $\Stab(p)$-invariant, the same holds for cubic curves passing through $p$ in any direction. By Lemma~\ref{lem: lines detect skyscrapers} this implies that $\mathcal{Q} \iso \O_p$. Then one easily computes that
  \[
    \Ext^1(B, \O_p) = \Ext^1(\mathcal{F}, \O_p) \iso \Ext^2(\mathcal{Q}, \O_p) \iso \k.
  \]
  
  Since the object $B$ is the projection of a skyscraper sheaf, by Corollary~\ref{cor: universal property of projection triangles} the vector space $\Ext^1(B, \O_p)$ is isomorphic to $\Ext^1(B, B)$. On the other hand, $\mathcal{F}$ is not locally free at a single point $p \in \P^2$, so $\Ext^1(\mathcal{F}, \mathcal{F})$ is at least two-dimensional by Lemma~\ref{lem: objects with a special point have large ext1}. This is a contradiction, so at least one cohomology sheaf of $B$ is not torsion-free.
\end{proof}

\begin{remark}
  The first part of the argument in Lemma~\ref{lem: always torsion in cohomology sheaves of B} shows that if $B$ is a single coherent sheaf, then it is necessarily a torsion-free sheaf which is a kernel of a map between a vector bundle and a skyscraper. This does not happen in Setting~\ref{set: b-projector}, but sheaves like that appear as the \emph{left} projections $\mA_L(\O_p)$ of skyscraper sheaves for those semiorthogonal decompositions~$\Dbcoh(\P^2) = \langle \mA, \mB \rangle$ where the subcategory $\mB$ is generated by a single exceptional vector bundle. For example, if $\mB = \langle \O \rangle$, the projection triangle is
  \[
    \O \to \O_p \to \mathcal{I}_p[1].
  \]
  Here the ideal sheaf $\mathcal{I}_p$ is exactly the sheaf described by the first part of the argument in Lemma~\ref{lem: always torsion in cohomology sheaves of B}. Thus the second part of the argument may be considered as a way to distinguish the left projection and the right projection of a skyscraper sheaf.
\end{remark}

\begin{lemma}
  \label{lem: torsion case of cohomology sheaf}
  Let $B$ be as in Setting~\textup{\ref{set: b-projector}}. Then $B$ is concentrated in degrees $[-1; 0]$, $\mH^0(B)$ is isomorphic to $\O_p$, the sheaf $\mH^{-1}(B)$ is locally free, and the projection triangle
  \[
    B \to \O_p \to A
  \]
  from Setting~\textup{\ref{set: b-projector}} is isomorphic to a truncation triangle of $B$, with $A \iso \mH^{-1}(B)[2]$.
\end{lemma}
\begin{proof}
  By Lemma~\ref{lem: always torsion in cohomology sheaves of B} we know that there exists some $i \in \Z$ such that the sheaf $\mH^i(B)$ is not torsion-free. Let $T \subset \mH^i(B)$ be the torsion subsheaf. It is $\Stab(p)$-invariant, so it is supported only at the point $p$. Consider the short exact sequence
  \[
    0 \to T \to \mH^i(B) \to \mH^i(B)/T \to 0.
  \]
  Consider the long exact sequence of derived pullbacks $L_\bullet j^*$ induced by that short exact sequence. The quotient $\mH^i(B)/T$ is a torsion-free sheaf on a smooth surface, so using Lemma~\ref{lem: torsion-free sheaves and divisors} we see $L_1j^*(\mH^i(B)/T) = 0$, and hence $L_1j^*\mH^i(B) \iso L_1j^*T$. This space is nonzero by Lemma~\ref{lem: lines detect point-torsion}. We also see that $L_0j^*\mH^i(B)$ contains the nonzero torsion subsheaf isomorphic to $L_0j^*T$.

  The relation between cohomology sheaves of $j^*B$ and derived pullbacks $L_\bullet j^* \mH^i(B)$ is described in Lemma~\ref{lem: derived restriction to a divisor}. In particular, this lemma implies that $\mH^{i-1}(j^*B)$ has a quotient isomorphic to $L_1 j^* \mH^i(B)$, and $\mH^i(j^*B)$ has a subsheaf isomorphic to $L_0j^*\mH^i(B)$. Thus both $i$'th and $(i-1)$'th cohomology sheaves of $j^*B$ are nonzero, and moreover $\mH^i(j^*B)$ has a nonzero torsion subsheaf.
  
  By Lemma~\ref{lem: p2 numerical lemma} this implies that $i = 0$, the object $j^*B$ is concentrated in degrees $[-1; 0]$, the sheaf $\mH^0(j^*B)$ is isomorphic to a skyscraper sheaf $\O_p$, and the sheaf $\mH^{-1}(j^*B)$ is locally free. By Lemma~\ref{lem: cohomology sheaves do not disappear after a restriction} the cohomology sheaves of the complex $B$ are also zero outside of the range $[-1; 0]$. Since in this case $L_0j^*\mH^0(B) \iso \mH^0(j^*B) \iso \O_p$, and the sheaf $\mH^0(B)$ is $\Stab(p)$-invariant, by Lemma~\ref{lem: lines detect skyscrapers} this implies that $\mH^0(B) \iso \O_p$.

  Since the sheaf $\mH^{-1}(j^*B)$ is locally free on a curve, its subsheaf $L_0j^*\mH^{-1}(B)$ is also locally free. The sheaf $\mH^{-1}(B)$ is $\Stab(p)$-invariant, and the curve $j\colon E \to \P^2$ passes through $p$, so the nonderived rank of the sheaf $\mH^{-1}(B)$ is constant over $\P^2$. Therefore $\mH^{-1}(B)$ is locally free.
  
  Thus $B$ is concentrated in degrees $-1$ and $0$, with $\mH^0(B) \iso \O_p$ and $\mH^{-1}(B)$ locally free. Using Lemma~\ref{lem: morphisms from topmost cohomology sheaf extend} it is easy to compute that $\Hom(B, \O_p)$ is one-dimensional. Any nonzero map is proportional to the truncation morphism $B \to \tau_{\geq 0}(B) \iso \O_p[0]$, and the cone of this map is isomorphic to $\mH^{-1}(B)[2]$. This confirms the last claim of the statement.
\end{proof}

\subsection{Full description of $A$ and $B$}
\label{ssec: exceptional bundles from local freeness}

\begin{lemma}
  \label{lem: exceptional projections are generators}
  Let $X$ be a smooth and proper variety, and let $\mA \subset \Dbcoh(X)$ be an admissible subcategory. Let $E \in \mA$ be an exceptional object and suppose that for any point $p \in X$ the projection $\mA_L(\O_p) \in \mA$ lies in the subcategory $\langle E \rangle \subset \mA$. Then $\mA = \langle E \rangle$.
\end{lemma}

\begin{proof}
  By Lemma~\ref{lem: exceptional collections are admissible} the subcategory $\langle E \rangle \subset \mA$ is admissible in $\mA$. Consider the induced semiorthogonal decomposition $\mA = \langle \mA^\prime, E \rangle$. Let $L_{\mA^\prime}\colon \Dbcoh(X) \to \mA^\prime$ be the left projection functor. It is equal to the composition of the left projection functor $\mA_L$ and the left projection to $\mA^\prime$ inside $\mA$. Thus the condition $\mA_L(\O_p) \in \langle E \rangle$ implies that $L_{\mA^\prime}(\O_p) = 0$ for all skyscrapers. Since $L_{\mA^\prime}$ is the left adjoint for the inclusion functor $\mA^\prime \monoarrow \Dbcoh(X)$, for any object $A \in \mA^\prime$ we have
  \[
    \RHom_{X}(A, \O_p) \caniso \RHom(A, L_{\mA^\prime}(\O_p)) = 0.
  \]
  This is true for all points $p \in X$, so the support of any object $A \in \mA^\prime$ is empty. Therefore the subcategory $\mA^\prime$ is a zero subcategory, which means that $\mA = \langle E \rangle$, as claimed.
\end{proof}

\begin{lemma}
  \label{lem: projection is rigid implies exceptional vector bundle}
  Let $A, B$ be as in Setting~\textup{\ref{set: b-projector}}. Then $\mA$ is generated by a single exceptional vector bundle.
\end{lemma}

\begin{proof}
  By Lemma~\ref{lem: torsion case of cohomology sheaf} we know that the object $A$ in the projection triangle $B \to \O_p \to A$ is isomorphic to $N[2]$ for some vector bundle $N$ on $\P^2$. Semiorthogonality of $A$ and $B$ implies that
  \[
    \RHom(N, N) \caniso \RHom(N[2], N[2]) \caniso \RHom(\O_p, N[2]).
  \]
  Since $N$ is locally free, the space $\RHom(\O_p, N[2])$ is concentrated in degree $0$. Therefore the graded vector space~$\Ext^*(N, N)$ is also concentrated in degree zero. Thus the bundle~$N$ is rigid. By \cite[Cor.~7]{drezet} all rigid vector bundles on $\P^2$ are direct sums of exceptional bundles. Suppose that $N$ is not a direct sum of several copies of the same exceptional bundle. Then it has two non-isomorphic direct summands $R_0$ and $R_1$, which are both exceptional bundles. It is known \cite[Lem.~4.3]{drezet-lepotier} that an exceptional vector bundle on $\P^2$ is uniquely determined by its slope, so without loss of generality we may assume that the slope of $R_0$ is strictly smaller than the slope of $R_1$.

  Since every exceptional bundle on $\P^2$ is stable \cite[Th.~4.1]{gorodentsev-rudakov}, the inequality of slopes implies that
  \[
    \mathrm{R}^0\Hom(R_1, R_0) = 0.
  \]
  Then the pair $R_0, R_1$ is semiorthogonal: indeed, $\Ext^*(N, N) = \Ext^0(N, N)$, so there are no higher $\Ext$s between the direct summands of $N$, and there are no $\mathrm{R}^0\Hom$s from $R_1$ to $R_0$ by semistability.

  The category $\mA$ is closed under direct summands, so both $R_0$ and $R_1$ lie in $\mA$. The orthogonal subcategory $\mB = {}^\perp \mA$ is contained inside ${}^\perp \langle R_0, R_1 \rangle$. By \cite[Th.~5.10]{gorodentsev-rudakov} the orthogonal to an exceptional pair on $\P^2$ is generated by a single exceptional vector bundle. In particular, this would imply that any object in $\mB$ is locally free, but we assumed from the very beginning in Setting~\ref{set: b-projector} that $B \in \mB$ is not a locally free object. This contradiction shows that $A \iso N[2] \iso \left(N^\prime\right)^{\oplus n}[2]$ is a direct sum of several copies of an exceptional vector bundle $N^\prime$.

  All exceptional bundles on $\P^2$ are rigid and therefore $\PGL(3)$-invariant. Thus by Proposition~\ref{prop: admissible subcategories are open} we know that the pullback of the projection triangle
  \[
    B \to \O_p \to A
  \]
  along some element $g \in \PGL(3)$ is a projection triangle for a skyscraper at the point $g^{-1}(p)$. Thus the projection of any skyscraper to $\mA$ is isomorphic to $\left(N^\prime\right)^{\oplus n}[2]$. By Lemma~\ref{lem: exceptional projections are generators} we see that the subcategory $\mA$ is generated by an exceptional vector bundle $N^\prime$. This establishes the second part of the statement.
\end{proof}

This lemma is the final step in the proof of Theorem~\ref{thm: reduction to B not locally free}, and hence it also establishes the main theorem of this chapter, Theorem~\ref{thm: all admissible subcategories are standard}.

\section{Admissible subcategories supported on a \texorpdfstring{$(-1)$}{-1}-curve}
\label{sec: minus one curves}

In the previous section we showed that any admissible subcategory of $\Dbcoh(\P^2)$ is one of the examples we know. We would like to generalize this statement to some other varieties, such as del Pezzo surfaces. A necessary step is to study admissible subcategories whose support is some curve in a del Pezzo surface. We have seen that for $\P^2$ there are no such subcategories, but in general they exist. For example, the structure sheaf of any smooth $(-1)$-curve is an exceptional object, and hence generates an admissible subcategory by Lemma~\ref{lem: exceptional collections are admissible}.

The main result of this section is Proposition~\ref{prop: new local classification on blow-ups}, where we prove that any admissible subcategory supported on a smooth $(-1)$-curve in a surface is one of the standard examples, i.e., it is generated by a twist of the structure sheaf of that $(-1)$-curve. It is possible and not too difficult to give a proof along the lines of Section~\ref{sec: projective plane}. The reason why this is possible is that the powerful Proposition~\ref{prop: numerical lemma without numbers} may be applied in this situation not only to the anticanonical divisor of the surface, but in fact to any divisor which in a neighborhood of the $(-1)$-curve is equivalent to an anticanonical divisor. For example, any curve which intersects the $(-1)$-curve transversely at a single point is an option. For a more detailed sketch of an argument along these lines we refer the interested reader to the thesis version of this paper \cite[Sec.~5.1]{mythesis}.

In this section we use a different, more conceptual approach, following a suggestion by Kuznetsov. It uses the additivity of Hochschild homology in the form proved in \cite{kuznetsov-oldhochschild}.

One application of this local classification result is given in Corollary~\ref{cor: no phantoms in blow-ups of antifano}, where we prove the non-existence of phantom subcategories in some blow-ups of surfaces. Note that any nontrivial blow-up has a nontrivial semiorthogonal decomposition \cite{orlov93}, so the non-existence of phantoms is interesting.

We start with a reminder on Hochschild homology and its interaction with semiorthogonal decompositions, following \cite{kuznetsov-oldhochschild}, in Subsection~\ref{ssec: reminder on hochschild}. We complete the classification of possible admissible subcategories supported on a $(-1)$-curve in Subsection~\ref{ssec: new argument for minus one curves}.

\subsection{Reminder on Hochschild homology}
\label{ssec: reminder on hochschild}

The material below is taken from \cite{kuznetsov-oldhochschild}. See the reference for additional details and the proofs.

Let $X$ be a smooth and proper variety, and let $\langle \mA, \mB \rangle = \Dbcoh(X)$ be a semiorthogonal decomposition. Let $\mB_R$ and $\mA_L$ denote the projection functors from $\Dbcoh(X)$ to $\mB$ and $\mA$, right and left respectively. By Proposition~\ref{prop: fm kernels for projections} there exist Fourier--Mukai kernels in $\Dbcoh(X \times X)$ representing those functors, and we denote the kernels with the same symbols. The kernels for the projection functors fit into a triangle in $\Dbcoh(X \times X)$:
\[
  \mB_R \to \Delta_*\O_X \to \mA_L.
\]

We denote the canonical line bundle $\Omega^{\dim X}_X$ of $X$ by $K_X$. The graded vector space
\[
  \mathrm{HH}_\bullet(X) := \RHom_{X \times X}(\Delta_*\O_X, \Delta_*K_X[\dim X])
\]
is called the \emph{Hochschild homology} of $X$. It is a straightforward consequence of this definition that there exists an isomorphism $\mathrm{HH}_{-\dim X}(X) \caniso H^0(X, K_X)$. 

For objects in $\Dbcoh(X \times X)$ there is a binary operation $\operatorname{-} \circ \operatorname{-}$, called convolution, which corresponds to the composition of Fourier--Mukai transforms. The structure sheaf $\Delta_*\O_X$ of the diagonal is the identity element for this operation, and thus there is an isomorphism between $\Delta_*K_X[\dim X]$ and $\Delta_*\O_X \circ \Delta_*K_X[\dim X]$. It is proved in \cite[Prop.~5.5]{kuznetsov-oldhochschild} that any morphism $\varphi \in \mathrm{HH}_{m}(X)$ can be uniquely extended to a morphism of triangles:

\begin{equation}
  \label{eqn: hochschild additivity}
  \begin{tikzcd}
    \mB_R \arrow[r] \arrow[d, "\gamma_\mA(\varphi)"] & \Delta_*\O_X \arrow[r] \arrow[d, "\varphi"] & \mA_L \arrow[d, "\gamma_\mB(\varphi)"] \\
    \mB_R \circ \Delta_*K_X[m + \dim X] \arrow[r] & \Delta_*K_X[m + \dim X] \arrow[r] & \mA_L  \circ \Delta_*K_X[m + \dim X]
  \end{tikzcd}
\end{equation}

The spaces $\RHom(\mB_R, \mB_R \circ \Delta_*K_X[\dim X])$ and $\RHom(\mA_L, \mA_L \circ \Delta_*K_X[\dim X])$ are called Hochschild homology spaces $\mathrm{HH}_\bullet(\mB)$ and $\mathrm{HH}_\bullet(\mA)$ respectively. They depend only on the categories $\mB$ and $\mA$, as proved in \cite{kuznetsov-oldhochschild}. The uniqueness and existence of the extension of the map $\varphi \in \mathrm{HH}_m(X)$ to a morphism of triangles (\ref{eqn: hochschild additivity}) using certain maps $\gamma_\mA(\varphi)$ and $\gamma_\mB(\varphi)$ produces a map
\[
  \mathrm{HH}_\bullet(X) \xrightarrow{(\gamma_\mA, \gamma_\mB)} \mathrm{HH}_\bullet(\mA) \oplus \mathrm{HH}_\bullet(\mB).
\]
Theorem 7.3 in \cite{kuznetsov-oldhochschild} shows that this map is an isomorphism, i.e., Hochschild homology is additive for semiorthogonal decompositions.

\subsection{Local classification on a $(-1)$-curve}
\label{ssec: new argument for minus one curves}

Let $S$ be a smooth proper surface. Suppose that $\mA \subset \Dbcoh(S)$ is an admissible subcategory supported (in the sense of Definition~\ref{def: support of a category}) on some smooth $(-1)$-curve $E \subset S$. In this subsection we first show that $\mathrm{HH}_{-2}(\mA) = 0$ in Lemma~\ref{lem: top hochschild homology of torsion subcategories}, and then in Lemmas~\ref{lem: hochschild homology near minus one curves} and~\ref{lem: section-theoretic support on a smooth curve} we deduce from the vanishing of this homology group the fact that any object of $\mA$ is a pushforward of some object from the derived category of coherent sheaves $\Dbcoh(E)$ on the $(-1)$-curve. Finally, in Proposition~\ref{prop: new local classification on blow-ups} we complete the classification.

\begin{lemma}
  \label{lem: top hochschild homology of torsion subcategories}
  Let $X$ be a smooth and proper variety, and let $\mA \subset \Dbcoh(X)$ be an admissible subcategory. Assume that $\mA$ is supported on a proper closed subset $Z \subset X$. Then the bottom Hochschild homology $\mathrm{HH}_{-\dim X}(\mA)$ vanishes.
\end{lemma}

\begin{proof}
  Let $\mB := {}^\perp \mA$ be the orthogonal subcategory in $\Dbcoh(X)$. Let $\gamma_\mB\colon \mathrm{HH}_\bullet(X) \to \mathrm{HH}_\bullet(\mB)$ be the restriction morphism defined in Subsection~\ref{ssec: reminder on hochschild}. By the additivity of Hochschild homology \cite[Th.~7.3]{kuznetsov-oldhochschild} the kernel of the map $\mathrm{HH}_{-\dim X}(X) \to \mathrm{HH}_{-\dim X}(\mB)$ is isomorphic to $\mathrm{HH}_{-\dim X}(\mA)$. Thus it is enough to prove that this map is injective.

  Using the definition given in Subsection~\ref{ssec: reminder on hochschild} it is easy to compute that the vector space $\mathrm{HH}_{-\dim X}(X)$ is isomorphic to $H^0(X, K_X)$. Suppose $s \in H^0(X, K_X)$ is a nonzero section such that $\gamma_\mB(s)$ is a zero class in $\mathrm{HH}_{-\dim X}(\mB)$. Pick a point $p$ in the open subset $X \setminus Z$ such that $s$ does not vanish at $p$. By assumption the skyscraper sheaf $\O_p$ is orthogonal to every object in $\mA$, and hence $\O_p \in \mB$. The morphism of triangles (\ref{eqn: hochschild additivity}) of objects in $\Dbcoh(X \times X)$ for the class $s \in \mathrm{HH}_{-\dim X}(X)$ produces the following morphism of triangles in $\Dbcoh(X)$ via a Fourier--Mukai transform of the skyscraper sheaf $\O_p$:
  \begin{equation}
    \begin{tikzcd}
      \mB_R(\O_p) \caniso \O_p \arrow[r, "\mathrm{id}"] \arrow[d, "\gamma_\mB(s)(\O_p)"] & \O_p \arrow[r] \arrow[d, "s(p)"] & 0 \arrow[d] \\
      \mB_R(\O_p \otimes K_X) \caniso \O_p \otimes K_X \arrow[r, "\mathrm{id}"] & \O_p \otimes K_X \arrow[r] & 0
    \end{tikzcd}
  \end{equation}

  By assumption $s(p) \neq 0$, so by the commutativity of the diagram the leftmost vertical morphism $\gamma_\mB(s)(\O_p)$ is not zero, hence the natural transformation $\gamma_\mB(s) \in \mathrm{HH}_{-\dim X}(\mB)$ is nonzero. This is a contradiction, so the morphism $\gamma_\mB$ is injective on $\mathrm{HH}_{-\dim X}$ and the lemma is proved.
\end{proof}

\begin{remark}
  The argument used in the proof implies that $\supp(\mA)$ is contained in the base locus of the canonical bundle. This gives an alternative proof of \cite[Th.~1.1]{kawatani-okawa} for smooth proper varieties.
\end{remark}

To any global section $s \in H^0(X, K_X)$ and any object $F \in \Dbcoh(X)$ we can associate the multiplication morphism $F \xrightarrow{\cdot s} F \otimes K_X$. These morphisms (for $F \iso \O_p$) were used in the proof of the preceding lemma. If $F$ is a torsion object, then the multiplication morphism is well-defined even if $s$ is only a local section as long as it is defined on a Zariski-open neighborhood of the set-theoretic support of $F$. We use this observation in the following statement.

\begin{lemma}
  \label{lem: hochschild and section-theoretic support}
  Let $X$ be a smooth and proper variety, and let $\mA \subset \Dbcoh(X)$ be an admissible subcategory. Assume that $\mA$ is supported on a proper closed subset $Z \subset X$. Let $U \subset X$ be some Zariski-open neighborhood of $Z$, and let $s \in H^0(U, K_X|_U)$ be a local section of the canonical bundle. Then for any object $A \in \mA$ the morphism
  \[
    A \xrightarrow{\cdot s} A \otimes K_X
  \]
  in the derived category $\Dbcoh(X)$ given by the multiplication with the section $s$ is well-defined, and it is a zero morphism.
\end{lemma}

\begin{proof}
  The multiplication morphism is well-defined since any object $A \in \mA$ is set-theoretically supported on the subset $Z \subset U$ by assumption.

  To show that the morphism vanishes, we interpret it as a result of a certain natural transformation. Let $\mA_L \in \Dbcoh(X \times X)$ denote the Fourier--Mukai kernel for the (left) projection functor from the category $\Dbcoh(X)$ to the subcategory $\mA$, as in Proposition~\ref{prop: fm kernels for projections}. By Lemma~\ref{lem: support of the kernel} the object $\mA_L$ is supported on $Z \times Z \subset X \times X$.

  Since $U \subset X$ is an open subscheme, there exists a natural restriction morphism $\O_X \to \O_U$ of quasicoherent sheaves on $X$. Denote by $\Delta$ the diagonal embedding $X \monoarrow X \times X$. Note that the local section $s \in H^0(U, K_X|_U)$ produces a morphism $\O_U \to K_X \otimes \O_U$ of quasicoherent sheaves on $X$. Using it, we can define the following morphism of quasicoherent sheaves on the product $X \times X$:
  \[
    \varphi_s\colon \Delta_*\O_U \xrightarrow{\cdot s}  \Delta_*(K_X \otimes \O_U).
  \]
  The application of the functor $\mA_L$ defines a morphism of Fourier--Mukai kernels
  \[
    \mA_L \circ \varphi_s\colon \mA_L \circ \Delta_*\O_U \to \mA_L \circ \Delta_*(K_X \otimes \O_U).
  \]

  Since the object $\mA_L$ is supported on the subset $Z \times Z \subset U \times U$, we see that $\mA_L \circ \Delta_*\O_U$ is naturally isomorphic to $\mA_L \circ \Delta_*\O_X \caniso \mA_L$, and similarly
  \[
    \mA_L \circ \Delta_*(K_X \otimes \O_U) \iso \mA_L \circ \Delta_*K_X
  \]
  as well. Thus there exists an isomorphism
  \[
    \Hom(\mA_L \circ \Delta_*\O_U, \mA_L \circ \Delta_*(K_X \otimes \O_U)) \caniso \Hom(\mA_L, \mA_L \circ \Delta_*K_X) \caniso \mathrm{HH}_{-\dim X}(\mA).
  \]
  Since the Hochschild homology space $\mathrm{HH}_{-\dim X}(\mA)$ vanishes by Lemma~\ref{lem: top hochschild homology of torsion subcategories}, the morphism of Fourier--Mukai kernels $\mA_L \circ \varphi_s$ is necessarily zero.

  Now let $A \in \mA$ be an arbitrary object. The result of the natural transformation $\mA_L \circ \varphi_s$ applied to the object $A \otimes K_X^\dual$ is the (left) projection of the morphism
  \[
    A \otimes K_X^\dual \xrightarrow{\cdot s} A \otimes K_X^\dual \otimes K_X \caniso A
  \]
  in the category $\Dbcoh(X)$ to the subcategory $\mA$. Note that since $A$ is an object of $\mA$, the projection $\mA_L(A)$ is canonically isomorphic to $A$.

  By Lemma~\ref{lem: projection triangles exist} the projection of the morphism $A \otimes K_X^\dual \xrightarrow{s} A$ to the subcategory $\mA \subset \Dbcoh(X)$ fits into the following commutative square:
  \begin{center}
    \begin{tikzcd}
      A \otimes K_X^\dual \arrow[r, "s"] \arrow[d] & A \arrow[d, "\mathrm{id}"] \\
      \mA_L(A \otimes K_X^\dual) \arrow[r] & \mA_L(A) \caniso A
    \end{tikzcd}
  \end{center}
  The bottom horizontal morphism is zero since it is given by the natural transformation arising from the zero morphism $\mA_L \circ \varphi_s$. Since the right vertical arrow is an isomorphism, the top horizontal map $A \otimes K_X^\dual \to A$ is also zero, which is equivalent to the vanishing of the morphism $A \to A \otimes K_X$, and this is exactly what we wanted to show.
\end{proof}

\begin{remark}
  A different way to prove Lemma~\ref{lem: hochschild and section-theoretic support} is to say that any local section of $K_X$ extends to a global section of some line bundle $L$ which is isomorphic to $K_X$ in a neighborhood of the subset $Z \subset X$, and then use the notion of Hochschild cohomology with support from the paper \cite{kuznetsov-oldhochschild}, i.e., use $\Delta_*L$ instead of $\Delta_*K_X$ in the definition of Hochschild homology. Then an analogue of Lemma~\ref{lem: top hochschild homology of torsion subcategories} would also be true by the additivity criterion for this generalized invariant \cite[Prop.~5.5]{kuznetsov-oldhochschild}.
\end{remark}

\begin{lemma}
  \label{lem: hochschild homology near minus one curves}
  Let $S$ be a smooth and proper surface, and let $C \subset S$ be a smooth $(-1)$-curve. Let $c \in \Gamma(S, \O_S(C))$ be the section cutting out the curve $C$. Let $\mA \subset \Dbcoh(S)$ be an admissible subcategory supported on $C$. Then for every object $A \in \mA$ the morphism $A \to A(C)$ in the derived category given by the multiplication with a section $c$ is the zero morphism.
\end{lemma}

\begin{proof}
  By the adjunction formula the intersection $K_S \cdot C$ equals $-1$, and since $C$ is isomorphic to the projective line $\P^1$, this determines the restriction $K_S|_C$ uniquely. Thus locally, in a Zariski neighborhood of the curve $C \subset S$, the line bundle $K_S$ is isomorphic to $\O_S(C)$. Since any object of the category $\mA$ is supported on the subset $C$, the global section $c \in H^0(S, \O_S(C))$ defines a local section of $K_S$ via that isomorphism. Then the statement follows from Lemma~\ref{lem: hochschild and section-theoretic support}.
\end{proof}

\begin{lemma}
  \label{lem: section-theoretic support on a smooth curve}
  Let $S$ be a smooth and proper surface, and let $C \subset S$ be a smooth curve. Pick a section $s \in \Gamma(S, \O_S(C))$ cutting out the curve $C$. Assume that $A \in \Dbcoh(S)$ is an object such that the morphism $A \xrightarrow{\cdot s} A(C)$ is zero in the derived category. Then $A$ is isomorphic to a pushforward of an object from $\Dbcoh(C)$.
\end{lemma}

\begin{remark}
  A stronger result valid in arbitrary dimension was recently proved in \cite[Th.~3.2]{lieblich-olsson}. The two-dimensional case is significantly easier than the general statement, so we include the direct proof.
\end{remark}

\begin{proof}
  Denote by $j\colon C \monoarrow S$ the inclusion morphism. Consider the restriction triangle for $A$:
  \[
    A(-C) \xrightarrow{s} A \to j_*j^*A.
  \]
  The first morphism in this triangle vanishes by assumption, thus the morphism $A \to j_*j^*A$ is a split monomorphism, i.e., an inclusion of a direct summand. By Lemma~\ref{lem: objects on smooth curves split} the derived pullback $j^*A \in \Dbcoh(C)$ in the derived category of a smooth curve is formal, i.e., it is a direct sum of shifts of cohomology sheaves. Then the pushforward $j_*j^*A$ is also formal, and any direct summand of a formal complex is formal, given by a choice of a direct summand in each cohomology sheaf. Thus $A \iso \oplus \mH^i(A)[-i]$, and each cohomology sheaf $\mH^i(A)$ is a direct summand of a sheaf $j_* \mH^i(j^*A)$. Any direct summand of the pushforward sheaf $j_* \mH^i(j^*A)$ is a pushforward of some direct summand of $\mH^i(j^*A)$. Thus $A$ is isomorphic to a pushforward of an object in $\Dbcoh(C)$.
\end{proof}

Now we can prove the main result of this section.

\statemaintheoremMinusOneCurves

\begin{proof}
  By Lemmas~\ref{lem: hochschild homology near minus one curves} and~\ref{lem: section-theoretic support on a smooth curve} any object of the subcategory $\mA$ is a pushforward of some object from $\Dbcoh(E)$. Let $G \in \Dbcoh(E)$ be an object such that $j_*G$ is a generator of $\mA$. Note that $\mA$ contains pushforwards of all objects in $\langle G \rangle \subset \Dbcoh(E)$. Suppose that $G$ generates the entire derived category of $E$. Then $\mA$ contains a skyscraper sheaf at some point of $E$. On the surface $S$ this skyscraper sheaf may be deformed into a skyscraper sheaf at some point away from $E \subset S$. Since admissible subcategories are closed under small deformations (Proposition~\ref{prop: admissible subcategories are open}), this is a contradiction with the assumption that $\mA$ is supported only on the curve $E$. Therefore $G \in \Dbcoh(E)$ cannot be a generator.

  Since $E$ is isomorphic to $\P^1$, any object in $\Dbcoh(E)$ splits into a direct sum of shifts of torsion sheaves and line bundles. It is easy to check that an object $G \in \Dbcoh(\P^1)$ is not a generator only in two cases: either $G$ is a torsion object, or $G$ is a direct sum of several copies of the same line bundle. The object $G$ cannot be torsion by the same argument as above. Thus $G$ is a direct sum of shifts of copies of $\O_{\P^1}(k)$ for some fixed $k$, and the subcategory $\mA$ generated by its pushforward $j_* G$ can also be generated by $j_*\O_E(k)$, as claimed.
\end{proof}

This local classification implies that there are no phantom subcategories supported on a smooth $(-1)$-curve. We may deduce from this the non-existence of phantom subcategories in some surfaces.

\begin{corollary}
  \label{cor: no phantoms in blow-ups of antifano}
  Let $S$ be a surface with a globally generated canonical bundle. Let $\pi\colon S^\prime \to S$ be a blow-up of several distinct points on $S$. Let $\mA \subset \Dbcoh(S^\prime)$ be an admissible subcategory. Then there exists a subset $\{ E_i \}_{i \in I}$ of exceptional divisors of the morphism $\pi$ and the corresponding integers $\{ k_i \in \Z \}_{i \in I}$ such that either $\mA$ is equal to the subcategory $\oplus_{i \in I} \langle \O_{E_i}(k_i E_i) \rangle$, or is the orthogonal to that subcategory. In particular, $\Dbcoh(S^\prime)$ does not contain any phantom subcategories.
\end{corollary}

\begin{proof}
  Since the canonical bundle of $S$ is globally generated, the base locus $Z \subset S^\prime$ of the canonical bundle $K_{S^\prime}$ is equal to the union of exceptional divisors. By \cite[Th.~1.1]{kawatani-okawa} any admissible subcategory of $S^\prime$ either is supported on $Z$, or its orthogonal ${}^\perp \mA$ is supported on $Z$. By switching to the orthogonal we may suppose without loss of generality that $\mA$ is supported on $Z$.

  By definition $Z$ is a disjoint union of several smooth $(-1)$-curves. Objects supported on different $(-1)$-curves are completely orthogonal to each other. Thus $\mA$ splits into a completely orthogonal sum of subcategories supported on each $(-1)$-curve separately. The options for each summand are classified in Proposition~\ref{prop: new local classification on blow-ups}, and there are no phantom subcategories among them.
\end{proof}


\section{Rational elliptic surfaces, del Pezzo surfaces, and phantoms}
\label{sec: del pezzos}

From the classification of admissible subcategories of $\Dbcoh(\P^2)$ given by Theorem~\ref{thm: all admissible subcategories are standard} we easily see that there are no phantom subcategories in $\P^2$. In fact, the full classification is not necessary, and it is not hard to come to the same conclusion right after Lemma~\ref{lem: p2 numerical lemma}. In this section we explore other situations where we can apply the strong structural result given by Proposition~\ref{prop: numerical lemma without numbers} to study phantoms.

To apply the methods of Section~\ref{sec: numerical lemmas} we need a surface with many anticanonical divisors. We concentrate on the case of rational elliptic surfaces. Those are the surfaces whose anticanonical linear system defines a fibration over $\P^1$. First we show that a phantom subcategory in a rational elliptic surface cannot have full support; in fact, in Theorem~\ref{thm: rational elliptic surfaces} we prove that the support of any phantom subcategory is contained in the union of sections and reducible fibers of the anticanonical fibration.

This is a relatively explicit description: all sections are smooth $(-1)$-curves, and Kodaira classified the reducible fibers of (minimal) elliptic surfaces. Since we have managed to prove a classification of admissible subcategories supported on a single smooth $(-1)$-curve in Section~\ref{sec: minus one curves}, it seems reasonable to expect that it may be possible to classify the admissible subcategories supported on this more complicated, but still quite specific, configuration of negative curves. We have done this partially.

Using Proposition~\ref{prop: numerical lemma without numbers} we have proved the classification of admissible subcategories in rational elliptic surfaces which are supported only along the sections of the anticanonical fibration. We show in Theorem~\ref{thm: admissible subcategories on sections} that the only option is to choose several pairwise disjoint sections, pick a line bundle on each one, and span the subcategory by their pushforwards.

Finally, we utilize the classification result above and a relation between del Pezzo surfaces and rational elliptic surfaces to show in Theorem~\ref{thm: no phantoms in del pezzos} that there are no phantoms in del~Pezzo surfaces.

The proofs in this section use the notion of a \emph{point-support} of an admissible subcategory at some point. This notion is introduced in Subsection~\ref{ssec: point-supports}. We also need some additional lemmas about objects set-theoretically supported on curves in surfaces. We study them by pulling them back along curves transverse to the support in Subsection~\ref{ssec: cutting lemmas}. We proceed by studying possible phantom subcategories in rational elliptic surfaces in Subsection~\ref{ssec: rational elliptic surfaces}. The classification of admissible subcategories supported on the union of sections of the elliptic fibration is established in Subsection~\ref{ssec: sections of anticanonical system}. We conclude by showing the non-existence of phantoms in del Pezzo surfaces in Subsection~\ref{ssec: no phantoms in del pezzos}.

\subsection{Point-supports of admissible subcategories}
\label{ssec: point-supports}

\begin{definition}
  \label{def: point-support}
  Let $X$ be an algebraic variety. Let $\mB \subset \Dbcoh(X)$ be an admissible subcategory. Pick a point $p \in X$. The \emph{point-support} of $\mB$ at $p$ is a set-theoretic support $\supp(\mB_R(\O_p))$ of the right projection of the skyscraper sheaf at $p$ to the subcategory $\mB$.
\end{definition}

\begin{remark}
  In this definition we chose to use the right projection functor. I do not know whether the supports of the right projection $\mB_R(\O_p)$ and the left projection $\mB_L(\O_p)$ always coincide.
\end{remark}

\begin{lemma}
  \label{lem: properties of point-supports}
  Let $X$ be an algebraic variety, let $p \in X$ be a point, and let $\mB \subset \Dbcoh(X)$ be an admissible subcategory. Let $S_p$ be the point-support of $\mB$ at $p$.
  \begin{enumerate}
    \item If $S_p$ is not empty, then it is a connected closed subset which contains $p$.
    \item If $S_p$ is empty, then $\O_p \in \mB^\perp$. If $S_p = \{ p \}$, then $\O_p \in \mB$.
    \item Let $\mA = \mB^\perp$ be the orthogonal subcategory. Consider the projection triangle
      \[
        \mB_R(\O_p) \to \O_p \to \mA_L(\O_p)
      \]
      If $S_p$ is neither empty nor $\{ p \}$, then $\supp(\mA_L(\O_p)) = S_p = \supp(\mB_R(\O_p))$.
  \end{enumerate}
\end{lemma}

\begin{proof}
  Let $\mA = \mB^\perp$ be the orthogonal subcategory, and let $A_p = \mA_L(\O_p)$ and $B_p = \mB_R(\O_p)$ denote the projections of the skyscraper sheaf $\O_p$, fitting into the projection triangle
  \[
    B_p \to \O_p \to A_p.
  \]
  Then $S_p = \supp(B_p)$ by definition. Suppose that this set has a nonempty connected component which does not contain $p$. Then by Lemma~\ref{lem: disjoint support forces splitting} the object $B_p$ has a nonzero direct summand~$B^\prime$ which is supported away from $p$. Then the map $B_p \to \O_p$ factors through the projection to $B_p/B^\prime$, and therefore its cone, which is $A_p$, has a direct summand isomorphic to~$B^\prime[1]$. But then $\RHom(B_p, A_p) \neq 0$, which is a contradiction with semiorthogonality. Thus the support of $B_p$ (and by a similar argument the support of $A_p$ as well) is either empty, or a connected subset containing $p$. This proves part (1).
  
  Suppose now that $S_p = \{ p \}$. The projection triangle implies in this case that $\supp(A_p)$ is contained in the one-point set $\{ p \}$. Assume that both $A_p$ and $B_p$ are nonzero objects supported at a single point $p$. Let $a \in \Z$ be the smallest number such that $\mH^a(A_p) \neq 0$ and let~$b \in \Z$ be the largest number such that $\mH^b(B_p) \neq 0$. Pick nonzero morphisms $\mH^b(B_p) \epiarrow \O_p$ and $\O_p \monoarrow \mH^a(A_p)$, which always exist for coherent sheaves supported at one point. Then the composition
  \[
    B_p \to \mH^b(B_p)[-b] \epiarrow \O_p[-b] \monoarrow \mH^a(A_p)[-b] \to A_p[a-b]
  \]
  with truncation morphisms is a morphism $B_p \to A_p[a-b]$ which by construction is nonzero on cohomology sheaves. This contradicts semiorthogonality. Thus~$A_p$ must be a zero object when~$S_p = \{ p \}$, so part (2) is proved.
  
  To deal with the last part, note that the long exact sequence of cohomology sheaves proves that $\supp(B_p) \subset \supp(A_p) \cup \{ p \}$ and similarly for $\supp(A_p)$. Since both of those supports are either empty or contain the point $p$, part (3) follows.
\end{proof}

\begin{lemma}
  \label{lem: no semiorthogonality for zero-dimensional intersections}
  Let $Y$ be a smooth variety, $S_1, S_2 \subset Y$ two closed subsets whose set-theoretic intersection $S_1 \cap S_2$ contains an isolated point. Let $F_1, F_2 \in \Dperf(Y)$ be objects whose set-theoretical supports are $S_1, S_2$ respectively. Then $\RHom(F_1, F_2) \neq 0$.
\end{lemma}

\begin{proof}
  We can compute the $\RHom$-space by the dualization:
  \[
    \RHom(F_1, F_2) \caniso \RGamma(Y, F_1^\dual \otimes F_2).
  \]
  The support of the tensor product $F_1^\dual \otimes F_2$ is the intersection $S_1 \cap S_2$. It contains an isolated point, so by Lemma~\ref{lem: disjoint support forces splitting} the object $F_1^\dual \otimes F_2$ has a nonzero direct summand supported only at a single point. Any object with zero-dimensional support has a nonvanishing (hyper)cohomology class given by a nonzero global section of the lowest degree cohomology sheaf, so the lemma is proved.
\end{proof}

\begin{lemma}
  \label{lem: skyscrapers and intersections}
  Let $Y$ be a smooth variety, $p, q \in Y$ distinct points. Let $\mB \subset \Dbcoh(Y)$ be an admissible subcategory. Denote by $S_p$ and $S_q$ the point-supports of $\mB$ at $p$ and $q$ respectively. Then the set-theoretic intersection $S_p \cap S_q$ does not contain any isolated points.
\end{lemma}

\begin{remark}
  This lemma is most useful on surfaces, where all nontrivial point-supports are curves, and curves usually intersect along finitely many points.
\end{remark}

\begin{figure*}[ht]
  \centering
  \def\svgwidth{0.3\columnwidth}
  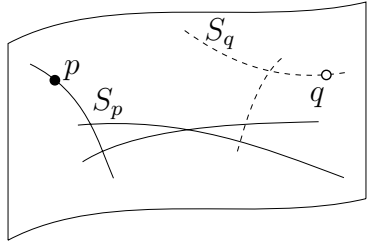
  \caption{An impossible situation}
  \label{fig: skyscrapers and intersections}
\end{figure*}

\begin{proof}
  Without loss of generality assume that $S_p$ is larger than just $\{ p \}$. Let $\mA = \mB^\perp$ be the orthogonal subcategory. Let~$A_p = \mA_L(\O_p)$ and~$B_p = \mB_R(\O_p)$ be the projections of the skyscraper sheaf at $p$ to the categories $\mA$ and $\mB$ respectively. By the last part of Lemma~\ref{lem: properties of point-supports} we see that $\supp(A_p) = \supp(B_p) =  S_p$ in this case.

  First, suppose that $S_q = \{ q \}$ and $S_p$ contains $q$. Then by Lemma~\ref{lem: properties of point-supports} this implies that~$\O_q \in \mB$. But then the graded space $\RHom(\O_q, A_p)$ is zero by semiorthogonality, which implies that $\supp(A_p) = S_p$ does not contain $q$, a contradiction.
  
  Thus we may assume that $S_q \neq \{ q \}$. Then by the same lemma $\supp(\mA_L(\O_q)) = S_q$. Consider the space $\RHom(B_p, \mA_L(\O_q))$. If the intersection $S_p \cap S_q$ contains an isolated point, then by Lemma~\ref{lem: no semiorthogonality for zero-dimensional intersections} this space is not zero, but this again is impossible by semiorthogonality of $\mA$ and $\mB$.
\end{proof}

\subsection{Cutting lemmas}
\label{ssec: cutting lemmas}

This subsection contains a few observations about objects in the derived categories of surfaces which are set-theoretically supported on curves.

\begin{definition}
  Let $S$ be a smooth surface, $F \in \Dbcoh(S)$ an object. Assume that the set-theoretic support of $F$ is a reduced curve $C \subset S$. A \emph{slice} of $F$ at a point $p \in C$ is the derived pullback $j^*F$ to a curve $j\colon D \to S$ which is smooth at the point $p$ and does not intersect $C$ anywhere else.
\end{definition}

Note that an alternative way to state the definition would be to let $D$ intersect $C$ at some other points, but replace the derived pullback by the largest direct summand supported at the point $p$. This is equivalent to replacing $D$ with an open neighborhood of $p$ in $D$.

\begin{lemma}
  \label{lem: equicutting lemma}
  Let $U$ be a smooth surface, let $i\colon C \monoarrow U$ be a curve, and let $p \in C$ be a point. Let $F \in \Dperf(U)$ be an object whose set-theoretic support is $C$. Suppose that there exists a slice of $F$ at $p$ which is a torsion object of length one. Then $C$ equipped with a reduced scheme structure is smooth at the point $p$, and, after possibly replacing $U$ by a Zariski-neighborhood of $p$, the object $F$ is isomorphic to $i_*\O_C[a]$ for some shift $a \in \Z$.
\end{lemma}

\begin{proof}
  Let $j\colon D \to U$ be a smooth curve such that the derived pullback $j^*F \in \Dbcoh(D)$ is a torsion object of length one, i.e., it is isomorphic to a shift of a skyscraper sheaf $\O_p[a]$ for some $a \in \Z$. By Lemma~\ref{lem: derived restriction to a divisor} the isomorphism $j^*F \iso \O_p[a]$ implies that in a neighborhood of the point $p$ the object $F$ has only one nonzero cohomology sheaf, $F \iso \mathcal{F}[a]$ for some coherent sheaf $\mathcal{F} \in \mathrm{Coh}(U)$. By shrinking $U$ we may assume that $U \iso \Spec(A)$ for some ring $A$, the coherent sheaf $\mathcal{F}$ corresponds to a module $M$ over the ring $A$, and the smooth curve $D \subset U$ is defined by an equation $\{ d = 0 \}$ for an element $d \in A$. The assumption that $M/dM$ is isomorphic to a skyscraper sheaf at the point $p$ implies that $M/\mathfrak{m}_pM$ is also a skyscraper sheaf, so by Nakayama's lemma $M$ is locally isomorphic to a cyclic module, i.e., there exists an isomorphism $M \iso A/I$ for some ideal $I \subset A$.

  Since $(A/I)/d = (A/d)/I$ has length one, and $A/(d)$ is a discrete valuation ring, this means that the image of $I$ in the quotient ring $A/(d)$ is generated by one regular element $\tilde{f} \in A/(d)$ such that $\tilde{f}$ generates the maximal ideal of $A/(d)$. Pick a preimage $f \in A$ of $\tilde{f}$ in the ideal $I$. We will show that $f$ generates $I$. Consider the short exact sequence
  \[
    0 \to I/fA \to A/fA \to A/I \to 0.
  \]
  The derived pullback to the smooth curve $j\colon D \to U$ produces a long exact sequence of modules over the quotient ring $A/d$. Consider the following fragment:
  \[
    L_1j^*(A/I) \to L_0j^*(I/fA) \to L_0j^*(A/fA) \to L_0j^*(A/I) \to 0
  \]
  Since $L_1j^*(A/I) \iso L_1j^*M = 0$ by the assumption of the theorem, this is in fact a short exact sequence. Note that $L_0j^*(A/fA)$ by the definition is isomorphic to the quotient $A/(d, f)$. Since $f$ is equal to $\tilde{f}$ modulo $(d)$, this quotient is isomorphic to $(A/d)/\tilde{f}$, which by the choice of $\tilde{f}$ is isomorphic to $(A/d)/I$. Thus the last two terms of the short exact sequence are both torsion sheaves of length one. Therefore $L_0j^*(I/fA)$ can only be zero. In particular, $I/fA$ is not supported at the point $p \in U$.

  Thus the inclusion $(f) \subset I$ is an isomorphism at the point $p$, so after shrinking $U$ we can assume that $I = (f)$, so the module $M \iso A/fA$ is the structure sheaf of the curve $\{ f = 0 \}$. Note additionally that since $\tilde{f} \in A/d$ has valuation 1, the curve $C = \{ f = 0 \}$ is smooth at the point $p$. This finishes the proof of the lemma.
\end{proof}

This local description may be improved to a global one if we consider the slices at all points of the curve instead of a single point.

\begin{lemma}
  \label{lem: global equicutting lemma}
  Let $S$ be a smooth surface, and let $i\colon C \monoarrow S$ be a connected curve. Let $F \in \Dperf(S)$ be an object whose set-theoretic support is $C$. Suppose that at each point $p \in C$ there exists a slice of $F$ which is a torsion object of length one. Then the curve $C$ is smooth, and the object $F$ is isomorphic to a pushforward $i_*(L)[a]$ for some line bundle $L \in \Pic(C)$ and a shift $a \in \Z$.
\end{lemma}

\begin{remark}
  If $C$ is not connected, the pushforwards of line bundles from different connected components may have different shifts, but otherwise the conclusion is the same.
\end{remark}

\begin{proof}
  By Lemma~\ref{lem: equicutting lemma}, applied at all points of $C = \supp(B)$, the object $B$ is a shift of some coherent sheaf $\mathcal{F} \in \mathrm{Coh}(Y)$. Moreover, the scheme-theoretic support of $\mathcal{F}$ is equal to the reduced scheme structure on the curve $C$. Therefore $\mathcal{F}$ is a pushforward of a coherent sheaf $\mathcal{F}^\prime \in \mathrm{Coh}(C)$. Locally the sheaf $\mathcal{F}^\prime$ is isomorphic to the structure sheaf of $C$, thus $\mathcal{F}^\prime$ is in fact a line bundle on $C$.
\end{proof}

The minimal possible length of any nonzero slice is equal to one, and lemmas above show what happens in this minimal case. In particular, we see that the support curve is necessarily smooth at the point of the slice. When the curve is singular, we will see below that the length of the slice is always greater or equal to the multiplicity of the singular point. Moreover, the minimal possible length only happens in some special situations. The rest of this subsection is dedicated to the study of that minimal case.

\begin{lemma}
  \label{lem: thin slices are minimal}
  Let $U$ be a smooth surface, and let $F \in \Dbcoh(U)$ be an object whose set-theoretic support is a reduced curve $C \subset U$. Let $j\colon D \monoarrow U$ be a smooth curve which intersects $C$ at a single point $p \in C$. Denote by $d$ the multiplicity of the curve $C$ at $p$.
  \begin{enumerate}
    \item For the length of the torsion object $j^*F \in \Dbcoh(D)$ we have an inequality $\mylength(j^*F) \geq d$.
    \item If the tangent vector to the curve  $D$ at $p$ lies in the tangent cone of the curve $C$ at $p$, then $\mylength(j^*F) > d$.
    \item If $F \iso \mathcal{F}[0]$ is a single coherent sheaf, then the bounds above hold for $\mylength(L_0 j^*\mathcal{F})$.
  \end{enumerate}
\end{lemma}

\begin{proof}
  Since $D \subset U$ is a Cartier divisor, by Lemma~\ref{lem: derived restriction to a divisor} we know that
  \[
    \mylength(j^*F) = \sum_{n \in \Z} \mylength(j^*\mH^n(F)).
  \]
  Thus we may replace the object $F$ with the coherent sheaf $\bigoplus_{n \in \Z} \mH^n(F)$ without changing the lengths of the slices. Thus it is enough to prove the bounds for the length of the nonderived pullback $L_0 j^* \mathcal{F}$ of a coherent sheaf $\mathcal{F}$ on $U$.

  In the proof we use the notion of a (zeroth) Fitting ideal of a coherent sheaf. Recall the definition: given a finitely generated module $M$ over a Noetherian ring $A$, pick an arbitrary free presentation:
  \[
    A^{k} \xrightarrow{Q} A^n \to M \to 0.
  \]
  The Fitting ideal $\Fit(M)$ is defined to be the ideal of $A$ generated by the $(n \times n)$-minors of the matrix $Q$. This construction globalizes to coherent sheaves. The Fitting ideal is contained in the annihilator ideal, and the formation of Fitting ideals is compatible with arbitrary base change (see, e.g., \cite[Tag~07Z6]{stacks-project}).

  Consider the coherent sheaf $L_0 j^*\mathcal{F}$ on a curve $D$. It is supported at a single point $p \in D$. Let $\mathfrak{m} \subset \O_D$ be the maximal ideal sheaf of the point $p$. Using Lemma~\ref{lem: smooth torsion sheaves on curves} it is easy to compute that for any coherent sheaf on a smooth curve $D$ supported at the point $p$ the Fitting ideal is equal to $\mathfrak{m}^{\mylength}$, where $\mylength$ is the length of the torsion sheaf. Thus in order to bound the length of $L_0 j^*\mathcal{F}$ it is enough to understand the Fitting ideal of this sheaf.

  By passing to an \'etale neighborhood of the point $p \in U$ we may assume that $U$ is the affine plane $\A^2 = \Spec \k[x, y]$ and $p$ is the origin. Let $f \in \k[x, y]$ be the reduced equation of the curve $C$. Since the set-theoretic support of $\mathcal{F}$ is $C$, we know that the annihilator ideal of the sheaf $\mathcal{F}$ is contained in the ideal $(f)$. The Fitting ideal $\Fit(\mathcal{F})$ is contained in the annihilator ideal, so $\Fit(\mathcal{F}) \subset (f)$.

  Since Fitting ideals are compatible with base change, we know that $\Fit(L_0 j^* \mathcal{F})$ is contained in the restriction of the ideal $(f)$ to the curve $D$. The pullback of $(f)$ is contained in $\mathfrak{m}^d$, where $d$ is the multiplicity of $C$ at $p$, which is the lowest degree of a monomial occuring in $f$ with nonzero coefficient. Thus $\mylength(L_0 j^*\mathcal{F}) \geq d$. Moreover, if the tangent vector to $D$ lies in the tangent cone of $C$ at the point $p$, by definition this means that the degree-$d$ part of the polynomial $f$ restricts to zero in the quotient $\mathfrak{m}^{d}/\mathfrak{m}^{d+1}$. Thus in this case the pullback of $(f)$ to the curve $D$ is contained in $\mathfrak{m}^{d+1}$, and then $\mylength(L_0 j^*\mathcal{F}) > d$, as claimed.
\end{proof}

Motivated by the bound above, we introduce the following definition.

\begin{definition}
  \label{def: thin objects}
  Let $U$ be a smooth surface, and let $F \in \Dbcoh(U)$ be an object whose set-theoretic support is a reduced curve $C \subset U$. We say that the object $F$ is \emph{thin at the point $p \in C$} if there exists a slice of $F$ at the point $p$ which is a torsion object of the length equal to the multiplicity of the curve $C$ at $p$.
\end{definition}

\begin{lemma}
  \label{lem: thin coherent sheaves}
  Let $U$ be a smooth surface, let $C \subset S$ be a reduced curve, and let $p \in C$ be a point. Let $\mathcal{F}$ be a coherent sheaf on $U$ whose set-theoretic support is $C \subset S$. Suppose that $\mathcal{F}$ is thin at the point $p \in C$. Then, after possibly replacing $U$ by a Zariski neighborhood of $p \in U$, the sheaf $\mathcal{F}$ is a pushforward of a torsion-free rank one sheaf $\mathcal{F}^\prime$ on $C$.
\end{lemma}

\begin{proof}
  The foundational case is when the multiplicity of the curve $C$ at the point $p$ is equal to one. Then by definition $\mathcal{F}$ is thin at $p$ if and only if there exists a slice of $\mathcal{F}$ of length one. This case is proved in Lemma~\ref{lem: equicutting lemma}. Otherwise, let $d > 1$ be the multiplicity of the curve $C$ at the point $p$.

  By shrinking $U$ we may assume that $U$ is affine. Let $f \in \mathrm{H}^0(\O_U)$ be the equation for the reduced scheme structure on the curve $C$. Since the characteristic of the base field is zero, by further shrinking $U$ we may assume that all points in $C \setminus \{ p \}$ are smooth in the curve $C$. Let $j\colon D \monoarrow S$ be a smooth curve passing through $p$ such that the derived pullback $j^*\mathcal{F}$ is a torsion object of the length $d$.

  We first show that $\mathcal{F}$ has no point-torsion at $p$. We know that the length
  \[
    \mylength(j^*\mathcal{F}) := \mylength(L_0 j^*\mathcal{F}) + \mylength(L_1 j^* \mathcal{F})
  \]
  is equal to $d$. By Lemma~\ref{lem: thin slices are minimal} the summand $\mylength(L_0 j^*\mathcal{F})$ is greater or equal to $d$. Thus $L_1 j^* \mathcal{F}$ has length zero, so it is a zero object. Consider the subsheaf $\mathcal{T} \subset \mathcal{F}$ spanned by sections supported only at the point $p$. Consider the short exact sequence
  \[
    0 \to \mathcal{T} \to \mathcal{F} \to \mathcal{F}/\mathcal{T} \to 0.
  \]
  Since $j$ is an inclusion of a Cartier divisor, $L_2 j^*(-)$ vanishes at every argument. Thus the long exact sequence of derived pullbacks along $j\colon D \monoarrow S$ shows that $L_1 j^* \mathcal{F}$ has a subsheaf isomorphic to $L_1 j^* \mathcal{T}$. If $\mathcal{T}$ is a nonzero sheaf, then by Lemma~\ref{lem: lines detect point-torsion} the sheaf $L_1 j^*\mathcal{T}$ is also nonzero, but this leads to a contradiction with the fact that $\mathcal{F}$ is thin at $p$. Thus $\mathcal{T} = 0$, i.e., the sheaf $\mathcal{F}$ has no point-torsion.

  Let $g \in \mathrm{H}^0(\O_U)$ be the equation of the smooth curve $D \subset U$. Consider a family of inclusions $j_t\colon D_t \monoarrow U$, where the curve $D_t$ is given by the equation $\{ g = t \}$. By Lemma~\ref{lem: thin slices are minimal} the tangent vector of $D$ at the point $p$ does not lie in the tangent cone of $C$, and thus for a general value of $t$ the curve $D_t$ intersects $C$ transversely in exactly $d$ distinct points (see, e.g., \cite[\S 5A]{mumford-algeom}; we use the assumption of characteristic zero here). Thus, after possibly shrinking $U$, by semicontinuity we may assume that at each point of $C \setminus \{ p \}$ the sheaf $\mathcal{F}$ has a slice which is a torsion object of length one (see Figure~\ref{fig: semicontinuity and slices}).

  
  \begin{figure*}[ht]
    \centering
    \def\svgwidth{0.3\columnwidth}
    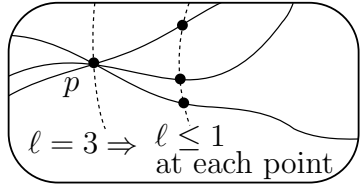
    \caption{Semicontinuity and slices}
    \label{fig: semicontinuity and slices}
  \end{figure*}

  In particular, by Lemma~\ref{lem: global equicutting lemma} this implies that on $U \setminus \{ p \}$ the sheaf $\mathcal{F}$ is a direct sum of pushforwards of line bundles from the irreducible components of $C \setminus \{ p \}$. Assume that on the unpunctured surface $U$ the sheaf $\mathcal{F}$ is not a pushforward from the curve $C$. Since $U$ is affine, this is equivalent to the fact that the equation $f \in \mathrm{H}^0(\O_U)$ does not annihilate $\mathcal{F}$. Then there exists a section $s \in \mathrm{H}^0(\mathcal{F})$ such that $f \cdot s$ is not zero. The equation $f$ annihilates any section on the open set $U \setminus \{ p \}$. Therefore $f \cdot s$ is a section of $\mathcal{F}$ which is supported only at a single point $p$. But we proved that $\mathcal{F}$ has no zero-dimensional torsion, a contradiction. Thus the scheme-theoretic support of the sheaf $\mathcal{F}$ is equal to $C$.
\end{proof}

\begin{lemma}
  \label{lem: thin objects}
  Let $S$ be a smooth surface, and let $F \in \Dperf(S)$ be an object whose set-theoretic support is a reduced curve $C \subset S$. Suppose that $F$ is thin at every point of $C$. Then $F$ is a formal complex, and each cohomology sheaf $\mH^n(F)$ is isomorphic to a pushforward of a torsion-free rank one sheaf from some subcurve $C_n \subset C$.
\end{lemma}

\begin{proof}
  Let $p \in C$ be a point, and let $j\colon D \monoarrow S$ be a smooth curve passing through $p$ such that $j^*F$ is a torsion object of length equal to the multiplicity of $C$ at $p$. By Lemma~\ref{lem: derived restriction to a divisor} we know that $\mylength(j^*F) = \sum_n \mylength(j^*\mH^n(F))$. Let $C_n := \supp(\mH^n(F)) \subset C$ be the set-theoretic support of the $n$'th cohomology sheaf. Denote by $I \subset \Z$ the subset of those indices $n \in \Z$ such that $C_n$ contains $p$ and $p$ is not an isolated point in $C_n$. For $n \in I$, let $m_n$ be the multiplicity of $C_n$ at the point $p$. By Lemma~\ref{lem: thin slices are minimal} we have $\mylength(j^*\mH^n(F)) \geq m_n$ for each $n \in I$. Let $m$ be the multiplicity of the curve $C$ at $p$. Since $C = \cup_{n \in I} C_n$ near the point $p$ set-theoretically, we have $\sum_{n \in I} m_n \geq m$. Taking all this information into account, we get a chain of inequalities:
  \[
    \mylength(j^*F) = \sum_{n \in \Z} \mylength(j^* \mH^n(F)) \geq
    \sum_{n \in I} \mylength(j^* \mH^n(F)) \geq \sum_{n \in I} m_n \geq m.
  \]
  The assumption that $F$ is thin at $p$ implies that each inequality is in fact an equality. Note that this holds for any point $p \in C$. Thus we conclude that:
  \begin{enumerate}
    \item For any $n \in \Z$ such that $\mH^n(F) \neq 0$, the subset $C_n = \supp(\mH^n(F))$ is a curve, and the sheaf $\mH^n(F)$ is thin at any point of its support $C_n$.
    \item For any two distinct $n, n^\prime \in \Z$ the intersection $C_n \cap C_{n^\prime}$ is a zero-dimensional set.
  \end{enumerate}

  Consider a nonzero cohomology sheaf $\mH^n(F)$. It is thin at every point of the curve $C_n$, so by Lemma~\ref{lem: thin coherent sheaves} the sheaf $\mH^n(F)$ is isomorphic to a pushforward of a torsion-free rank one sheaf from $C_n$. It only remains to show that $F$ is a formal complex.

  By Lemma~\ref{lem: complexes on a surface are given by adjacent glueings} the glueing data for $F$ consists of classes in $\Ext^2(\mH^n(F), \mH^{n-1}(F))$ for each~$n \in \Z$. Since the supports $C_n \cap C_{n - 1}$ intersect along a zero-dimensional set, the $\Ext$-group may be computed locally at each intersection point, i.e.
  \[
    \Ext^2(\mH^n(F), \mH^{n-1}(F)) = \mathrm{H}^0(\mathcal{E}xt^2(\mH^n(F), \mH^{n-1}(F)).
  \]
  Let $p \in S$ be any point. Since $\mH^n(F)$ is a pushforward of a torsion-free sheaf from a curve via an inclusion $C_n \monoarrow S$, it has no point-torsion at $p$. By definition this means that the depth of the coherent sheaf $\mH^n(F)$ at $p \in S$ is not zero. Since $S$ is a smooth surface, Auslander--Buchsbaum formula implies that the projective dimension of $\mH^n(F)$ over the local ring $\O_{S, p}$ is at most one. Since this true for every point $p \in S$, the local $\Ext$-sheaf~$\mathcal{E}xt^2(\mH^n(F), -)$ vanishes for any second argument. Therefore $\Ext^2(\mH^n(F), \mH^{n-1}(F)) = 0$. This shows that complex $F$ splits into a direct sum of its cohomology sheaves, and the lemma is proved.
\end{proof}

\subsection{Rational elliptic surfaces}
\label{ssec: rational elliptic surfaces}

We are interested in the following class of surfaces.

\begin{definition}
  \label{def: rational elliptic surface}
  A smooth proper surface $S$ is a \emph{rational elliptic surface} if the anticanonical linear system $| {-K_S} |$ defines a regular morphism $\pi\colon S \to \P^1$.
\end{definition}

There are other, sometimes incompatible, definitions of rational elliptic surfaces in the literature. We follow the terminology of \cite{heckman-looijenga}. Note that in this convention the elliptic fibration $\pi\colon S \to \P^1$ is relatively minimal, i.e., there are no smooth $(-1)$-curves in the fibers. See the reference for more information about this class of surfaces.

Since the anticanonical divisors of a rational elliptic surface $S$ are nothing but the fibers of the morphism $\pi\colon S \to \P^1$, the methods of Section~\ref{sec: numerical lemmas} may be productively used to study admissible subcategories in rational elliptic surfaces. In this subsection we study phantom subcategories using Proposition~\ref{prop: numerical lemma without numbers}, the notion of a point-support introduced in Subsection~\ref{ssec: point-supports}, and the observations about objects set-theoretically supported on curves from Subsection~\ref{ssec: cutting lemmas}. The main result is Theorem~\ref{thm: rational elliptic surfaces}.

We start by a couple of general observations about phantom subcategories and objects with zero-dimensional support.

\begin{lemma}
  \label{lem: zero-dimensional is not a phantom}
  Let $X$ be a smooth projective variety, and let $\mB \subset \Dbcoh(X)$ be an admissible subcategory. If $\mB$ contains an object with zero-dimensional support, then $\mB$ is not a phantom subcategory.
\end{lemma}

\begin{proof}
  Let $G \in \mB$ be an object with zero-dimensional support. Let $i \in \Z$ be the smallest integer such that $\mH^i(G) \neq 0$. Then $\mH^i(G)$ is a coherent sheaf with zero-dimensional support, and in particular there exists an embedding $\O_q \monoarrow \mH^i(G)$ of a skyscraper sheaf $\O_q$ at some point $q \in X$ into that coherent sheaf. Then the composition $\O_q \monoarrow \mH^i(G) \to G$ induces a nonzero map on cohomology sheaves. Thus by Lemma~\ref{lem: maps from torsion cannot be nonzero on cohomology} we see that any object of $\mB^\perp$ is supported on the complement $X \setminus \{ q \}$. In particular, the rank of any object of $\mB^\perp$ at the generic point of $X$ is zero. This shows that $K_0(\mB^\perp)$ cannot be equal to $K_0(X)$. Since there exists a direct sum decomposition $K_0(X) = K_0(\mB^\perp) \oplus K_0(\mB)$, we see that $K_0(\mB) \neq 0$.
\end{proof}

\begin{lemma}
  \label{lem: splitting from single sheaf}
  Let $X$ be a smooth projective variety, and let $\mB \subset \Dbcoh(X)$ be an admissible subcategory. Let $B \in \mB$ be an object. Suppose that $\mH^n(B)$ for some $n \in \Z$ is the only cohomology sheaf of $B$ which is supported at a point $p \in X$. If $p \in \supp(\mH^n(B))$ is an isolated point, then $\mB$ is not a phantom subcategory.
\end{lemma}

\begin{proof}
  Assume that $p \in \supp(\mH^n(B))$ is an isolated point. Since all other cohomology sheaves of $B$ are supported away from $p$, this point is still isolated in the union
  \[
    \supp(B) := \cup_{i \in \Z} \supp(\mH^i(B)).
  \]
  Then by Lemma~\ref{lem: disjoint support forces splitting} the object $B \in \mB$ has a direct summand $T$ with $\supp(T) = \{ p \}$. Since admissible subcategories are closed under direct summands, we have $T \in \mB$. Then by Lemma~\ref{lem: zero-dimensional is not a phantom} the category $\mB$ is not a phantom.
\end{proof}

We continue with more specific properties of phantom subcategories in rational elliptic surfaces. As in Section~\ref{sec: projective plane}, we use the projections of skyscraper sheaves to study admissible subcategories. We fix the notation used in the rest of this subsection below.

\begin{setting}
  \label{set: rational elliptic surfaces}
  Let $\pi\colon S \to \P^1$ be a rational elliptic surface. Let $\mB \subset \Dbcoh(S)$ be an admissible subcategory. Pick a point $p \in S$ lying on the fiber $F \subset S$ of the projection map $\pi\colon S \to \P^1$ such that
  \begin{enumerate}
  \item $F$ is irreducible;
  \item $p \in F$ is a smooth point of that curve.
  \end{enumerate}
  Let $B := \mB_R(\O_p)$ be the (right) projection of the skyscraper sheaf at the point $p$ to the subcategory $\mB$. Denote by $S_p := \supp(B) \subset S$ the set-theoretic support of $B$.
\end{setting}

\begin{lemma}
  \label{lem: fiber and point-support}
  Let $p, \mB, B, F$ be as in Setting~\textup{\ref{set: rational elliptic surfaces}}. If $\mB$ is a phantom subcategory, then the restriction $B|_F$ is isomorphic to either
  \begin{itemize}
  \item a zero object, in which case $B = 0$ and $\O_p \in \mB^\perp$; or
  \item to a direct sum $\O_p[0] \oplus \O_q[a]$ for some smooth point $q \in F$, possibly coinciding with the point $p$, and some odd shift $a \in \Z$.
  \end{itemize}
\end{lemma}

\begin{proof}
  Consider the restriction $B|_F \in \Dperf(F)$. The fiber $F$ is an anticanonical divisor of the surface $S$. Note that $F$ is a reduced irreducible projective curve of arithmetic genus one. Since $p \in F$ is a smooth point, the skyscraper sheaf $\O_p$ is a perfect object on $F$. Then we can apply Proposition~\ref{prop: numerical lemma without numbers} to see that the object $B|_F$ is either
  \begin{enumerate}
  \item a cone between two skyscrapers on smooth points; or
  \item a cone of a map from the skyscraper $\O_p$ to a simple vector bundle on $F$.
  \end{enumerate}

  In the case (2) the object $B|_F$ has a nonzero rank at the generic point of $F$. Then the class of $B|_F$ in $K_0(F)$ is nonzero, and hence the class of $B \in \mB$ is $K_0(S)$ is also nonzero. This shows that $\mB$ is not a phantom subcategory, a contradiction.

  Suppose we are in the case (1). Then by Corollary~\ref{cor: autoequivalence triangle} the object $B|_F$ fits into a distinguished triangle
  \[
    B|_F \to \O_p[0] \to \O_q[n]
  \]
  where $q \in F$ is some smooth point and $n \in \Z$ is some shift. If $n$ is an odd integer, then the class of $B|_F$ in $K_0(F)$ is equal to the class of $\O_p \oplus \O_q$, a nonzero element, and this gives a contradiction with the assumption that $\mB$ is a phantom subcategory. Thus we see that $n$ is necessarily an even integer.

  If the morphism $\O_p \to \O_q[n]$ is zero, then $B|_F \iso \O_p \oplus \O_q[n-1]$, as claimed. For even~$n$, the morphism of skyscrapers can be nonzero only when $q = p$, $n = 0$, and the map is an isomorphism $\O_p[0] \isoarrow \O_p[0]$. In this case the cone $B|_F$ is a zero object. Then the intersection of the support $\supp(B)$ with $F$ is empty. In particular, $\supp(B)$ does not contain the point~$p$. Since $B$ is the projection of the skyscraper sheaf, by Lemma~\ref{lem: properties of point-supports} this implies that $B = 0$ and hence the skyscraper sheaf $\O_p$ lies in the orthogonal subcategory $\mB^\perp$, as claimed.
\end{proof}

\begin{lemma}
  \label{lem: phantoms like minus one curves}
  Let $p, \mB, B, F$ be as in Setting~\textup{\ref{set: rational elliptic surfaces}}. If $\mB$ is a phantom subcategory and $B \neq 0$, then the point $p$ lies on some section of the map $\pi\colon S \to \P^1$.
\end{lemma}
\begin{proof}
  For any torsion object $G \in \Dperf(F)$ denote by $\mylength_p(G)$ the sum of lengths of cohomology sheaves of $G$ at the point $p \in F$. By Lemma~\ref{lem: derived restriction to a divisor} we have an equality
  \begin{equation}
    \label{eqn: additivity of lengths}
    \mylength_p(B|_F) = \sum_{i \in \Z} \mylength_p (\mH^i(B)|_F).
  \end{equation}

  From Lemma~\ref{lem: fiber and point-support} we know that if $B \neq 0$, then $B|_F$ is isomorphic to a direct sum of two skyscrapers $\O_p[0] \oplus \O_q[a]$ where $a$ is an odd integer.

  Consider first the case where the point $q$ is distinct from the point $p$. Then $\mylength_p(B|_F) = 1$. From the formula (\ref{eqn: additivity of lengths}) we see that there exists exactly one cohomology sheaf $\mH^n(B)$ which is nonzero at the point $p$, and moreover the length $\mylength_p(\mH^n(B)|_F)$ is equal to one. By Lemma~\ref{lem: splitting from single sheaf} we see that $p \in \supp(\mH^n(B))$ is not an isolated point, so there exists an irreducible curve $C \subset \supp(\mH^n(B))$ passing through $p$.

  Suppose that $C$ also intersects the fiber $F$ at some point $r \in F$ other than $p$. Then by Lemma~\ref{lem: derived restriction to a divisor} this implies that $\mH^n(B|_F)$ is not zero at both points $p$ and $r$. But this contradicts the fact that $B|_F \iso \O_p[0] \oplus \O_q[a]$ with odd $a$, and each cohomology sheaf is supported only at a single point. Therefore $C$ intersects $F$ only at the point $p$.

  Moreover, since $\mylength_p(\mH^n(B)|_F) = 1$, by Lemma~\ref{lem: thin slices are minimal} we see that $C$ intersects $F$ at $p$ transversely, i.e., with multiplicity one. Since $F$ is a fiber of the fibration $\pi\colon S \to \P^1$, this implies that the composition $C \monoarrow S \xrightarrow{\pi} \P^1$ is a morphism of curves of degree one into a smooth curve. Any such map is necessarily an isomorphism. Thus $C \iso \P^1 \subset S$ is a section of the map $\pi$, and it contains the point $p$ by construction, so the lemma is proved in the case where $q \neq p$.

  Consider now the case where $B|_F \iso \O_p[0] \oplus \O_p[a]$ is supported only at the point $p$. Here the length $\mylength_p(B|_F)$ is two. As above, from the formula (\ref{eqn: additivity of lengths}) we see that there are at most two cohomology sheaves of $B$ which are nonzero at $p$. If there are two distinct ones, $\mH^i(B)$ and $\mH^j(B)$ with $i \neq j$, then we have
  \[
    \mylength_p(\mH^i(B)|_F) = \mylength_p(\mH^j(B)|_F) = 1.
  \]
  In this case we can repeat the argument above using either of those two cohomology sheaves.

  Thus it only remains to deal with the case where only one cohomology sheaf, $\mH^n(B)$, is nonzero at $p$ and has $\mylength_p(\mH^n(B)) = 2$. Again, using Lemma~\ref{lem: splitting from single sheaf} we can pick an irreducible curve $C \subset \supp(\mH^n(B))$ containing the point $p$. It intersects $F$ only at the point $p$ since $\supp(B|_F) = \{ p \}$. If $C$ intersects $F$ at $p$ with multiplicity one, then it is a section of the fibration $\pi\colon S \to \P^1$, and we are done by the same argument as above. If the intersection multiplicity is $m > 1$, then either $C$ is singular at $p$, or it intersects $F$ non-transversely. In both those cases from the last part of Lemma~\ref{lem: thin slices are minimal} we see that the length of the nonderived restriction of $\mH^n(B)$ along the inclusion $j\colon F \monoarrow S$ is at least~$2$. From the short exact sequences in Lemma~\ref{lem: derived restriction to a divisor} we see that
  \[
    \mylength_p(\mH^n(B|_F)) \geq \mylength_p(L_0 j^* \mH^n(B)) \geq 2.
  \]
  But this is a contradiction with the fact that $B|_F \iso \O_p[0] \oplus \O_p[a]$ with odd $a \in \Z$, and each individual cohomology sheaf has length at most one. Therefore, the intersection of $C$ with $F$ is always transverse, and the lemma is proved.
\end{proof}

\statemaintheoremRES

\begin{remark}
  Note that any section is a $(-1)$-curve by the adjunction formula. Moreover, each smooth $(-1)$-curve is a section, and by Kodaira's classification of singular fibers each reducible fiber is a union of several $(-2)$-curves. In other words, the conclusion is that any phantom subcategory is supported on the union of $(-1)$-curves and $(-2)$-curves.
\end{remark}

\begin{proof}
  Let $p \in S$ be a smooth point of an irreducible fiber $F \subset S$, and suppose that $p$ does not lie on any section of $\pi\colon S \to \P^1$. Then by Lemmas~\ref{lem: fiber and point-support} and~\ref{lem: phantoms like minus one curves} we see that the projection of the skyscraper sheaf $\O_p$ to $\mB$ is zero, so $\O_p$ lies in the subcategory $\mB^\perp$. By semiorthogonality this implies that any object of $\mB$ is supported on the complement $S \setminus \{ p \}$. Since this holds for any smooth point of an irreducible fiber, we see that the support of $\mB$ is contained in the union of
  \begin{itemize}
  \item sections of the fibration $\pi\colon S \to \P^1$;
  \item reducible fibers of the fibration $\pi\colon S \to \P^1$; and
  \item singular points of irreducible fibers of $\pi$.
  \end{itemize}
  
  Note that there are finitely many irreducible singular fibers of the morphism $\pi$, and each has only a single singular point. So singular points of irreducible fibers form a discrete subset of the surface $S$. By Lemma~\ref{lem: support of a category} the support of $\mB$ does not contain any isolated points. Thus the subset $\supp(\mB)$ is in fact a union of curves, as claimed in the statement.
\end{proof}

\subsection{Admissible subcategories supported on sections}
\label{ssec: sections of anticanonical system}

Theorem~\ref{thm: rational elliptic surfaces} shows that any phantom subcategory in a rational elliptic surface is supported on the union of sections and reducible fibers. Recall that each section is a $(-1)$-curve in a rational elliptic surface, and in Section~\ref{sec: minus one curves} we have classified admissible subcategories supported on a single $(-1)$-curve. In this subsection we generalize the classification of admissible subcategories to arbitrary configurations of sections.

The main result of this subsection is Theorem~\ref{thm: admissible subcategories on sections}. We show that the only option is to choose several disjoint sections, pick a line bundle on each one, and span the subcategory by their pushforwards. As for the classification of admissible subcategories on a projective plane in Section~\ref{sec: projective plane}, the key tool is Proposition~\ref{prop: numerical lemma without numbers}. The most difficult part is to deal with configurations of two intersecting sections. We use various lemmas from Subsection~\ref{ssec: cutting lemmas} to handle this situation.

\begin{setting}
  \label{set: sections}
  Let $\pi\colon S \to \P^1$ be a rational elliptic surface. Let $\mB \subset \Dbcoh(S)$ be an admissible subcategory such that $\supp(\mB)$ is contained in the union of sections of the map $\pi$. Denote the support by $C := \supp(\mB)$. Note that by Lemma~\ref{lem: support of a category} the subset $C$ is a curve, i.e., a union of finitely many sections of the fibration $\pi$.

  Define a subset $C^\circ \subset C$ by saying that a point $c \in S$ lies in the complement $C \setminus C^\circ$ if the fiber $F_c$ of the fibration $\pi\colon S \to \P^1$ containing $c$ is singular or if it contains an intersection point of two distinct sections in $C$. Note that all points in $C^\circ$ are smooth in the curve $C$.

  \begin{figure*}[ht]
    \centering
    \def\svgwidth{0.3\columnwidth}
    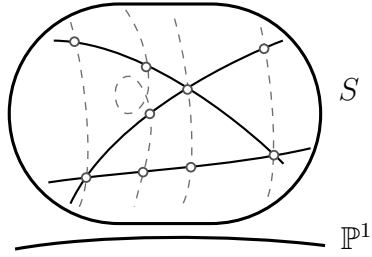
    \caption{The subset $C^\circ \subset C$ is the complement to the punctured points.}
  \end{figure*}
  
  We also fix a point $p \in C^\circ$ and denote by $S_p \subset S$ the point-support of the admissible subcategory $\mB \subset \Dbcoh(S)$ at the point $p$. Denote by $B \in \mB$ the (right) projection $\mB_R(\O_p)$ of the skyscraper sheaf at the point $p$ to the subcategory $\mB$.
\end{setting}

\begin{lemma}
  \label{lem: projection from support is nonzero}
  Let $S, \mB, p, B$ be as in Setting~\textup{\ref{set: sections}}. Then $B \neq 0$.
\end{lemma}
\begin{proof}
  Assume that $B$, the projection of the skyscraper sheaf $\O_p$ to $\mB$, is a zero object. Then the skyscraper sheaf $\O_p$ lies in the orthogonal subcategory $\mB^\perp$. By Lemma~\ref{lem: maps from torsion cannot be nonzero on cohomology} this implies that $\supp(\mB)$ does not contain the point $p$, but this is a contradiction, since by assumption the point $p$ lies in $C^\circ \subset C = \supp(\mB)$.
\end{proof}

\begin{lemma}
  \label{lem: at most two irreducible components}
  Let $S, \mB, C, C^\circ, p, S_p$ be as in Setting~\textup{\ref{set: sections}}. Then $S_p$ is a curve with at most two irreducible components.
\end{lemma}

\begin{proof}
  By Lemma~\ref{lem: projection from support is nonzero} the projection $B$ of the skyscraper sheaf $\O_p$ to $\mB$ is nonzero and we have $S_p = \supp(B)$. Since $\supp(\mB)$ is not the whole surface $S$, the subset $S_p$ cannot be equal to the point $\{ p \}$, as in that case we could deform $B$ by moving the point $p$ along the surface, and admissible subcategories are closed under small deformations by Proposition~\ref{prop: admissible subcategories are open}. Then by Lemma~\ref{lem: properties of point-supports} we know that $\supp(B) = S_p \subset \supp(\mB)$ is a connected union of several sections of the map $\pi\colon S \to \P^1$.

  Let $F \subset S$ be the fiber of the fibration morphism $\pi\colon S \to \P^1$ containing the point $p$. By definition each section of the map $\pi$ intersects $F$ at a single point transversely. Since by assumption $p \in C^\circ$, we know that the intersection points of $F$ with sections in $C$ are all distinct. Therefore the number of irreducible components of $S_p$ equals the cardinality of the set-theoretic intersection $S_p \cap F$, which is the same as $\supp(B|_F)$. Since $p \in C^\circ$ by assumption, the curve $F$ is a reducible irreducible anticanonical divisor of the surface $S$, so by Proposition~\ref{prop: numerical lemma without numbers} we know that $B|_F$ is a cone of a map between two skyscrapers. In particular, the curve~$S_p$ has at most two irreducible components.
\end{proof}

The next several lemmas prove by contradiction that $S_p$ cannot have exactly two irreducible components. The overall plan is to deform the curve $S_p$ into a curve  that does not lie in the closed subset $\supp(\mB) \subset S$, and prove that we may also deform the object $B \in \Dbcoh(S)$, set-theoretically supported on $C_1 \cup C_2$, along with the curve. This would be a contradiction with the fact that admissible subcategories are closed under small deformations. It is impossible to control the deformations of an object in the derived category if the only thing we know about it is its set-theoretic support, so in order to realize our plan we study other properties of the object $B$ in the following lemmas.

\begin{lemma}
  \label{lem: unions of minus one curves}
  Let $S$ be a smooth projective surface. Let $C_1$ and $C_2$ be two distinct smooth $(-1)$-curves on $S$. Let $m = C_1 \cdot C_2$ be the intersection number, counted with multiplicity. Then the dimension $\dim H^0(C_1 \cup C_2, N_{C_1 \cup C_2 / S})$ of the space of global sections of the normal bundle to the union $C_1 \cup C_2$ is at least $m$.
\end{lemma}

\begin{proof}
  If $m = 0$, there is nothing to prove, so we assume that $C_1$ and $C_2$ are not disjoint. By the adjunction formula the normal bundle $N_{C_1 \cup C_2 / S}$ to the divisor $C_1 \cup C_2 \subset S$, restricted to a subcurve $C_i$ for $i = 1, 2$, has degree $C_i \cdot (C_1 + C_2) = m-1$. Since both components of $C_1 \cup C_2$ are isomorphic to $\P^1$, on each component the dimension of the space $H^0(C_i, (N_{C_1 \cup C_2 / S})|_{C_i})$ is equal to $m$. Since the intersection multiplicity $C_1 \cdot C_2$ is $m$, there are $m$ linear conditions on those sections in order for them to glue. Thus there are at least $2m - m = m > 0$ global sections of the normal bundle $N_{C_1 \cup C_2 / S}$.
\end{proof}

Recall the notion of a thin object from Definition~\ref{def: thin objects}.

\begin{lemma}
  \label{lem: two components imply thinness}
  Let $S, \mB, C, C^\circ, p, S_p$, and $B$ be as in Setting~\textup{\ref{set: sections}}. If $S_p = C_1 \cup C_2$ with irreducible components $C_1$ and $C_2$, then $S_p$ is a nodal curve and $B$ is a thin object at each point of its support $S_p$.
\end{lemma}

\begin{proof}
  Consider the fiber $j\colon F \to S$ of the fibration $\pi\colon S \to \P^1$ which contains the point $p$. By the assumption that $p \in C^\circ$ we know that $F$ is smooth and the intersection $F \cap S_p$ is $\{ p, q \}$ for some $q \neq p$. Since $F$ is an anticanonical divisor of the surface $S$, from Proposition~\ref{prop: numerical lemma without numbers} we get that the only option for the object $j^* B$ to be supported on two smooth points of the curve $F$ is the direct sum $\O_p[0] \oplus \O_q[a]$ for some integer $a \in \Z$.

  Since both components $C_1$ and $C_2$ of $S_p$ are sections of the map $\pi\colon S \to \P^1$, this implies that the intersection of each $C_i$ with an arbitrary fiber $F^\prime \subset S$ has multiplicity one, in particular the intersection $S_p \cap F^\prime$ consists of smooth points of the fiber $F^\prime$ by Lemma~\ref{lem: thin slices are minimal}.

  Now we can apply Lemma~\ref{lem: arbitrary anticanonical divisors} to see that the restriction of $B$ to any other fiber $F^\prime$ is a torsion object of length at most two. Thus, if the fiber $F^\prime$ intersects $S_p$ in two distinct points, each intersection point has length exactly one, and if both curves $C_1$ and $C_2$ intersect $F^\prime$ at the same point, then this is a torsion object of length two.

  Since each intersection point of $C_1$ and $C_2$ has multiplicity at least $2$, Lemma~\ref{lem: thin slices are minimal} applied to the intersections $S_p \cap F^\prime$ for all fibers $F^\prime$ implies that
  \begin{enumerate}
  \item each singular point of $C_1 \cup C_2$ has multiplicity exactly $2$, i.e., it is a node;
  \item the object $B$ is thin at each point of its support $C_1 \cup C_2$.
  \end{enumerate}
  This finishes the proof of the lemma.
\end{proof}

\begin{lemma}
  \label{lem: two components exts}
  Let $S, \mB, C, C^\circ, p, S_p$, and $B$ be as in Setting~\textup{\ref{set: sections}}. If $S_p = C_1 \cup C_2$ with irreducible components $C_1$ and $C_2$, then
  \(
    \REnd(B) \iso \k[0] \oplus \k[-1].
  \)
\end{lemma}
\begin{proof}
  By Lemma~\ref{lem: two components imply thinness} we know that the object $B$ is thin everywhere, in particular at the point $p \in S_p$. Since by assumption $p \in C^\circ$, this is a smooth point of $S_p$. Thus by Lemma~\ref{lem: equicutting lemma} in a Zariski neighborhood of the point $p \in S$ the object $B$ is isomorphic to the structure sheaf of the smooth curve passing through $p$ transversely. Using this local description, we compute
  \[
    \RHom(B, \O_p) \iso \k[0] \oplus \k[-1].
  \]

  Since the object $B$ is the (right) projection of the skyscraper sheaf $\O_p$, by the universal property of the projections (Corollary~\ref{cor: universal property of projection triangles}) we see that, as claimed in the statement,
  \begin{equation}
    \label{eqn: two components endomorphisms}
    \RHom(B, B) \iso \RHom(B, \O_p) \iso \k[0] \oplus \k[-1]. \qedhere
  \end{equation}
\end{proof}

\begin{lemma}
  \label{lem: two components implies torsion-free sheaf}
  Let $S, \mB, C, C^\circ, p, S_p$, and $B$ be as in Setting~\textup{\ref{set: sections}}. If $S_p = C_1 \cup C_2$ with irreducible components $C_1$ and $C_2$, then $B$ is isomorphic to a pushforward $i_*\mathcal{F}$ of a simple rank one torsion-free sheaf $\mathcal{F} \in \Coh(S_p)$ along the inclusion $i\colon S_p \monoarrow S$.
\end{lemma}

\begin{proof}
  Lemma~\ref{lem: two components imply thinness} shows that $B$ is thin at each point of its support. We described the general structure of thin objects in Lemma~\ref{lem: thin objects}. This lemma shows that $B$ is a formal complex and each cohomology sheaf is a rank one torsion-free sheaf on (a subcurve of) $S_p$. However, by Lemma~\ref{lem: two components exts} we know that the object $B$ is simple, in particular indecomposable. Thus there exists only one cohomology sheaf and hence $B \iso i_*\mathcal{F}$ for some torsion-free rank one sheaf on $S$. Since $i\colon S_p \monoarrow S$ is a closed embedding, we have
  \[
    \mathrm{R}^0\Hom_{S_p}(\mathcal{F}, \mathcal{F}) \iso \mathrm{R}^0\Hom_S(i_*\mathcal{F}, i_*\mathcal{F}) = \mathrm{R}^0\Hom_S(B, B) \iso \k.
  \]
  Thus the torsion-free rank one sheaf $\mathcal{F}$ on $S_p$ is simple.
\end{proof}

\begin{lemma}
  \label{lem: two components imply line bundle}
  Let $S, \mB, C, C^\circ, p, S_p$, and $B$ be as in Setting~\textup{\ref{set: sections}}. If $S_p = C_1 \cup C_2$ with irreducible components $C_1$ and $C_2$, then the category $\mB$ contains a pushforward $i_*L[0]$ of some line bundle $L \in \Pic(S_p)$ along the inclusion $i\colon S_p \to S$. Moreover, we can assume that
  \[
    \dim \Ext^1(i_*L, i_*L) \leq 1.
  \]
\end{lemma}
\begin{proof}
  By Lemma~\ref{lem: two components implies torsion-free sheaf} we see that $B \iso i_*\mathcal{F}$ for a simple rank one torsion-free sheaf on the curve $S_p$. Since $B \in \mB$ and admissible subcategories are closed under small deformations by Proposition~\ref{prop: admissible subcategories are open}, it is enough to deform the sheaf $\mathcal{F}$ on $S_p$ into a line bundle, and the inequality for the dimension follows from Lemma~\ref{lem: two components exts} and semicontinuity.

  Since $S_p \subset S$ is a connected reduced curve with planar singularities, by \cite[Th.~2.3]{melo-rapagnetta-viviani} the moduli space of simple rank one torsion-free sheaves on $S_p$ has a dense open subset consisting of line bundles. Moreover, by \cite[Fact~2.2]{melo-rapagnetta-viviani} there exists a universal sheaf on the moduli space of simple rank one torsion-free sheaves. Then we can find a family of sheaves on $S_p$ containing $\mathcal{F}$ such that the generic member is a line bundle. Thus the lemma is proved.
\end{proof}

\begin{proposition}
  \label{prop: never two components}
  Let $S, \mB, C, C^\circ, p, S_p$, and $B$ be as in Setting~\textup{\ref{set: sections}}. Then the curve $S_p$ is irreducible, i.e., $S_p$ is a section of the fibration $\pi\colon S \to \P^1$.
\end{proposition}

\begin{proof}
  By Lemma~\ref{lem: at most two irreducible components} we know that $S_p$ has at most two irreducible components. Suppose that $S_p = C_1 \cup C_2$ with two irreducible components $C_1$ and $C_2$. The chain of lemmas above, culminating in Lemma~\ref{lem: two components imply line bundle}, shows that in this case the category $\mB$ contains a pushforward~$i_*L$ of a line bundle $L \in \Pic(S_p)$ along the inclusion $i\colon S_p \to S$. Below we construct a deformation of an object~$i_*L$ into an object whose support does not lie inside the union of the sections of the fibration $\pi\colon S \to \P^1$, and this leads to a contradiction with the assumption on the subset $\supp(\mB) \subset S$ in Setting~\ref{set: sections}.

  Let $N$ be the normal bundle to the curve $S_p = C_1 \cup C_2$ in the surface $S$. Since $i$ is the inclusion of the Cartier divisor $S_p$, by \cite[Cor.~11.4]{HuybFM}\footnote{In the reference the triangle is constructed only for smooth divisors, but the same proof works for perfect objects on an arbitrary divisor.} the object $i^*i_*L$ fits into a distinguished triangle :
  \[
    L \otimes N^\dual[1] \to i^*i_*L \to L.
  \]
  in $\Dperf(S_p)$. An application of the functor $\RHom_{S_p}(-, L)$ produces a long exact sequence of vector spaces. Consider the following fragment:
  \[
    \Ext^{-1}(L \otimes N^\dual, L) \to \Ext^1(L, L) \to \Ext^1(j_*L, j_*L) \to \Hom(L \otimes N^\dual, L) \to \Ext^2(L, L).
  \]
  Since $L$ is a coherent sheaf, there are no negative $\Ext$'s from $L \otimes N^\dual$ to $L$. Since $L$ is a line bundle, we have $\Ext^\bullet(L, L) \caniso H^\bullet(\O_{S_p})$. In particular, $\Ext^2(L, L) = 0$. Thus we obtain a short exact sequence
  \[
    0 \to H^1(\O_{S_p}) \to \Ext^1(i_*L, i_*L) \to H^0(N) \to 0.
  \]

  By Lemma~\ref{lem: unions of minus one curves} the dimension of the space $H^0(N)$ is greater or equal to the intersection multiplicity $C_1 \cdot C_2$. However, by Lemma~\ref{lem: two components imply line bundle} we have the bound $\dim \Ext^1(i_*L, i_*L) \leq 1$. Since the curves $C_1$ and $C_2$ are not disjoint, we see that the only option is $C_1 \cdot C_2 = 1$. This may only happen if there is a single intersection point, and the intersection is transverse.

  Since there is only one intersection point and each irreducible component is a rational curve, we see that $\Pic(S_p) \iso \Z \oplus \Z$, and a line bundle is uniquely determined by the degrees of its restrictions to the components $C_1, C_2 \subset S_p$. By Lemma~\ref{lem: unions of minus one curves} we know that $H^0(N) \neq 0$. Since $S_p$ is a curve in a smooth surface, there are no obstructions to deformations of $S_p$. We can deform the divisor of $L$ together with the curve:
  
  \begin{figure*}[ht]
    \centering
    \def\svgwidth{0.3\columnwidth}
    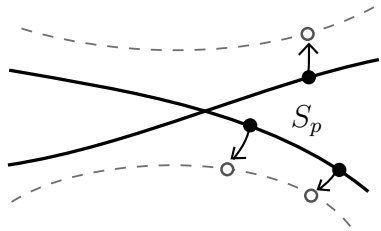
    \caption{Deforming the divisor in a family of curves.}
  \end{figure*}

  By Proposition~\ref{prop: admissible subcategories are open} admissible subcategories are closed under small deformations. Since $\mB$ contains $i_* L$, it also contains the pushforward of a line bundle from a small deformation of the curve $S_p$. This deformation is clearly not supported in $\supp(\mB)$, since that subset is by assumption a union of finitely many rigid curves. This is a contradiction, and therefore the situation where $S_p$ has two irreducible components is impossible.
\end{proof}

\begin{remark}
  After establishing that $C_1 \cdot C_2 = 1$ in Proposition~\ref{prop: never two components} above, one may use the classification of rank one torsion-free sheaves on nodal curves \cite[Prop.~10.1]{oda-seshadri} to prove that in fact the object $B$ itself is a pushforward of a line bundle from $S_p = C_1 \cup C_2$, but we do not need this.
\end{remark}

\statemaintheoremSections

\begin{proof}
  Let $E$ be a section of $\pi$ contained in the subset $C := \supp(\mB)$. Let $C^{\circ}$ be the open subset of $C$ defined as in Setting~\ref{set: sections}. Since by definition $C \setminus C^\circ$ is a finite set of points, the intersection $E \cap C^\circ$ is nonempty. By Lemma~\ref{lem: projection from support is nonzero} the projection of any point $p \in E \cap C^\circ$ is nonzero, and from Proposition~\ref{prop: never two components} we know that the support of the projection $B := \mB_R(\O_p)$ of the skyscraper sheaf can only be the curve $E$. Thus the point-support $S_p$ at the point $p$ is equal to $E$.

  Suppose $E, E^\prime$ are two distinct sections of $\pi$ contained in $\supp(\mB)$. By the argument above we can find points $p, p^\prime$ in $E \cap C^\circ$ and $E^\prime \cap C^\circ$, respectively, such that $S_p = E$ and $S_{p^\prime} = E^\prime$. If the curves $E$ and $E^\prime$ are not disjoint, then they intersect along a zero-dimensional subset. But the point-supports of distinct points $p, p^\prime$ cannot have zero-dimensional intersections by Lemma~\ref{lem: skyscrapers and intersections}. Thus $E \cap E^\prime = \emptyset$.

  We conclude that $\supp(\mB)$ is a disjoint union of several sections of the morphism $\pi$. The objects supported at disjoint curves are completely orthogonal to each other, so $\mB$ splits into an orthogonal sum of subcategories, where for each section $E \subset \supp(\mB)$ we have a nonzero subcategory $\mB_E$ supported on $E$. Since each section in a rational elliptic surface is a $(-1)$-curve by the adjunction formula, the possible options for subcategories $\mB_E$ are classified in Proposition~\ref{prop: new local classification on blow-ups}. This finishes the proof of the theorem.
\end{proof}

\subsection{Non-existence of phantoms in del Pezzo surfaces}
\label{ssec: no phantoms in del pezzos}

In the previous two subsections we have studied rational elliptic surfaces. In this subsection we apply the results of the preceding subsections to del Pezzo surfaces. The upshot is that there are no phantom subcategories in del Pezzo surfaces.

\begin{lemma}
  \label{lem: del pezzo and rational elliptic surfaces}
  Let $Y$ be a del Pezzo surface. Then there exists a blow-up $f\colon Y^\prime \to Y$ of several distinct points such that $Y^\prime$ is a rational elliptic surface whose fibration map $\pi\colon Y^\prime \to \P^1$ has no reducible fibers.
\end{lemma}

\begin{proof}
  By blowing up some general points we may assume that $Y$ is a del Pezzo surface of degree one. Then the anticanonical system has a single base point, and its blow-up is a rational elliptic surface $Y^\prime$. The fibers of the fibration map $Y^\prime \to \P^1$ are isomorphic to anticanonical divisors in $Y$. Suppose that some anticanonical divisor of $Y$ is reducible. Then we can write $(-K_Y)$ as a sum $C_1 + C_2$ with two effective curves $C_1, C_2$. Then
  \[
    (-K_Y) \cdot (-K_Y) = (-K_Y \cdot C_1) + (-K_Y \cdot C_2)
  \]
  Since $Y$ is a del Pezzo surface of degree one, the left hand side is equal to one. However, since the line bundle $(-K_Y)$ is ample, the intersection with each effective divisor is positive, so the right hand side is at least two. This is a contradiction. Therefore each anticanonical divisor in a del Pezzo surface of degree one is irreducible, and the map $\pi\colon Y^\prime \to \P^1$ has only irreducible fibers.
\end{proof}

\statemaintheoremDelPezzo

\begin{proof}
  Let $Y$ be a del Pezzo surface, and assume that $\mB \subset \Dbcoh(Y)$ is a phantom subcategory. By Lemma~\ref{lem: del pezzo and rational elliptic surfaces} there exists a blow-up $f\colon Y^\prime \to Y$ of several distinct points which is a rational elliptic surface with no reducible fibers. Consider the category $f^*\mB \subset \Dbcoh(Y^\prime)$ consisting of the pullbacks of objects in $\mB$. By Orlov's theorem on the semiorthogonal decomposition of the blow-up~\cite{orlov93} the category $f^*\mB$ is equivalent to $\mB$, and it is an admissible subcategory of~$\Dbcoh(Y^\prime)$. The pullback $f^*$ induces an isomorphism of Grothendieck groups $K_0(\mB) \to K_0(f^*\mB)$, hence $f^*\mB$ is a phantom subcategory of $\Dbcoh(Y^\prime)$.

  By Theorem~\ref{thm: rational elliptic surfaces} we see that the support of $f^*\mB$ is contained in the union of sections of the anticanonical fibration $\pi\colon Y^\prime \to \P^1$ and reducible fibers of this map. By construction $Y^\prime$ has no reducible fibers, so $f^*\mB$ is supported on the union of sections. Admissible subcategories like that are completely classified in Theorem~\ref{thm: admissible subcategories on sections}, and in particular there are no phantom subcategories of this kind. This is contradiction. Therefore there are no phantoms in the derived category of $Y$.
\end{proof}

It would be interesting to improve this result to the full classification of admissible subcategories, like we have managed to do for the projective plane in Section~\ref{sec: projective plane}. Note that exceptional objects and exceptional collections on del Pezzo surfaces have been studied in the paper \cite{kuleshov-orlov}, but the classification there is not as explicit as the one for the projective plane in \cite{gorodentsev-rudakov}. Another obstacle is that, unlike the case of the projective plane, a semiorthogonal decomposition on a del Pezzo surface is not uniquely determined by the restriction of a projection of a skyscraper sheaf at a single point to a single anticanonical divisor.


\printbibliography[title={References}]

\end{document}